\documentclass[11pt,a4paper,reqno]{amsart}
\usepackage[T1]{fontenc}
\usepackage{lmodern,ifpdf}
\ifpdf\usepackage{microtype}\fi
\usepackage[margin=26mm]{geometry}
\usepackage{amsmath,amssymb,amsthm,mathtools,mathrsfs}
\usepackage{booktabs,array,enumitem,xcolor,url,indentfirst,cite}
\usepackage{hyperref}
\hypersetup{hidelinks,pdftitle={Sharp parameter gaps for quantized Mabuchi energies and cscK metrics}}
\makeatletter
\renewcommand\subsubsection{\@startsection{subsubsection}{3}\z@
  {.7\baselineskip\@plus.2\baselineskip}{-.5em}{\normalfont\bfseries}}
\renewcommand\paragraph{\@startsection{paragraph}{4}\z@
  {.7\baselineskip\@plus.2\baselineskip}{-.5em}{\normalfont\bfseries}}
\makeatother
\numberwithin{equation}{section}
\newtheorem{theorem}{Theorem}[section]
\newtheorem{lemma}[theorem]{Lemma}
\newtheorem{proposition}[theorem]{Proposition}
\newtheorem{corollary}[theorem]{Corollary}
\theoremstyle{definition}

\theoremstyle{remark}
\newtheorem{remark}[theorem]{Remark}
\DeclareMathOperator{\AM}{AM}
\DeclareMathOperator{\Ent}{Ent}
\DeclareMathOperator{\Ric}{Ric}
\DeclareMathOperator{\Aut}{Aut}
\DeclareMathOperator{\MA}{MA}
\DeclareMathOperator{\Fut}{Fut}
\DeclareMathOperator{\Lie}{Lie}
\newcommand{\cM}{\mathcal M}
\newcommand{\cJ}{\mathcal J}
\newcommand{\cH}{\mathcal H}
\newcommand{\NA}{\mathrm{NA}}
\newcommand{\ENA}{E^{\NA}}
\newcommand{\rad}{\mathrm{rad}}
\newcommand{\R}{\mathbb R}
\newcommand{\Q}{\mathbb Q}
\providecommand{\C}{}
\renewcommand{\C}{\mathbb C}
\providecommand{\G}{}
\renewcommand{\G}{\mathbb G}
\newcommand{\HH}{\mathbb H}
\newcommand{\SSS}{\mathbb S}
\newcommand{\ZZ}{\mathbb Z}
\newcommand{\JJ}{\mathscr J}

\newcommand{\Tcal}{\mathcal T}
\newcommand{\RR}{\mathscr R}

\theoremstyle{plain}
\newtheorem{maintheorem}{Theorem}

\DeclareMathOperator{\tr}{tr}

\allowdisplaybreaks[1]
\title[Sharp parameter gaps and \protect\lowercase{csc}K metrics]{Sharp parameter gaps for quantized Mabuchi energies\protect\\
and \protect\lowercase{csc}K metrics}
\author{Kefeng Liu}
\address{Mathematical Sciences Research Center, Chongqing University of
Technology, Chongqing 400054, China}
\email{kefengliu@cqut.edu.cn}
\author{Liangdi Zhang}
\address{Mathematical Science Research Center, Chongqing University of
Technology, Chongqing 400054, China}
\email{ldzhang91@163.com}
\date{October 12, 2026}
\makeatletter
\@ifundefined{subjclassname@2020}{%
  \@namedef{subjclassname@2020}{\textup{2020} Mathematics Subject Classification}}{}
\makeatother
\begin{document}
\begin{abstract}
For polarized manifolds, we show that semistability for the quantized
Mabuchi energy at one fixed finite parameter is equivalent to uniform
K-stability and to unique existence of constant scalar curvature
K\"ahler (cscK) metrics. We obtain sharp comparisons between quantized
Mabuchi slopes at different parameters and with the ordinary Mabuchi
slope. In the presence of automorphisms, we give an equivariant criterion
using torus twists under the vanishing of the Futaki character.
The estimates also hold for finite-energy geodesic rays in compact
K\"ahler classes. On a fixed ruled surface with discrete automorphism
group, we construct cscK polarizations with unbounded critical parameters.
Their ratios to the $\alpha$-invariant are also unbounded.
\end{abstract}

\maketitle
\enlargethispage{6pt}
\setcounter{tocdepth}{1}
\begingroup\fontsize{9.5}{11}\selectfont
\tableofcontents
\endgroup

\section{Introduction}

The Yau--Tian--Donaldson conjecture relates the existence of constant
scalar curvature K\"ahler (cscK) metrics to stability of one-parameter
degenerations. It grew out of the extremal-metric programme of
Calabi~\cite{calabi82,calabi85}, the stability condition of
Tian~\cite{tian97}, and the projective approach of
Donaldson~\cite{donaldson01,donaldson02}. We consider uniform
K-stability, which
requires a lower bound by a positive multiple of a non-Archimedean
norm. In the presence of automorphisms, we use the reduced norm.

A cscK metric is a critical point of the energy introduced by
Mabuchi~\cite{mabuchi86}. This functional is geodesically
convex~\cite{bb,bdl}. Its finite-energy geometry~\cite{bdl,darvas,dl,gz}, together with
properness and regularity results~\cite{dr,bdlregularity,CC1,CC2},
relates cscK existence to coercivity. Test configurations determine
geodesic rays~\cite{psrays,rosswitt}, and the slopes of the energy
functionals along these rays are expressed by non-Archimedean
invariants~\cite{BHJ,bhj19,bhjerratum,BJ,bjmetric}.
Li~\cite{Li} proves a criterion for the existence of cscK metrics in
terms of model metrics. Boucksom--Jonsson~\cite[Theorem~A]{BJweighted}
and Darvas--Zhang~\cite[Theorem~1.1]{DZ} prove versions of the
Yau--Tian--Donaldson correspondence in terms of non-Archimedean
potentials and quantized Mabuchi energies, respectively.
Mesquita-Piccione~\cite{MP} extends the non-Archimedean framework to
compact K\"ahler manifolds and proves a sufficient condition for cscK
existence. For polarized manifolds, Trusiani~\cite{Tru} proves
equivalence between ordinary reduced uniform K-stability and the
existence of cscK metrics.

We prove that if the quantized Mabuchi slopes are nonnegative at a
fixed finite parameter, then at every larger parameter they are
bounded below by a positive multiple of the radial $J$-slopes,
independent of the ray. The constants are
sharp and depend only on the dimension and the $\alpha$-invariant.
The estimate holds for every finite-energy geodesic ray from a smooth
potential in a compact K\"ahler class. For polarized manifolds, the
parameter must be independent of the test configuration. This gives
a cscK criterion from finite-parameter semistability. We also treat
automorphisms and construct a fixed manifold on which the required
parameters are unbounded as the polarization varies.

Let $(X,\omega_0)$ be a
connected compact K\"ahler manifold of
dimension $n\ge1$, put $\Omega=[\omega_0]$, and write
$A=\alpha(X,\Omega)$ for its $\alpha$-invariant.
For a finite-energy geodesic ray $u$, set
$s_\beta(u)=(\cM^\beta)^{\rad}(u)$.
Darvas--Zhang~\cite[Section~3]{DZ} introduce the quantized entropy
$\Ent^\beta$ and the associated Mabuchi energy
$\cM^\beta=\Ent^\beta-\cJ_{\Ric(\omega_0)}$.
We recall their definitions and our normalizations in
Sections~\ref{sec:quantized-energies} and~\ref{sec:finite-energy-geometry}.
Define
\begin{equation}\label{intro:gap-function}
 G_A(\beta)=\int_\beta^\infty\min\{1,(A/b)^{n+1}\}\,db
 =\begin{cases}
 (n+1)A/n-\beta,&0<\beta\le A,\\
 A^{n+1}/(n\beta^n),&\beta\ge A,
 \end{cases}
\end{equation}
We write $\Gamma_A(\beta,B)=G_A(\beta)-G_A(B)$.

\begin{maintheorem}[Sharp parameter bounds]\label{main:sharp-gap}
On a connected compact K\"ahler manifold of dimension $n\ge1$,
every finite-energy geodesic ray $u$ from a smooth potential
and every $0<\beta<B$ satisfy
\begin{align*}
 s_B(u)-s_\beta(u)
   &\ge\bigl(G_A(\beta)-G_A(B)\bigr)J^{\rad}(u),\\
 \cM^{\rad}(u)-s_\beta(u)
   &\ge G_A(\beta)J^{\rad}(u).
\end{align*}
Here $\cM^{\rad}$ and $J^{\rad}$ are the radial slopes of the Mabuchi
functional and Aubin's \(J\)-functional, respectively. The coefficients
$G_A(\beta)$ and $G_A(\beta)-G_A(B)$ are optimal among bounds depending
only on $n$ and $A$.
\end{maintheorem}

Equality is attained on products of curves with discrete automorphism
groups, whose classes contain unique cscK metrics; see
Section~\ref{sec:sharp-gap}.
The coefficients in Theorem~\ref{main:sharp-gap} are independent of
the ray and, in the polarized case, of the test configuration and its
exponent. On every fixed polarized manifold, the lower bound for the
difference from the ordinary Mabuchi slope has optimal decay order
$\beta^{-n}$; see Proposition~\ref{prop:dimension-obstruction}.

Fix a smooth potential $u$ and write $E(b)=\Ent^b(u)$.
Lemma~\ref{lem:parameter-second-derivative} gives
\[
 0\le-bE''(b)\le(n+1)E'(b),
\]
so $b^{n+1}E'(b)$ is nondecreasing.
Uniform $\alpha$-integrability bounds $E'(b)$ from below at a small
parameter. Integrating this bound gives the pointwise estimates of
Theorem~\ref{thm:analytic-gap} for smooth potentials. We extend them
to finite-energy potentials and take radial slopes to prove
Theorem~\ref{thm:gap}. Section~\ref{sec:sharp-gap} proves optimality,
including on products of curves with discrete automorphism groups.

Papazachariou~\cite[Corollary~2.14 and Theorems~3.10
and~4.8]{Pap} proved that finite-parameter semistability persists as the
parameter increases and obtained quantitative comparisons of stability
thresholds. We prove lower bounds for both the finite-parameter
increment and the difference from the ordinary Mabuchi slope, with
optimal coefficients.

For a test configuration $\Tcal$, let $u^\Tcal$ denote its
associated geodesic ray, and write $\mathrm{DF}(\Tcal)$ and
$J^{\NA}(\Tcal)$ for its Donaldson--Futaki invariant and
non-Archimedean $J$-invariant, respectively. In the polarized
setting, $\phi_\Tcal$ denotes the associated non-Archimedean
metric, so that we also write $J^{\NA}(\phi_\Tcal)=J^{\NA}(\Tcal)$.
In each statement below, the constants are independent of the test
configuration in the specified family.

For a polarized pair $(X,L)$, write
$m_\beta(\phi_\Tcal)=s_\beta(u^\Tcal)$ and
$A=\alpha(X,c_1(L))=\alpha(X,L)$. Throughout the polarized statements below,
all positive integer exponents are allowed, and all invariants
are normalized to $L$.

\begin{corollary}[The polarized correspondence]
\label{cor:polarized-correspondence}
Let $(X,L)$ be a connected smooth polarized manifold. The
following are equivalent:
\begin{enumerate}[label=\textup{(\roman*)}]
\item $c_1(L)$ contains a unique cscK metric;
\item for some finite $\beta>0$, $m_\beta(\phi_\Tcal)\ge0$
for every normal ample algebraic test configuration $\Tcal$;
\item for some finite $B>0$ and $\gamma>0$,
$m_B(\phi_\Tcal)\ge\gamma J^{\NA}(\phi_\Tcal)$ for every such
$\Tcal$;
\item $(X,L)$ is uniformly K-stable: for some $\delta>0$,
$\mathrm{DF}(\Tcal)\ge\delta J^{\NA}(\phi_\Tcal)$ for every
such $\Tcal$.
\end{enumerate}
\end{corollary}

Darvas--Zhang~\cite[Equation~(9) and Section~5]{DZ} introduce
uniform $K^\beta$-stability. Papazachariou~\cite[Definition~2.12]{Pap}
defines $K^\beta$-semistability for rational $\beta>1$ by
nonnegativity on every normal ample test configuration.
Here we use the radial definition for every real $\beta>0$.

In Corollary~\ref{cor:polarized-correspondence}, uniqueness means
uniqueness of the K\"ahler form in its class. Condition \textup{(ii)}
at $\beta$ implies
condition \textup{(iii)} at every $B>\beta$, with
$\gamma=\Gamma_A(\beta,B)$.
Proofs and the ordinary-stability conventions are given in
Sections~\ref{sec:criteria} and~\ref{ordinary:section}.
The polarized correspondence uses Darvas--Zhang~\cite[Theorem~1.1]{DZ}
and Trusiani~\cite[Corollary~A]{Tru}. The corresponding statements for
arbitrary compact K\"ahler classes are given in
Section~\ref{ordinary:kahler}.

Product configurations are generated by one-parameter subgroups
of the automorphism group. Papazachariou~\cite[Definition~4.2,
Theorem~4.5, and Remark~4.6]{Pap} takes the supremum of the quantized
Mabuchi slope over rational product twists by a maximal torus.
Assuming that the polarized automorphism group is reductive and
that the Futaki character vanishes, he identifies uniform reduced
$K^\beta$-polystability at some rational $\beta>1$ with cscK existence.
Theorem~\ref{main:levi} requires only nonnegativity at a fixed finite
parameter, independent of the test configuration. The equivariant
setup follows Li~\cite{LiG,Li}.
For the Futaki character,
reductivity, and properness modulo automorphisms, see
\cite{futaki,fujiki,matsushima,dr,dervantwist,dervanrelative,His,Li}.

Let $\G=\Aut_0(X,L)$ be the identity component of the group of
automorphisms $g$ of $X$ with $g^*L\simeq L$. This is a linear
algebraic group. We write $\Aut_0(X)$ for the identity component
of the full holomorphic automorphism group. Choose a maximal
connected reductive subgroup $\HH\subset\G$, a maximal torus
$\SSS\subset\HH$, and let $\ZZ_\HH=Z(\HH)^\circ$ be the connected
center of $\HH$. A test configuration is $\HH$-equivariant if the
$\HH$-action extends to its polarized total space over $\C$ and commutes
with the $\C^*$-action covering multiplication on the base.
For such a configuration, set
\[
 \RR_{\beta,\HH}^{\rm max}(\Tcal)
 =\sup_{\xi\in N_\Q(\SSS)}
       (\cM^\beta)^{\rad}(u^{\Tcal_\xi}),
\]
For an algebraic torus $T$, we write
$N_\Q(T)=\operatorname{Hom}(\C^*,T)\otimes_\mathbf Z\Q$.
For integral $\xi$, $\Tcal_\xi$ has the same polarized family as
$\Tcal$, with the $\C^*$-action composed with $\xi$.
For rational $\xi$, this is defined after normalized base change.
Write $\Fut_{X,L}:\Lie\G\to\mathbb C$ for the Futaki character.
Since $\ZZ_\HH\subset\HH\subset\G$, its restriction to
$\Lie\ZZ_\HH$ is well defined.

\begin{maintheorem}[Equivariant cscK criterion]
\label{main:levi}
For the maximal connected reductive subgroup $\HH\subset\G$ fixed
above, the class
$c_1(L)$ contains a cscK metric if and only if the Futaki
character \(\Fut_{X,L}\) vanishes on $\Lie\ZZ_\HH$ and there is a
fixed finite $\beta>0$ such that
\[
 \RR_{\beta,\HH}^{\rm max}(\Tcal)\ge0
\]
for every normal ample $\HH$-equivariant test configuration $\Tcal$.
\end{maintheorem}

The sufficiency proof uses Theorem~\ref{main:sharp-gap} and the
reduced norm. It holds for every connected reductive subgroup by
Theorem~\ref{thm:groups}; the converse is proved for a maximal connected reductive subgroup.
We do not assume that the full group $\G$ is reductive. Reductivity
follows from cscK existence. When $\G$ is reductive,
Trusiani~\cite[Theorem~A]{Tru} proves that cscK existence is equivalent
to ordinary reduced uniform K-stability. The model-metric and
geodesic-ray proofs are given in Sections~\ref{sec:groups} and
\ref{sec:center-direct}.

The finite parameter may depend on the class. Theorem~\ref{main:ruled}
shows that no single parameter works for all polarizations of a fixed
surface, even when its automorphism group is discrete. Put
\[
 \beta_{\rm crit}(X,L)
 =\inf\{\beta>0:m_\beta(\phi_\Tcal)\ge0
                   \text{ for every normal ample }\Tcal\}.
\]
Corollary~\ref{cor:critical-ordinary-margin} also gives a lower bound
for $\beta_{\rm crit}/A$ in terms of the ordinary stability constant.
The examples use projectively flat ruled surfaces
\cite{acgt08,acgt11,NS} and deformations to the normal cone
\cite{rossthomas}.

\begin{maintheorem}[Unbounded critical parameters on a fixed surface]
\label{main:ruled}
There exist a fixed smooth ruled surface $X$ with $\Aut_0(X)=\{1\}$
and ample line bundles $L_d$, for integers $d\ge4$, such that each $c_1(L_d)$
contains a cscK metric and
\[
 \alpha(X,L_d)=\frac1d,\qquad
 \liminf_{d\to\infty}\frac{\beta_{\rm crit}(X,L_d)}d\ge\frac34.
\]
\end{maintheorem}

Thus both $\beta_{\rm crit}(X,L_d)$ and
$\beta_{\rm crit}(X,L_d)/\alpha(X,L_d)$ are unbounded.
The line bundles $L_d$ are integral, have fiber degree one, and are
pairwise nonproportional. Since $\beta_{\rm crit}/\alpha$ is unchanged
by positive tensor powers, rescaling does not remove this
unboundedness. The critical parameter for smooth-divisor
deformations, computed in Section~\ref{sec:bounds}, is a lower
bound for the full critical parameter.

A small negative lower bound at a finite parameter also suffices for
cscK existence. For the groups $\HH$ and $\ZZ_\HH$ defined above,
write
\[
 N_\HH(\Tcal)=\inf_{\zeta\in N_\Q(\ZZ_\HH)}
           J^{\NA}(\phi_{\Tcal_\zeta}).
\]

\begin{corollary}[Quantitative forms of the correspondences]
\label{intro:quantitative-correspondences}
Let $(X,L)$ be a connected smooth polarized manifold, put
$A=\alpha(X,L)$, and fix $0\le\vartheta<1$. All test configurations
below are normal and ample, of arbitrary positive exponent, with
invariants normalized to $L$.
\begin{enumerate}[label=\textup{(\roman*)}]
\item The class $c_1(L)$ contains a unique cscK metric if and only if
there is a finite $\beta>0$, independent of the test configuration,
such that
\begin{equation}\label{intro:robust-unreduced}
 m_\beta(\phi_\Tcal)
 \ge-\vartheta G_A(\beta)J^{\NA}(\phi_\Tcal)
 \qquad\text{for every }\Tcal.
\end{equation}
Then, for every $B>\beta$,
\begin{equation}\label{intro:robust-gain}
 m_B(\phi_\Tcal)
 \ge\bigl((1-\vartheta)G_A(\beta)-G_A(B)\bigr)
       J^{\NA}(\phi_\Tcal).
\end{equation}
\item Let $\HH\subset\Aut_0(X,L)$ be a maximal connected reductive
subgroup, with $\SSS$ and $\ZZ_\HH$ as above. Assume that the Futaki
character vanishes on $\Lie\ZZ_\HH$. Then $c_1(L)$ contains a cscK
metric if and only if there is a finite $\beta>0$, independent of
the test configuration, such that
\begin{equation}\label{intro:robust-reduced}
 \RR_{\beta,\HH}^{\rm max}(\Tcal)
 \ge-\vartheta G_A(\beta)N_\HH(\Tcal)
 \qquad\text{for every normal ample $\HH$-equivariant }\Tcal.
\end{equation}
Then, for every $B>\beta$,
\[
 \RR_{B,\HH}^{\rm max}(\Tcal)
 \ge\bigl((1-\vartheta)G_A(\beta)-G_A(B)\bigr)N_\HH(\Tcal).
\]
\end{enumerate}
\end{corollary}

Corollary~\ref{cor:existence} proves the unreduced assertion,
and Section~\ref{sec:center-direct} proves the reduced assertion.
The minus signs in \eqref{intro:robust-unreduced} and
\eqref{intro:robust-reduced} allow the quantized slopes to be negative.
The parameter gain makes the lower bounds positive when
$G_A(B)<(1-\vartheta)G_A(\beta)$. For $\beta\ge A$, one may
take $B>\beta(1-\vartheta)^{-1/n}$. The anticanonical endpoint
$\vartheta=1$ is treated in Theorem~\ref{thm:fano-tolerance-endpoint}.
Remark~\ref{rem:log-endpoint} shows why the analogous endpoint
statement fails for log pairs.

For an individual configuration, the sections of maximal weight
generate a base ideal. Proposition~\ref{prop:valuative-parameter-refinement}
replaces $A$ in the parameter bounds by a possibly larger number
determined by the log canonical threshold of this ideal.
Section~\ref{ordinary:section} relates the finite-parameter bounds to
ordinary uniform K-stability, first for polarized manifolds and then
for arbitrary K\"ahler classes.

\section{Preliminaries}\label{sec:prelim}

\subsection{K\"ahler potentials and quantized energies}\label{sec:quantized-energies}

Throughout the paper, $d^c=\frac{i}{4\pi}(\bar\partial-\partial)$,
so $dd^c=\frac{i}{2\pi}\partial\bar\partial$.
Let $X$ be a connected compact K\"ahler manifold of dimension $n\geq1$, fix a
K\"ahler form $\omega_0$, and set
\[
\begin{gathered}
 V=\int_X\omega_0^n,\qquad d\nu_0=V^{-1}\omega_0^n,\qquad
 d\mu_u=V^{-1}\omega_u^n,\\
 \cH=\{u\in C^\infty(X,\R):\omega_u:=\omega_0+dd^cu>0\}.
\end{gathered}
\]
In the polarized setting, $\omega_0\in c_1(L)$ and $V=L^n$.
The volume-normalized Aubin--Yau energy \cite{bdl,DZ} is
\begin{equation}\label{eq:AM}
 \AM(u)=\frac1{(n+1)V}\sum_{j=0}^n
       \int_Xu\,\omega_u^j\wedge\omega_0^{n-j}.
\end{equation}
We write
\[
 I(u)=\int_Xu(d\nu_0-d\mu_u),\qquad
 J(u)=\int_Xu\,d\nu_0-\AM(u),\qquad
 J_0(u)=\AM(u)-\int_Xu\,d\mu_u.
\]
Here $J$ is Aubin's $J$-functional.
These energies satisfy
\[
 I=J+J_0,\qquad \frac{I}{n+1}\le J,J_0\le\frac{nI}{n+1}.
\]

For $v,w\in\cH$, put $h=w-v$ and $v_s=v+sh$, $0\le s\le1$.
The first variation formula and integration by parts give
\begin{align*}
 \frac{d}{ds}\AM(v_s)
 &=\int_X h\,d\mu_{v_s},\\
 \frac{d^2}{ds^2}\AM(v_s)
 &=\frac{n}{V}\int_X h\,dd^ch\wedge\omega_{v_s}^{n-1}\\
 &=-\frac{n}{2\pi V}
   \int_X i\partial h\wedge\bar\partial h
          \wedge\omega_{v_s}^{n-1}\le0.
\end{align*}
Thus $\AM$ is concave along affine segments, and hence
\[
 \AM(w)\le
 \AM(0)+\left.\frac{d}{ds}\AM(sw)\right|_{s=0}
 =\int_X w\,d\nu_0,
\]
so $J(w)\ge0$.

For a smooth closed real $(1,1)$-form $\theta$, put
\begin{equation}\label{eq:AMtheta}
 \AM_\theta(u)=\frac1{nV}\sum_{j=0}^{n-1}
 \int_Xu\,\theta\wedge\omega_u^j\wedge\omega_0^{n-1-j},
 \qquad \bar\theta=\frac nV\int_X\theta\wedge\omega_0^{n-1},
\end{equation}
and
\begin{equation}\label{eq:Jtheta}
 \cJ_\theta=n\AM_\theta-\bar\theta\AM.
\end{equation}
In particular, $\cJ_{\omega_0}=J_0$.
With $\Ric(\omega_0)=-dd^c\log\omega_0^n$, the Mabuchi functional is
\[
 \cM(u)=\Ent(u)-\cJ_{\Ric(\omega_0)}(u),\qquad
 \Ent(u)=\int_X\log\frac{\omega_u^n}{\omega_0^n}\,d\mu_u.
\]
Thus $\cM(0)=0$ and
\begin{equation}\label{eq:M-variation}
 d\cM_u(v)=-\int_Xv(S_u-\bar S)\,d\mu_u,
 \qquad
 \bar S=\frac nV\int_X\Ric(\omega_0)\wedge\omega_0^{n-1},
\end{equation}
where $S_u=\tr_{\omega_u}\Ric(\omega_u)$.
For these energy identities, see
\cite{chenenergy,mabuchi86}.

For $\beta>0$, the quantized energies of
Darvas--Zhang~\cite[Section~3]{DZ} are defined by
\begin{equation}\label{eq:analytic-quantization}
\begin{split}
 \Ent^\beta(u)&=\sup_{w\in\cH}
 \left\{-\log\int_Xe^{\beta(w-u)}d\nu_0
       +\beta\bigl(\AM(w)-\AM(u)\bigr)\right\},\\
 \cM^\beta(u)&=\Ent^\beta(u)-\cJ_{\Ric(\omega_0)}(u).
\end{split}
\end{equation}
For $u\in\cH$, let $u^\beta$ be the unique smooth solution of
\begin{equation}\label{eq:quantized-optimizer}
 (\omega_0+dd^cu^\beta)^n
       =e^{\beta(u^\beta-u)}\omega_0^n.
\end{equation}
The equation fixes the additive constant of $u^\beta$.
For smooth $u$, existence and uniqueness follow from the
Monge--Amp\`ere theory of Aubin~\cite{aubin} and Yau~\cite{yau}.
Darvas--Zhang~\cite[Section~3]{DZ} extend the equation to
$u\in\mathcal E^1$, with a unique continuous solution.
They prove the identity
$\Ent^\beta(u)=\beta(\AM(u^\beta)-\AM(u))$ and show that $u^\beta$
attains the supremum over $w\in\mathcal E^1$
\cite[Proposition~3.6 and the paragraph following its proof]{DZ}.
The supremum over $w\in\cH$ has the same value and admits a smooth
maximizing sequence. For smooth $u$, $u^\beta$ itself belongs to
$\cH$ and attains the supremum in \eqref{eq:analytic-quantization}.

\subsection{Finite-energy geometry}\label{sec:finite-energy-geometry}

The path norm $\|v\|_{1,u}=\int_X|v|\,d\mu_u$ induces the
$d_1$ metric on $\cH$. Its completion is the space $\mathcal E^1$
of finite-energy potentials~\cite[Theorem~2]{darvas}.
For $u,v\in\mathcal E^1$ with $u\le v$, the ordered $d_1$ formula
\cite{darvas} is
\begin{equation}\label{eq:ordered-d1}
 d_1(u,v)=\AM(v)-\AM(u).
\end{equation}
There is a constant $C>0$, depending only on the reference form,
such that
\begin{equation}\label{eq:sup-mean}
 0\le \sup_Xu-\int_Xu\,d\nu_0\le C,
 \qquad u\in\mathcal E^1.
\end{equation}

We work on the normalized slice
$\mathcal E^1_0=\{u\in\mathcal E^1:\AM(u)=0\}$.
The energy functionals extend to $\mathcal E^1$;
see, for example, \cite{darvas,bdl,DZ}. The metric estimates in
\cite{darvas,bdl} give a constant $C\ge1$, depending only on the
reference form, such that
\begin{equation}\label{eq:normalized-distance-J}
 C^{-1}J(u)-C\le d_1(0,u)\le CJ(u)+C,
 \qquad u\in\mathcal E^1_0.
\end{equation}
Darvas--Lu~\cite[Equation~(3)]{dl} define the chordal distance between
finite-energy geodesic rays from the same point by
\[
 \rho(u,v)=\lim_{t\to\infty}t^{-1}d_1(u_t,v_t).
\]
The limit exists by Busemann convexity of $d_1$~\cite[Proposition~5.1]{bdl}.

For a finite-energy geodesic ray $u=\{u_t\}$ from a smooth
reference, write
\[
 J^{\rad}(u)=\lim_{t\to\infty}\frac{J(u_t)}t,
 \qquad
 (\cM^\beta)^{\rad}(u)=
       \liminf_{t\to\infty}\frac{\cM^\beta(u_t)}t.
\]
The function $J(u_t)$ is convex: $u_t$ is pointwise convex in $t$
and $\AM(u_t)$ is affine. Subtracting the affine function
$\AM(u_t)$ leaves $J(u_t)$ unchanged and puts the ray in
$\mathcal E^1_0$. Since the ray has constant $d_1$-speed,
\eqref{eq:normalized-distance-J} gives $J(u_t)=O(1+t)$.
It follows that $J^{\rad}(u)$ exists and is finite.
Geodesic convexity of the Mabuchi functional gives
the extended limit
\[
 \cM^{\rad}(u)=\lim_{t\to\infty}\frac{\cM(u_t)}t
 \in\R\cup\{+\infty\}.
\]
For these definitions, see \cite{bdl,darvas} and
\cite[Section~2]{DZ}. The quantized slope is defined by a lower
limit. Adding a function of time that is constant on $X$ does not
change these slopes.

The group $G=\Aut_0(X)$ acts by pullback of metrics. The symbol
$\G$ denotes the polarized automorphism group defined in the Introduction.

There exist $\alpha_0,C_\alpha>0$, depending only on
$(X,\omega_0)$, such that every $u\in\operatorname{PSH}(X,\omega_0)$
with $\sup_Xu=0$ satisfies
\begin{equation}\label{eq:alpha-input}
 \int_Xe^{-\alpha_0u}d\nu_0\le C_\alpha.
\end{equation}
This uniform estimate follows from the positivity of the
$\alpha$-invariant \cite{tian}; it is used in this form in
\cite[proof of Proposition~6.9]{DZ}. For local integrability, see
\cite[Chapter~VIII, Section~8]{demailly}.

\subsection{Test configurations and associated rays}

A normal ample test configuration of exponent $r\in\mathbb Z_{>0}$ for $(X,L)$
is a flat projective morphism
$\pi:(\mathcal X,\mathcal L)\to\mathbf A^1=\operatorname{Spec}\C[t]$,
where $\mathcal X$ is normal, together with a $\C^*$-action lifting
multiplication on the base, a relatively ample linearized line bundle
$\mathcal L$, and a $\C^*$-equivariant identification
\[
 (\mathcal X,\mathcal L)|_{\mathbf A^1\setminus\{0\}}
 \simeq (X,L^{\otimes r})\times\C^*.
\]
For intersection invariants we use the normalized $\Q$-line bundle
$r^{-1}\overline{\mathcal L}$ on the standard compactification
$\overline{\mathcal X}\to\mathbf P^1$.

If $\rho:\mathcal X'\to\mathcal X$ is an equivariant birational
morphism over the base, the pullback model
$(\mathcal X',\rho^*\mathcal L)$ determines the same non-Archimedean
metric \cite{BHJ,bjmetric}, denoted by $\phi_\Tcal$.
The test configuration also determines a weak geodesic ray
$u^\Tcal$ emanating from $0$ \cite{psrays,rosswitt}.

We parametrize rays by $t=-\log|\tau|^2$.
When a torus acts and $\xi$ is a rational cocharacter, choose
$N\ge1$ such that
$N\xi$ is integral. After the base change $\tau\mapsto\tau^N$,
twist by $N\xi$ and divide homogeneous slopes and norms by $N$.
The central fiber need not be reduced; the correction term relating
the Donaldson--Futaki and non-Archimedean Mabuchi invariants is
recorded in
Section~\ref{ordinary:section}. For a general K\"ahler class, we use normal relatively K\"ahler test configurations in the sense of \cite[Definition~5.4]{DZ}.

\subsection{Non-Archimedean energies}
\label{sec:2.4}
We follow the non-Archimedean pluripotential theory of
Boucksom--Jonsson~\cite{BJ,bjmetric,BJdivisorial}.
In the polarized setting, write $X^{\NA}$ for the Berkovich
analytification of $X$ over
$\C$ endowed with the trivial absolute value, and $\ENA$ for
its energy functional.
For a model metric represented by the normalized polarization
$\mathcal A=r^{-1}\overline{\mathcal L}$, the Monge--Amp\`ere measure is
\[
 \MA(\phi)=\frac1V\sum_E b_E(\mathcal A^n\cdot E)\delta_{v_E},
 \qquad
 b_E=\operatorname{ord}_E(t),\quad
 v_E=b_E^{-1}\operatorname{ord}_E|_{\C(X)}.
\]
Here $E$ ranges over the prime components of the central fiber,
$\operatorname{ord}_E$ is the order of vanishing along $E$, and
$b_E$ is the multiplicity of $E$ in that fiber \cite{BHJ,BJ}. For semipositive model metrics we use the normalization
\[
\sup\phi=\phi(v_{\rm triv})=0,
\]
where $v_{\rm triv}$ denotes the trivial valuation.
With this normalization,
\[
 J^{\NA}(\phi)=\sup\phi-\ENA(\phi)=-\ENA(\phi).
\]

Write
\[
 H^{\NA}(\phi)=\int_{X^{\NA}}A_X\,\MA(\phi),
\]
where $A_X$ is the log-discrepancy function, and let
$\mathcal R^{\NA}$ denote the non-Archimedean counterpart of the
Ricci energy $\cJ_{\Ric(\omega_0)}$. Then
\[
 \cM^{\NA}=H^{\NA}-\mathcal R^{\NA},\qquad
 m_\beta=e_\beta-\mathcal R^{\NA}.
\]
For $\beta>1$, the non-Archimedean quantized entropy is
\begin{equation}\label{eq:quantized-NA-entropy}
 e_\beta(\phi)=\beta\bigl(\ENA(P(\phi+\beta^{-1}A_X))
                              -\ENA(\phi)\bigr).
\end{equation}
Here $P(f)$ is the upper-semicontinuous regularization of the supremum
of all semipositive non-Archimedean potentials bounded above by $f$.
For a test configuration with associated ray $u^\Tcal$, this agrees with
\begin{equation}\label{eq:radial-quantized-entropy}
 e_\beta(\phi_\Tcal)=\liminf_{t\to\infty}
                   \frac{\Ent^\beta(u_t^\Tcal)}{t}.
\end{equation}
The algebraic formula of Darvas--Zhang~\cite[Theorem~9.5]{DZ} is
stated for rational $\beta>1$; Section~\ref{sec:entropy-envelope}
explains the extension to real $\beta>1$. For $0<\beta\le1$, we
define $e_\beta$ by the right-hand side of
\eqref{eq:radial-quantized-entropy}. The slope formulas \cite{DZ,bhj19} give
\begin{equation}\label{eq:slope-bridge}
\begin{aligned}
 \lim_{t\to\infty}\frac{\AM(u_t^\Tcal)}{t}
 &=\ENA(\phi_\Tcal)=-J^{\NA}(\phi_\Tcal),\\
 \lim_{t\to\infty}\frac{J(u_t^\Tcal)}{t}
 &=J^{\rad}(u^\Tcal)=J^{\NA}(\phi_\Tcal),\\
 (\cM^\beta)^{\rad}(u^\Tcal)
 &=m_\beta(\phi_\Tcal).
\end{aligned}
\end{equation}

There is a constant $C>0$, depending only on $\omega_0$, such that
\begin{equation}\label{eq:ricci-energy-bound}
 |\cJ_{\Ric(\omega_0)}(u)|\le C\bigl(J(u)+1\bigr),
 \qquad u\in\mathcal E^1,
\end{equation}
and, for every normal ample configuration metric $\phi$,
\begin{equation}\label{eq:ricci-slope-bound}
 |\mathcal R^{\NA}(\phi)|\le C J^{\NA}(\phi).
\end{equation}
Indeed, choose $C_0\ge0$ with
$-C_0\omega_0\le\Ric(\omega_0)\le C_0\omega_0$.
For smooth $w$ with $\sup_Xw=0$, the mixed-energy formulas
\eqref{eq:AMtheta} and \eqref{eq:Jtheta} give
\[
 |\cJ_{\Ric(\omega_0)}(w)|
 \le\bigl((n+1)C_0+|\bar S|\bigr)(-\AM(w)).
\]
Since $-\AM(w)=J(w)-\int_Xw\,d\nu_0$, \eqref{eq:sup-mean}
proves \eqref{eq:ricci-energy-bound} for smooth normalized potentials.
Continuity of the energies under $d_1$ convergence extends the bound
to $\mathcal E^1$. Since $\cJ_{\Ric(\omega_0)}$ and $J$ are unchanged
by adding constants, applying the bound to $w=u-\sup_Xu$ removes
the normalization assumption. Finally, divide \eqref{eq:ricci-energy-bound}
by $t$ along the associated ray. The Ricci-energy slope formula
\cite{bhj19} and \eqref{eq:slope-bridge} give
\eqref{eq:ricci-slope-bound}.

\subsection{The \texorpdfstring{$\alpha$}{alpha}-invariant}\label{sec:alpha-invariant}

For a compact K\"ahler class $\Omega=[\omega_0]$, put
\[
 A=\alpha(X,\Omega):=\sup\left\{a>0:
 \sup_{u\in\cH}\int_Xe^{-a(u-\sup_Xu)}d\nu_0<\infty\right\}.
\]
In the polarized case $\Omega=c_1(L)$, we also write
$A=\alpha(X,L)$. By the analytic--algebraic identification of
Demailly~\cite[Theorem~A.3]{DemaillyAlpha},
\[
 A=\inf_{\substack{\Delta\ge0\\\Delta\sim_\Q L}}
                  \operatorname{lct}(X;\Delta)\in(0,\infty).
\]
We use uniform integrability for $a<A$. The limit $a\uparrow A$
is taken only after deriving the numerical or radial inequalities;
see the proof of Theorem~\ref{thm:gap}.

\subsection{The envelope formula for quantized entropy}
\label{sec:entropy-envelope}

Let $\phi$ be a semipositive model potential. Choose a smooth SNC
test-configuration model $\mathcal Y$ dominating
$X\times\mathbb P^1$ on which $\phi$ is represented, and write
$r_{\mathcal Y}:X^{\NA}\to\Delta_{\mathcal Y}$ for the
retraction to its dual complex.
After replacing $\mathcal Y$ by a higher model, we may assume
$\phi=\phi\circ r_{\mathcal Y}$ \cite[Appendix~A, Lemma~A.2]{BJ}.
Since $X$ is smooth in characteristic zero, the envelope property
holds by Boucksom--Jonsson~\cite[Theorem~5.20]{BJ}.
With the envelope operator $P$ from Section~\ref{sec:2.4}, set
\begin{equation}\label{eq:envelope}
 a_{\mathcal Y}=A_X\circ r_{\mathcal Y},\qquad
 \psi_\beta=P(\phi+a_{\mathcal Y}/\beta),\qquad \beta>1.
\end{equation}
The function $a_{\mathcal Y}$ is nonnegative, bounded and continuous, and
$a_{\mathcal Y}(v_{\rm triv})=0$.
For relatively ample configuration metrics,
Darvas--Zhang~\cite[Lemma~9.1, Remark~9.2, Proposition~9.4 and
Theorem~9.5]{DZ} give
\begin{equation}\label{eq:entropy-slope}
 e_\beta(\phi)=\beta\bigl(\ENA(\psi_\beta)-\ENA(\phi)\bigr),\qquad
 e_\beta(\phi)\nearrow H^{\NA}(\phi)\quad(\beta\to\infty).
\end{equation}
The intersection formula in \cite[Theorem~9.5]{DZ} is proved for
rational $\beta>1$. The envelope operator is a contraction in the uniform
norm, and $\ENA$ is continuous for uniform convergence. Hence the
right-hand side of \eqref{eq:entropy-slope} is continuous in
$\beta$. Approximating a real $\beta>1$ from both sides by rational
parameters and using parameter monotonicity
\cite[Proposition~3.10]{DZ} gives the same identity for every
real $\beta>1$. The formula extends to semipositive model
potentials by relatively ample approximation on a fixed
dominating model. Indeed, the variational formula
\eqref{eq:analytic-quantization} and comparison of the associated
maximal rays from $0$ give, at fixed $\beta$,
\[
 |e_\beta(\phi)-e_\beta(\chi)|
 \le2\beta\|\phi-\chi\|_\infty
\]
for semipositive model potentials $\phi,\chi$. The envelopes
converge uniformly by contraction, and $\ENA$ is continuous for
uniform convergence. The limit as $\beta\to\infty$ follows
directly from \cite[Theorem~6.3 and Corollary~8.4]{DZ}.

The segment estimate \cite[Lemma~7.33]{BJ} gives
\begin{equation}\label{eq:scaling}
 J^{\NA}(t\phi)\le t^{1+1/n}J^{\NA}(\phi),\qquad0\le t\le1.
\end{equation}

\begin{lemma}[Independence of the SNC model]
\label{lem:full-discrepancy}
Let $\phi$ be a semipositive model potential, and let $\mathcal Y$
be a smooth SNC test-configuration model dominating
$X\times\mathbb P^1$ with $\phi=\phi\circ r_{\mathcal Y}$.
Then, for every $h>0$,
\[
 P(\phi+hA_X\circ r_{\mathcal Y})=P(\phi+hA_X).
\]
\end{lemma}
\begin{proof}
Retraction decreases discrepancy, so $a_{\mathcal Y}\le A_X$ and the
left-hand envelope is at most the right-hand one. Conversely, the order relation for retractions and the
antimonotonicity of psh functions
\cite[Theorem~A.1 and Proposition~4.12]{BJ} imply that every psh
function $\chi\le\phi+hA_X$ satisfies
\[
 \chi(v)\le\chi(r_{\mathcal Y}(v))
       \le\phi(r_{\mathcal Y}(v))+hA_X(r_{\mathcal Y}(v))=\phi(v)+h a_{\mathcal Y}(v).
\]
Taking envelopes proves the reverse inequality. See also
\cite[Lemma~9.1 and Remark~9.2]{DZ}.
\end{proof}

\section{Parameter estimates}\label{sec:proof}

Let $(X,\omega_0)$ be a connected compact K\"ahler manifold of
dimension $n\ge1$, and put $A=\alpha(X,[\omega_0])$.
We prove the inequalities in Theorem~\ref{main:sharp-gap}.
The analytic estimates hold in arbitrary K\"ahler classes.
For algebraic configurations we assume a polarization
$[\omega_0]=c_1(L)$, take normal ample test configurations of any
positive exponent, and normalize all invariants to $L$.

\subsection{Concavity and parameters below the \texorpdfstring{$\alpha$}{alpha}-invariant}

For $u,w\in\cH$, write
\[
 F_\beta(w;u)=-\log\int_Xe^{\beta(w-u)}d\nu_0+\beta\bigl(\AM(w)-\AM(u)\bigr).
\]
Thus $\Ent^\beta(u)=\sup_{w\in\cH}F_\beta(w;u)$.
We use $e_\beta$ as defined in Section~\ref{sec:2.4}.
The parameter monotonicity and $d_1$-continuity used below are due
to Darvas--Zhang~\cite[Lemma~3.9 and Proposition~3.10]{DZ}.
The concavity and increment bound follow from the variational formula,
H\"older's inequality, and the concavity of $\AM$.

\begin{lemma}[Parameter concavity and a uniform upper increment]
\label{lem:parameter-concavity}
For every finite-energy potential $u\in\mathcal E^1$, the
quantized entropy $\Ent^\beta(u)$ is finite, nondecreasing and
concave as a function of $\beta>0$, with
\[
 0\le\Ent^B(u)-\Ent^\beta(u)\le(B-\beta)J(u)
 \qquad(0<\beta<B).
\]
For every fixed normal ample test configuration $\Tcal$,
the radial slopes $e_\beta$ and $m_\beta$ are finite and concave
on $(0,\infty)$, and
\[
 0\le e_\beta\le \beta J^{\NA},\qquad
 0\le m_b-m_a=e_b-e_a\le(b-a)J^{\NA}
 \quad(0<a<b).
\]
\end{lemma}

\begin{proof}
First let $u\in\cH$. Fix $a,b>0$ and $0<\theta<1$, and choose
$w_a,w_b\in\cH$. Set
\[
 \beta=\theta a+(1-\theta)b,\qquad
 w=\frac{\theta a w_a+(1-\theta)b w_b}{\beta}.
\]
Then $w\in\cH$ and
\[
 \beta(w-u)=\theta a(w_a-u)+(1-\theta)b(w_b-u).
\]
By H\"older's inequality,
\[
 \int_Xe^{\beta(w-u)}\,d\nu_0
 \le
 \left(\int_Xe^{a(w_a-u)}\,d\nu_0\right)^\theta
 \left(\int_Xe^{b(w_b-u)}\,d\nu_0\right)^{1-\theta},
\]
while the concavity of $\AM$ along affine segments gives
\[
 \beta\AM(w)\ge
 \theta a\AM(w_a)+(1-\theta)b\AM(w_b).
\]
Hence
\[
 F_\beta(w;u)\ge
 \theta F_a(w_a;u)+(1-\theta)F_b(w_b;u).
\]
Taking the supremum over $w_a$ and $w_b$ yields
\begin{equation}\label{eq:entropy-parameter-concavity}
 \Ent^{\theta a+(1-\theta)b}(u)
 \ge
 \theta\Ent^a(u)+(1-\theta)\Ent^b(u).
\end{equation}

Taking $w=u$ gives $\Ent^\beta(u)\ge F_\beta(u;u)=0$.
Since $d\nu_0$ is a probability measure, Jensen's inequality gives
for every $w\in\cH$,
\begin{align*}
 F_\beta(w;u)
 &\le
 -\beta\int_X(w-u)\,d\nu_0
 +\beta\bigl(\AM(w)-\AM(u)\bigr)\\
 &=\beta\bigl[J(u)-J(w)\bigr]
 \le\beta J(u).
\end{align*}
Therefore
\begin{equation}\label{eq:entropy-pointwise-bound}
 0\le\Ent^\beta(u)\le\beta J(u).
\end{equation}

Smooth approximation in $d_1$ and the $d_1$-continuity of $J$
\cite{darvas} and $\Ent^\beta$ \cite[Lemma~3.9]{DZ} extend
\eqref{eq:entropy-parameter-concavity} and
\eqref{eq:entropy-pointwise-bound} to $u\in\mathcal E^1$.
Moreover, parameter monotonicity
\cite[Proposition~3.10]{DZ} gives
\begin{equation}\label{eq:entropy-parameter-monotonicity}
 \Ent^a(u)\le\Ent^b(u),\qquad 0<a<b.
\end{equation}
By \eqref{eq:entropy-pointwise-bound},
$\Ent^\beta(u)$ extends continuously to $\beta=0$ by setting
$\Ent^0(u)=0$. Thus
\eqref{eq:entropy-parameter-concavity},
\eqref{eq:entropy-parameter-monotonicity}, and
\eqref{eq:entropy-pointwise-bound} give
\begin{equation}\label{eq:entropy-secant-bound}
 0\le
 \frac{\Ent^B(u)-\Ent^\beta(u)}{B-\beta}
 \le
 \frac{\Ent^\beta(u)}{\beta}
 \le J(u),
 \qquad 0<\beta<B.
\end{equation}
Multiplying by $B-\beta$ proves the first increment bound.

Fix a normal ample test configuration and its associated ray
$u_t$. Applying
\eqref{eq:entropy-parameter-concavity} to $u_t$, dividing by $t>0$,
and taking lower limits gives
\begin{align}
 e_{\theta a+(1-\theta)b}
 &\ge
 \liminf_{t\to\infty}
 \left(
 \theta\frac{\Ent^a(u_t)}{t}
 +(1-\theta)\frac{\Ent^b(u_t)}{t}
 \right)\notag\\
 &\ge
 \theta e_a+(1-\theta)e_b.
 \label{eq:radial-entropy-concavity}
\end{align}
Hence $\beta\mapsto e_\beta$ is concave. By
\eqref{eq:entropy-pointwise-bound} and \eqref{eq:slope-bridge},
\begin{equation}\label{eq:radial-entropy-bound}
 0\le e_\beta\le\beta J^{\NA}.
\end{equation}
Likewise, \eqref{eq:entropy-parameter-monotonicity} implies
\begin{equation}\label{eq:radial-entropy-monotonicity}
 e_a\le e_b,\qquad 0<a<b.
\end{equation}

The Ricci energy $\mathcal R^{\NA}$ in Section~\ref{sec:2.4} is
finite and independent of $\beta$. Since
\begin{equation}\label{eq:radial-mabuchi-splitting}
 m_\beta=e_\beta-\mathcal R^{\NA},
\end{equation}
$m_\beta$ is finite and concave. By
\eqref{eq:radial-entropy-bound}, $e_\beta$ extends continuously to
$\beta=0$ by setting $e_0=0$. Using
\eqref{eq:radial-entropy-concavity},
\eqref{eq:radial-entropy-bound}, and
\eqref{eq:radial-entropy-monotonicity}, we obtain
\begin{equation}\label{eq:radial-secant-bound}
 0\le
 \frac{m_b-m_a}{b-a}
 =
 \frac{e_b-e_a}{b-a}
 \le
 \frac{e_a-e_0}{a}
 \le J^{\NA},
 \qquad 0<a<b.
\end{equation}
Multiplying \eqref{eq:radial-secant-bound} by $b-a$ proves the
second increment bound.
\end{proof}

\begin{lemma}[Quantized entropy below the $\alpha$-invariant]\label{lem:alpha-cutoff}
Every normal ample test configuration $\Tcal$, with associated
metric $\phi=\phi_\Tcal$, satisfies
\begin{equation}\label{eq:alpha-cutoff}
 e_\beta(\phi)=\beta J^{\NA}(\phi)\qquad(0<\beta\le A).
\end{equation}
\end{lemma}
\begin{proof}
Fix $0<\beta<A$. By the definition of the $\alpha$-invariant,
there exists $B_\beta\ge1$, depending only on the manifold, the reference
form and $\beta$, such that
\[
 \int_Xe^{-\beta(u-\sup_Xu)}d\nu_0\le B_\beta
 \qquad(u\in\cH).
\]
Both $J$ and $\Ent^\beta$ are invariant under adding constants to
$u$, so we may first normalize $\sup_Xu=0$. Put
$C_\beta=\log B_\beta$. Taking $w=0$ in the definition of $\Ent^\beta(u)$ gives
\begin{align*}
 \Ent^\beta(u)
 &\ge F_\beta(0;u)
  =-\log\int_Xe^{-\beta u}d\nu_0-\beta\AM(u)\\
 &\ge-C_\beta-\beta\AM(u)\\
 &=\beta J(u)-\beta\int_Xu\,d\nu_0-C_\beta\\
 &\ge\beta J(u)-C_\beta,
\end{align*}
where the last inequality uses $u\le0$.
Together with \eqref{eq:entropy-pointwise-bound}, this yields
\begin{equation}\label{eq:alpha-entropy-two-sided}
 \beta J(u)-C_\beta
 \le \Ent^\beta(u)
 \le \beta J(u).
\end{equation}
Both sides of \eqref{eq:alpha-entropy-two-sided} are unchanged by
adding constants, so the inequality holds for every smooth $u$.
Smooth $d_1$-approximation extends it to $\mathcal E^1$, since $J$
\cite{darvas} and $\Ent^\beta$ \cite[Lemma~3.9]{DZ} are
$d_1$-continuous.

Along the associated ray $u^\Tcal=\{u_t\}$,
\eqref{eq:alpha-entropy-two-sided} gives
\[
 \beta\frac{J(u_t)}t-\frac{C_\beta}t
 \le\frac{\Ent^\beta(u_t)}t
 \le\beta\frac{J(u_t)}t.
\]
Letting $t\to\infty$ and using \eqref{eq:slope-bridge}, we obtain
\begin{equation}\label{eq:alpha-cutoff-open}
 e_\beta(\phi)=\beta J^{\NA}(\phi),
 \qquad 0<\beta<A.
\end{equation}

By Lemma~\ref{lem:parameter-concavity},
$e_\beta(\phi)$ is finite and concave on $(0,\infty)$, hence
continuous at $A$. Letting $\beta\uparrow A$ in
\eqref{eq:alpha-cutoff-open} gives
$e_A(\phi)=A J^{\NA}(\phi)$.
\end{proof}

\subsection{Sharp parameter increments}

\begin{lemma}[Second parameter derivative]
\label{lem:parameter-second-derivative}
Fix $u\in\cH$ and write $E(b)=\Ent^b(u)$ for $b>0$.
Then $E(b)$ is smooth and
\begin{equation}\label{eq:parameter-differential-inequality}
 0\le -b E''(b)\le(n+1)E'(b).
\end{equation}
In particular, $b^{n+1}E'(b)$ is nondecreasing.
\end{lemma}
\begin{proof}
For $b>0$, let $v=v_b:=u^b$ be the optimizer determined by
\eqref{eq:quantized-optimizer}. Consider
\[
 \Phi(b,v)
 :=
 \log\frac{\omega_v^n}{\omega_0^n}-b(v-u).
\]
Then $\Phi(b,v_b)=0$, and its derivative with respect to $v$ is
\[
 D_v\Phi(b,v_b)[f]=(\Delta_v-b)f\quad\text{with}\quad\Delta_vf=\frac{n\ dd^cf\wedge\omega_v^{n-1}}{\omega_v^n}.
\]
Since $b>0$, the maximum principle gives
$\ker(\Delta_v-b)=\{0\}$. The operator $\Delta_v-b$ is elliptic
of index zero, so it is invertible. The implicit function theorem
shows that $b\mapsto v_b$, and hence $E(b)$, is smooth on
$(0,\infty)$.

Define
\[
 Q_v(f):=-\int_X f\Delta_vf\,d\mu_v
 =\frac{n}{V}\int_X df\wedge d^cf\wedge\omega_v^{n-1}\ge0.
\]

Set
\[
 \eta=v-u,\qquad w=\partial_bv,\qquad z=\Delta_vw.
\]
Differentiating \eqref{eq:quantized-optimizer} with respect to $b$
gives
\begin{equation}\label{eq:optimizer-parameter-variation}
 z=\eta+bw,\qquad
 \partial_bd\mu_v=z\,d\mu_v.
\end{equation}
In particular,
\[
 \int_X z\,d\mu_v=0,
 \qquad
 b\int_X w\,d\mu_v=-\int_X\eta\,d\mu_v.
\]
By the optimizer identity
\cite[Proposition~3.6]{DZ},
\[
 E(b)=b\bigl(\AM(v)-\AM(u)\bigr).
\]
Using
\[
 \frac{d}{db}\AM(v)=\int_Xw\,d\mu_v
\]
and \eqref{eq:optimizer-parameter-variation}, we obtain
\begin{equation}\label{eq:entropy-first-variation}
 E'(b)
 =\AM(v)-\AM(u)-\int_X\eta\,d\mu_v
 =:\JJ_v(u).
\end{equation}

Differentiating \eqref{eq:entropy-first-variation} and using
\eqref{eq:optimizer-parameter-variation}, we obtain
\begin{align}
 E''(b)
 &=\int_Xw\,d\mu_v
   -\int_Xw\,d\mu_v
   -\int_X\eta z\,d\mu_v \notag\\
 &=-\int_X\eta z\,d\mu_v
  =-\|z\|_{L^2(d\mu_v)}^2-bQ_v(w).
 \label{eq:entropy-second-variation}
\end{align}
In particular, $E''(b)\le0$.

Since $\Delta_v$ is self-adjoint with respect to $d\mu_v$ and
$\Delta_vw=z$, the identity $\eta=z-bw$ gives
\begin{equation}\label{eq:Q-eta-expansion}
 Q_v(\eta)
 =Q_v(z)+2b\|z\|_{L^2(d\mu_v)}^2+b^2Q_v(w).
\end{equation}
Combining \eqref{eq:entropy-second-variation} and
\eqref{eq:Q-eta-expansion}, we obtain
\begin{equation}\label{eq:second-variation-Q-bound}
 Q_v(\eta)+bE''(b)
 =Q_v(z)+b\|z\|_{L^2(d\mu_v)}^2\ge0.
\end{equation}

To compare $Q_v(\eta)$ with $\JJ_v(u)$, set
\[
 h=u-v=-\eta,\qquad
 v_s=v+sh=(1-s)v+su,\qquad
 g(s)=\AM(v_s).
\]
The energy variation formula gives
\[
 g''(s)=-Q_{v_s}(h).
\]
Hence Taylor's formula with integral remainder yields
\[
 \JJ_v(u)
 =g(0)+g'(0)-g(1)
 =\int_0^1(1-s)Q_{v_s}(h)\,ds.
\]
Since
\[
 \omega_{v_s}=(1-s)\omega_v+s\omega_u
 \ge(1-s)\omega_v,
\]
we have
\[
 Q_{v_s}(h)\ge(1-s)^{n-1}Q_v(h).
\]
As $h=-\eta$, it follows that
\begin{equation}\label{eq:Q-by-first-variation}
 E'(b)=\JJ_v(u)
 \ge Q_v(\eta)\int_0^1(1-s)^n\,ds
 =\frac{1}{n+1}Q_v(\eta).
\end{equation}

Combining \eqref{eq:entropy-second-variation},
\eqref{eq:second-variation-Q-bound}, and
\eqref{eq:Q-by-first-variation}, we obtain
\[
 0\le-bE''(b)\le Q_v(\eta)\le(n+1)E'(b).
\]
Finally,
\[
 \frac{d}{db}\bigl(b^{n+1}E'(b)\bigr)
 =b^n\bigl((n+1)E'(b)+bE''(b)\bigr)\ge0,
\]
so $b^{n+1}E'(b)$ is nondecreasing.
\end{proof}

For $a>0$, define
\begin{align}
 k_a(b)&=\min\{1,(a/b)^{n+1}\},\label{eq:sharp-kernel}\\
 G_a(\beta)&=\int_\beta^\infty k_a(b)\,db
 =\begin{cases}
 (1+1/n)a-\beta,&0<\beta\le a,\\
 a^{n+1}/(n\beta^n),&\beta\ge a,
 \end{cases}\label{eq:constants}\\
 \Gamma_a(\beta,B)&=G_a(\beta)-G_a(B)
                  =\int_\beta^B k_a(b)\,db.
 \label{eq:sharp-increment-kernel}
\end{align}

At $a=A:=\alpha(X,[\omega_0])$, the formula
\eqref{eq:constants} reduces to \eqref{intro:gap-function}.
For $0<c<A$, choose $C_c\ge1$ such that
\[
 \int_Xe^{-c(u-\sup_Xu)}\,d\nu_0\le C_c,
 \qquad u\in\cH,
\]
and set $L_c=\log C_c$.

\begin{theorem}[Pointwise entropy gap]\label{thm:analytic-gap}
Fix $0<a<c<A$ and set $J_*=L_c/(c-a)$.
For every $u\in\mathcal E^1$ and $0<\beta<B$,
\begin{align}
 \Ent^B(u)-\Ent^\beta(u)
 &\ge\Gamma_a(\beta,B)\bigl(J(u)-J_*\bigr)_+,
 \label{eq:pointwise-increment-sharp}\\
 \Ent(u)-\Ent^\beta(u)
 &\ge G_a(\beta)\bigl(J(u)-J_*\bigr)_+.
 \label{eq:pointwise-gap}
\end{align}
Here $x_+=\max\{x,0\}$.
\end{theorem}
\begin{proof}
Let $u\in\cH$, normalize $\sup_Xu=0$, and write
$E(b)=\Ent^b(u)$ for $b>0$.
Taking $w=0$ in \eqref{eq:analytic-quantization} and using the
definition of $L_c$ gives
\begin{align*}
 E(c)
 &\ge-\log\int_Xe^{-cu}\,d\nu_0-c\AM(u)\\
 &\ge-L_c-c\AM(u)\\
 &=cJ(u)-c\int_Xu\,d\nu_0-L_c\\
 &\ge cJ(u)-L_c,
\end{align*}
where the last inequality follows from $u\le0$. On the other hand,
\eqref{eq:entropy-pointwise-bound} gives $E(a)\le aJ(u)$.

Since $E$ is smooth by Lemma~\ref{lem:parameter-second-derivative}
and concave by \eqref{eq:entropy-parameter-concavity},
\[
 E'(a)\ge\frac{E(c)-E(a)}{c-a}\ge J(u)-J_*.
\]
Moreover, \eqref{eq:entropy-parameter-monotonicity} and the
smoothness of $E$ imply $E'(a)\ge0$. Hence
\begin{equation}\label{eq:initial-derivative-gap}
 E'(a)\ge\bigl(J(u)-J_*\bigr)_+.
\end{equation}

If $0<b\le a$, concavity of $E$ gives
\[
 E'(b)\ge E'(a)
 \ge\bigl(J(u)-J_*\bigr)_+.
\]
If $b\ge a$, Lemma~\ref{lem:parameter-second-derivative} gives the
monotonicity of $b^{n+1}E'(b)$, and therefore
\[
 E'(b)
 \ge\left(\frac{a}{b}\right)^{n+1}E'(a)
 \ge\left(\frac{a}{b}\right)^{n+1}
      \bigl(J(u)-J_*\bigr)_+.
\]
Thus, by \eqref{eq:sharp-kernel},
\begin{equation}\label{eq:derivative-gap-kernel}
 E'(b)\ge
 k_a(b)\bigl(J(u)-J_*\bigr)_+,
 \qquad b>0.
\end{equation}
Integrating \eqref{eq:derivative-gap-kernel} over $[\beta,B]$
and using \eqref{eq:sharp-increment-kernel} proves
\eqref{eq:pointwise-increment-sharp} for $u\in\cH$.

Let $u\in\mathcal E^1$ and choose $u_j\in\cH$ with
$d_1(u_j,u)\to0$. For fixed $\beta$ and $B$, the $d_1$-continuity
of $J$ \cite{darvas} and of $\Ent^\beta$ and $\Ent^B$
\cite[Lemma~3.9]{DZ} allows us to pass to the limit in \eqref{eq:pointwise-increment-sharp}.

Finally, letting $B\to\infty$ gives \eqref{eq:pointwise-gap},
since $\Ent^B(u)\nearrow\Ent(u)$ by \cite[Corollary~3.11]{DZ}
and $\Gamma_a(\beta,B)=G_a(\beta)-G_a(B)\to G_a(\beta)$
by \eqref{eq:constants} and \eqref{eq:sharp-increment-kernel}.
\end{proof}

\begin{theorem}[Parameter bounds for geodesic rays]
\label{thm:gap}\label{main:parameter-gain}
Let $u=\{u_t\}$ be a finite-energy geodesic ray from a smooth
potential
on a connected compact K\"ahler manifold of dimension $n\ge1$.
Write $A=\alpha(X,[\omega_0])$ and
$s_\beta(u)=(\cM^\beta)^{\rad}(u)$.
For every $0<\beta<B$,
\begin{align}
 s_B(u)-s_\beta(u)
 &\ge\Gamma_A(\beta,B)J^{\rad}(u),
 \label{eq:main-gap}\\
 \cM^{\rad}(u)-s_\beta(u)
 &\ge G_A(\beta)J^{\rad}(u).
 \label{eq:main-infinite}
\end{align}
For a normal ample test configuration with associated metric $\phi$,
these inequalities read
\begin{equation}
\begin{aligned}
 m_B(\phi)-m_\beta(\phi)
 &\ge\Gamma_A(\beta,B)J^{\NA}(\phi),\\
 \cM^{\NA}(\phi)-m_\beta(\phi)
 &\ge G_A(\beta)J^{\NA}(\phi).
\end{aligned}
\end{equation}
\end{theorem}

\begin{proof}
By Section~\ref{sec:finite-energy-geometry},
$j:=J^{\rad}(u)$ exists and is finite, and $J(u_t)=O(1+t)$.
We first check that $s_b(u)$ is finite for every $b>0$.
Set $w_t=u_t-\sup_Xu_t\le0$. By \eqref{eq:ordered-d1} and \eqref{eq:sup-mean},
\[
 d_1(0,w_t)=-\AM(w_t)=J(u_t)+O(1)=O(1+t).
\]
Applying \cite[Lemma~5.3]{DZ} to $w_t$ and $0$ yields
$|\AM_{\Ric(\omega_0)}(w_t)|=O(1+t)$.
Using \eqref{eq:Jtheta} and invariance under additive constants,
we obtain $|\cJ_{\Ric(\omega_0)}(u_t)|=O(1+t)$.
Since $0\le\Ent^b(u_t)\le bJ(u_t)$ by
\eqref{eq:entropy-pointwise-bound}, we obtain
$|\cM^b(u_t)|=O(1+t)$, and hence $s_b(u)\in\R$.

Fix $0<\beta<B$ and $0<a<c<A$, and set $J_*=L_c/(c-a)$.
Keeping these parameters fixed as $t\to\infty$, we have
\begin{equation}\label{eq:radial-positive-part-limit}
 q_t:=\frac{(J(u_t)-J_*)_+}{t}
 =\left(\frac{J(u_t)}t-\frac{J_*}t\right)_+
 \longrightarrow j.
\end{equation}
Since $\cJ_{\Ric(\omega_0)}$ is independent of the parameter,
\[
 \cM^B(u_t)-\cM^\beta(u_t)
 =\Ent^B(u_t)-\Ent^\beta(u_t).
\]
Hence \eqref{eq:pointwise-increment-sharp} gives
\[
 \frac{\cM^B(u_t)}t
 \ge\frac{\cM^\beta(u_t)}t+\Gamma_a(\beta,B)q_t.
\]
Taking lower limits and using \eqref{eq:radial-positive-part-limit},
we obtain
\[
 s_B(u)
 \ge\liminf_{t\to\infty}
 \left(\frac{\cM^\beta(u_t)}t+\Gamma_a(\beta,B)q_t\right)
 =s_\beta(u)+\Gamma_a(\beta,B)j.
\]
For each $a<A$, choose $c\in(a,A)$ before letting $t\to\infty$.
The lower bound for $s_B(u)-s_\beta(u)$ is
independent of $c$ and $J_*$. Since
$\Gamma_a(\beta,B)\uparrow\Gamma_A(\beta,B)$ as $a\uparrow A$,
we obtain \eqref{eq:main-gap}. No integrability at $A$ is used.

By \cite[Corollary~3.11]{DZ}, $\Ent^B\le\Ent$, hence
$\cM^B\le\cM$ on $\mathcal E^1$. Thus, for every fixed $B>\beta$,
\[
 \cM^{\rad}(u)\ge s_B(u)
 \ge s_\beta(u)+\Gamma_A(\beta,B)J^{\rad}(u).
\]
Finally, letting $B\to\infty$ and using
$\Gamma_A(\beta,B)\uparrow G_A(\beta)$ proves
\eqref{eq:main-infinite}, also when $\cM^{\rad}(u)=+\infty$.

For a normal ample test configuration $\Tcal$, take
$u=u^\Tcal$ and $\phi=\phi_\Tcal$. By \eqref{eq:slope-bridge}
and \cite[Theorem~1.7(2)]{Li},
\[
 s_\gamma(u)=m_\gamma(\phi),\qquad
 J^{\rad}(u)=J^{\NA}(\phi),\qquad
 \cM^{\rad}(u)=\cM^{\NA}(\phi).
\]
These identities give the stated non-Archimedean inequalities.
\end{proof}

\begin{corollary}[Parameter interpolation]
\label{cor:rescaled-parameter-convexity}
For fixed $u\in\mathcal E^1$, the function
$x\mapsto\Ent^{x^{-1/n}}(u)$ is convex on $(0,\infty)$.
It is strictly convex when $u$ is smooth and nonconstant.
In particular, for $0<a<\beta<B$,
\[
 \Ent^B(u)-\Ent^\beta(u)
 \ge\frac{\beta^{-n}-B^{-n}}{a^{-n}-\beta^{-n}}
       \bigl(\Ent^\beta(u)-\Ent^a(u)\bigr).
\]
\end{corollary}

\begin{proof}
First suppose that $u\in\mathcal E^1$ is smooth, and use the
notation from the proof of Lemma~\ref{lem:parameter-second-derivative}.
The optimizer $v=u^b$ is smooth with $\omega_v>0$.
The derivative calculations remain valid when $\omega_u$ is only
semipositive: along $v_s=(1-s)v+su$, we have
$\omega_{v_s}\ge(1-s)\omega_v>0$ for $0\le s<1$,
and the endpoint identities follow by continuity. Write
\[
 E(b)=\Ent^b(u),\qquad f(x)=E(x^{-1/n}).
\]
For $b=x^{-1/n}$, the chain rule gives
\[
 f''(x)
 =\frac{b^{2n+1}}{n^2}
  \bigl((n+1)E'(b)+bE''(b)\bigr)\ge0
\]
by \eqref{eq:parameter-differential-inequality}. Hence $f$ is
convex.

If $f''(x)=0$, then
\eqref{eq:second-variation-Q-bound} and
\eqref{eq:Q-by-first-variation} give
\[
 0\ge Q_v(\eta)+bE''(b)
   =Q_v(z)+b\|z\|_{L^2(d\mu_v)}^2\ge0.
\]
Thus $z=0$. Since $z=\Delta_vw$, $w$ is constant, and hence
$\eta=z-bw$ is constant. By \eqref{eq:quantized-optimizer},
\[
 \omega_v^n=e^{b\eta}\omega_0^n.
\]
Integrating over $X$ gives $e^{b\eta}=1$, hence $\eta=0$.
Therefore $\omega_v^n=\omega_0^n$. Uniqueness for the complex
Monge--Amp\`ere equation then implies that $v$, and hence $u$, is
constant. Therefore $f$ is strictly convex whenever $u$ is
nonconstant.

For general $u\in\mathcal E^1$, choose $u_j\in\cH$ with
$d_1(u_j,u)\to0$. By the fixed-parameter continuity of
$\Ent^\beta$ \cite[Lemma~3.9]{DZ},
\[
 \Ent^{x^{-1/n}}(u_j)\longrightarrow
 \Ent^{x^{-1/n}}(u),\qquad x>0.
\]
Passing to the limit in the convexity inequality proves convexity
for $u\in\mathcal E^1$.

Finally, if $0<a<\beta<B$, then
$B^{-n}<\beta^{-n}<a^{-n}$. The secant-slope inequality for $f$
gives
\[
 \frac{\Ent^\beta(u)-\Ent^B(u)}
      {\beta^{-n}-B^{-n}}
 \le
 \frac{\Ent^a(u)-\Ent^\beta(u)}
      {a^{-n}-\beta^{-n}},
\]
which is equivalent to the stated interpolation inequality.
\end{proof}

\subsection{Valuative bounds for polarized classes}

For the rest of this section, let $(X,L)$ be a connected smooth
polarized manifold with $[\omega_0]=c_1(L)$. We use the base ideal
of the sections of maximal weight to bound the quantized entropy
of each test configuration.

\begin{lemma}[The threshold $\tau_\phi$]
\label{lem:top-weight-threshold}
Let $(X,L)$ be a connected smooth polarized manifold of dimension
$n\ge1$, and let $\phi\ne0$ be the non-Archimedean metric of a
normal ample test configuration, normalized by $\sup\phi=0$.
Define
\[
 \tau_\phi
 =\inf_{\substack{v\in X^{\rm div}\\\phi(v)<0}}
       \frac{A_X(v)}{-\phi(v)},
\]
where $X^{\rm div}$ denotes the set of positive rational multiples of divisorial valuations.
Write $\phi$ as the rational Fubini--Study function
\[
 \phi=\frac1m\max_i\{\log|s_i|+\lambda_i\},
 \qquad s_i\in H^0(X,mL),\qquad
 \lambda_i\le0,\quad\max_i\lambda_i=0,
\]
where the sections have no common zero. Let $\mathfrak b_0$ be the base ideal generated by the sections
$s_i$ with $\lambda_i=0$. Then
\begin{equation}\label{eq:valuative-top-ideal}
 0<\tau_\phi
   =m\operatorname{lct}(X;\mathfrak b_0)<\infty,
 \qquad \tau_\phi\ge\alpha(X,L).
\end{equation}
Consequently,
$m\operatorname{lct}(X;\mathfrak b_0)$ is independent of the
choice of sections and weights defining $\phi$.
\end{lemma}
\begin{proof}
The ideal $\mathfrak b_0$ is nonzero and proper: if it were the unit
ideal, the normalized Fubini--Study potential would vanish identically.
Thus $0<\operatorname{lct}(X;\mathfrak b_0)<\infty$.
For $v\in X^{\rm div}$, let $v(s_i)$ denote the vanishing order of
$s_i$ along $v$. Evaluating the Fubini--Study function at $v$ gives
\begin{equation}\label{eq:FS-valuative-formula}
 -\phi(v)
 =\frac1m\min_i\{v(s_i)-\lambda_i\}
 \le\frac1m v(\mathfrak b_0).
\end{equation}
Recall that
\begin{equation}\label{eq:lct-valuative}
 \operatorname{lct}(X;\mathfrak b_0)
 =
 \inf_{\substack{v\in X^{\rm div}\\v(\mathfrak b_0)>0}}
 \frac{A_X(v)}{v(\mathfrak b_0)}.
\end{equation}
For $v\in X^{\rm div}$ with $\phi(v)<0$,
\eqref{eq:FS-valuative-formula} and \eqref{eq:lct-valuative} give
\[
 \frac{A_X(v)}{-\phi(v)}
 \ge
 m\frac{A_X(v)}{v(\mathfrak b_0)}
 \ge
 m\operatorname{lct}(X;\mathfrak b_0).
\]
Taking the infimum proves
$\tau_\phi\ge m\operatorname{lct}(X;\mathfrak b_0)$.

For the reverse inequality, let $\mu:Y\to X$ be a log resolution of $\mathfrak b_0$.
Choose a prime divisor $E\subset Y$ computing the log canonical
threshold, and set $v_*=\operatorname{ord}_E$. Thus
\[
 \operatorname{lct}(X;\mathfrak b_0)
 =\frac{A_X(v_*)}{v_*(\mathfrak b_0)}.
\]
Choose $r\in\Q_{>0}$ small enough that
\[
 r\bigl(v_*(\mathfrak b_0)-v_*(s_i)\bigr)\le-\lambda_i
 \qquad(\lambda_i<0).
\]
Then \eqref{eq:FS-valuative-formula} gives
\[
 -\phi(rv_*)=\frac r m v_*(\mathfrak b_0),\qquad
 A_X(rv_*)=rA_X(v_*).
\]
Consequently,
\[
 \tau_\phi
 \le\frac{A_X(rv_*)}{-\phi(rv_*)}
 =m\frac{A_X(v_*)}{v_*(\mathfrak b_0)}
 =m\operatorname{lct}(X;\mathfrak b_0).
\]

Choose a nonzero zero-weight section $s$ with divisor
$D=(s)\in|mL|$. Since $(s)\subset\mathfrak b_0$, monotonicity of
the log canonical threshold gives
\[
 \operatorname{lct}(X;\mathfrak b_0)
 \ge\operatorname{lct}(X;D).
\]
Since $D/m\sim_\Q L$, Demailly's formula
\cite[Theorem~A.3]{DemaillyAlpha} gives
\[
 m\operatorname{lct}(X;D)\ge\alpha(X,L).
\]
Therefore
\[
 \tau_\phi
 =m\operatorname{lct}(X;\mathfrak b_0)
 \ge\alpha(X,L).
\]
This proves \eqref{eq:valuative-top-ideal}; independence of the
choice of sections and weights follows from the definition of $\tau_\phi$.
\end{proof}

\begin{proposition}[Valuative parameter bounds]
\label{prop:valuative-parameter-refinement}
In the setting of Lemma~\ref{lem:top-weight-threshold}, let
$u=\{u_t\}$ be the associated weak geodesic ray and put
$J_\phi=J^{\NA}(\phi)>0$. For every $0<\beta\le\tau_\phi$,
\begin{equation}\label{eq:valuative-initial-segment}
 \lim_{t\to\infty}\frac{\Ent^\beta(u_t)}t
 =e_\beta(\phi)=\beta J_\phi.
\end{equation}
For every $0<\beta<B$,
\begin{align}
 e_B(\phi)-e_\beta(\phi)
 &\ge\Gamma_{\tau_\phi}(\beta,B)J_\phi,
 \label{eq:valuative-finite-increment}\\
 H^{\NA}(\phi)-e_\beta(\phi)
 &\ge G_{\tau_\phi}(\beta)J_\phi.
 \label{eq:valuative-terminal-gap}
\end{align}
\end{proposition}
\begin{proof}
For $b>0$, set
\[
 \ell_b(t):=-\frac1t\log\int_Xe^{-bu_t}\,d\nu_0.
\]
Applying \cite[Proposition~8.1 and Corollary~8.2]{DZ} with the zero
ray as the comparison ray gives
\begin{equation}\label{eq:valuative-logarithmic-slope}
 \liminf_{t\to\infty}\ell_b(t)
 =\inf_{v\in X^{\rm div}}
   \{A_X(v)+b\phi(v)\}.
\end{equation}

Since $\sup\phi=0$, we have $\phi(v)\le0$. If
$0<b\le\tau_\phi$, then
\[
 A_X(v)+b\phi(v)\ge0,\qquad v\in X^{\rm div}.
\]
For any fixed $v\in X^{\rm div}$ and $r\in\Q_{>0}$,
\[
 A_X(rv)=rA_X(v)\longrightarrow0,
 \qquad
 \phi(rv)\longrightarrow0
 \qquad(r\downarrow0),
\]
where the second limit follows from
\eqref{eq:FS-valuative-formula}. Hence
\[
 \inf_{v\in X^{\rm div}}\{A_X(v)+b\phi(v)\}=0.
\]

Taking $w=0$ in \eqref{eq:analytic-quantization} and using
\eqref{eq:entropy-pointwise-bound}, we obtain
\[
 \ell_b(t)-b\frac{\AM(u_t)}t
 \le\frac{\Ent^b(u_t)}t
 \le b\frac{J(u_t)}t.
\]
By \eqref{eq:slope-bridge},
\[
 \frac{\AM(u_t)}t\longrightarrow\ENA(\phi)=-J_\phi,\qquad
 \frac{J(u_t)}t\longrightarrow J_\phi.
\]
Together with \eqref{eq:valuative-logarithmic-slope}, this gives
\[
 bJ_\phi
 \le\liminf_{t\to\infty}\frac{\Ent^b(u_t)}t
 \le\limsup_{t\to\infty}\frac{\Ent^b(u_t)}t
 \le bJ_\phi.
\]
Thus
\eqref{eq:valuative-initial-segment} holds for
$0<b\le\tau_\phi$.

Assume $\tau_\phi\le\beta<B$ and fix
$0<a<c<\tau_\phi$. Since
\[
 B^{-n}<\beta^{-n}<c^{-n}<a^{-n},
\]
Corollary~\ref{cor:rescaled-parameter-convexity} gives
\begin{equation}\label{eq:valuative-secant-comparison}
 \Ent^B(w)-\Ent^\beta(w)
 \ge
 \frac{\beta^{-n}-B^{-n}}{a^{-n}-c^{-n}}
 \bigl(\Ent^c(w)-\Ent^a(w)\bigr),
 \qquad w\in\mathcal E^1.
\end{equation}
By \eqref{eq:valuative-initial-segment},
\begin{equation}\label{eq:valuative-initial-difference-limit}
 \frac{\Ent^c(u_t)-\Ent^a(u_t)}t
 \longrightarrow(c-a)J_\phi.
\end{equation}
Applying \eqref{eq:valuative-secant-comparison} to $u_t$ gives
\[
 \frac{\Ent^B(u_t)}t
 \ge
 \frac{\Ent^\beta(u_t)}t
 +\frac{\beta^{-n}-B^{-n}}{a^{-n}-c^{-n}}
  \frac{\Ent^c(u_t)-\Ent^a(u_t)}t.
\]
Taking lower limits and using
\eqref{eq:valuative-initial-difference-limit}, we obtain
\[
 e_B(\phi)-e_\beta(\phi)
 \ge
 (\beta^{-n}-B^{-n})
 \frac{c-a}{a^{-n}-c^{-n}}J_\phi.
\]
This holds for every $0<a<c<\tau_\phi$. By the mean value theorem,
\[
 \frac{c-a}{a^{-n}-c^{-n}}
 =\frac{\xi^{n+1}}n,\qquad \xi\in(a,c),
\]
and letting $a,c\uparrow\tau_\phi$ gives
\[
 e_B(\phi)-e_\beta(\phi)
 \ge
 \frac{\tau_\phi^{n+1}}n
 (\beta^{-n}-B^{-n})J_\phi
 =
 \Gamma_{\tau_\phi}(\beta,B)J_\phi.
\]

If $B\le\tau_\phi$, then
\eqref{eq:valuative-initial-segment} gives
\[
 e_B(\phi)-e_\beta(\phi)
 =(B-\beta)J_\phi
 =\Gamma_{\tau_\phi}(\beta,B)J_\phi.
\]
If $\beta<\tau_\phi<B$, then
\begin{align*}
 e_B(\phi)-e_\beta(\phi)
 &=e_B(\phi)-e_{\tau_\phi}(\phi)
   +(\tau_\phi-\beta)J_\phi\\
 &\ge
 \bigl(\Gamma_{\tau_\phi}(\tau_\phi,B)
       +\tau_\phi-\beta\bigr)J_\phi\\
 &=\Gamma_{\tau_\phi}(\beta,B)J_\phi.
\end{align*}
This proves \eqref{eq:valuative-finite-increment}.

Finally, \eqref{eq:entropy-slope} gives
$e_B(\phi)\uparrow H^{\NA}(\phi)$ as $B\to\infty$.
Since $\Gamma_{\tau_\phi}(\beta,B)\uparrow G_{\tau_\phi}(\beta)$,
letting $B\to\infty$ in \eqref{eq:valuative-finite-increment}
proves \eqref{eq:valuative-terminal-gap}.
\end{proof}

Since $\tau_\phi\ge\alpha(X,L)$,
Proposition~\ref{prop:valuative-parameter-refinement} improves
Theorem~\ref{thm:gap}. The coefficient in
\eqref{eq:valuative-terminal-gap} is strictly larger when
$\tau_\phi>\alpha(X,L)$. We do not prove that $\tau_\phi$ is the
largest endpoint of the linear segment $\beta\mapsto e_\beta(\phi)$,
or that the bounds are optimal for each $\phi$.

\section{Sharpness of the parameter bounds}\label{sec:sharp-gap}

We apply the known formula for the $\alpha$-invariant of toric
polarizations and use divisors to prove optimality in Theorem~\ref{main:sharp-gap}.
For the toric geometry used below, see \cite{CLS}; for toric
K\"ahler metrics, see
\cite{abreu,guillemin,donaldson02,donaldson09,zhouzhu}.

\subsection{The \texorpdfstring{$\alpha$}{alpha}-invariant of a toric polarization}

For a smooth projective toric manifold, the $\alpha$-invariant is
determined by the widths of its moment polytope. The formula below
is a restatement of Blum--Jonsson~\cite[Corollary~7.16]{BlumJonsson}.

\begin{proposition}[Toric polarizations]\label{prop:toric-roof-gap}
Let $(X,L)$ be a smooth polarized toric manifold; the polarization
may be an ample $\mathbf Q$-line bundle. Let $P$ be its rational moment polytope,
and let $v_\rho$ be the primitive generators of the fan rays. Set
\[
 D_L=\max_\rho\{h_P(v_\rho)+h_P(-v_\rho)\},
 \qquad h_P(v)=\max_{p\in P}\langle p,v\rangle.
\]
Then $\alpha(X,L)=D_L^{-1}$. Consequently, the inequalities of
Theorem~\ref{thm:gap} hold with $A=D_L^{-1}$. In particular, for
every normal ample test configuration and $D_L^{-1}\le a<b$,
\begin{equation}\label{eq:toric-roof-gap}
 \begin{split}
 m_b(\phi)-m_a(\phi)
 &\ge\frac{a^{-n}-b^{-n}}{nD_L^{n+1}}J^{\NA}(\phi),\\
 \cM^{\NA}(\phi)-m_a(\phi)
 &\ge\frac{1}{nD_L^{n+1}a^n}J^{\NA}(\phi).
 \end{split}
\end{equation}
\end{proposition}

\begin{proof}
For integral $L$, choose a torus-invariant divisor
$\sum_\rho a_\rho D_\rho$ representing $L$. With the convention
$P=\{p:\langle p,v_\rho\rangle\ge-a_\rho\}$, ampleness gives
$a_\rho=h_P(-v_\rho)$. By
Blum--Jonsson~\cite[Corollary~7.16]{BlumJonsson},
\[
 \alpha(X,L)^{-1}
 =\max_{\substack{p\in\operatorname{Vert}(P)\\\rho}}
       \bigl(\langle p,v_\rho\rangle+a_\rho\bigr)
 =D_L.
\]
For rational $L$, apply the integral case to $qL$, with $qL$ integral,
and use $D_{qL}=qD_L$ and
$\alpha(X,qL)=q^{-1}\alpha(X,L)$.
Substitution in Theorem~\ref{thm:gap} gives
\eqref{eq:toric-roof-gap}.
\end{proof}

\subsection{Optimality of the decay order}
\label{subsec:dim-obstruction}

\begin{proposition}[Sharpness on each polarization]
\label{prop:dimension-obstruction}
Let $(X,L)$ be a connected smooth polarized manifold of dimension
$n\ge1$, and choose an integer $k\ge1$ such that $kL$ is very ample.
There exists a normal ample test configuration $\Tcal$ for $(X,L)$
such that
\begin{equation}\label{eq:dnc-gap-ratio}
 \frac{\cM^{\NA}(\phi_\Tcal)-(\cM^\beta)^{\rad}(u^\Tcal)}
      {J^{\NA}(\phi_\Tcal)}
 =\frac{k^{n+1}}{n\beta^n},\qquad\beta\ge k.
\end{equation}
Moreover, $\tau_{\phi_\Tcal}=k$, and equality holds in
\eqref{eq:valuative-finite-increment} and
\eqref{eq:valuative-terminal-gap} for every $0<\beta<B$.
Consequently, if $a<n$, there are no constants $C>0$ and
$\beta_0>0$ such that
\[
 \cM^{\NA}(\phi_\Tcal)-(\cM^\beta)^{\rad}(u^\Tcal)
 \ge C\beta^{-a}J^{\NA}(\phi_\Tcal)
\]
for every normal ample test configuration $\Tcal$ and every
$\beta\ge\beta_0$.
\end{proposition}

\begin{proof}
Choose a general divisor $D\in|kL|$. By Bertini's theorem, $D$ is
smooth and irreducible when $n\ge2$, and reduced when $n=1$.
Let $\varpi:Y\to X\times\mathbf P^1$ be the blowup of
$D\times\{0\}$, with exceptional divisor $F$, and set
\[
 \mathcal L_c=\varpi^*p_1^*L-cF,\qquad
 \phi_c(v)=-c\min\{v(D),1\},\qquad 0<c\le1/k.
\]
For rational $c<1/k$, the pair $(Y,\mathcal L_c)$ is an ample
deformation to the normal cone because $L-cD=(1-kc)L$ is ample.
At $c=1/k$, the line bundle
\[
 k\mathcal L_{1/k}=\varpi^*p_1^*(kL)-F
\]
is relatively globally generated: over $\mathbf A^1$, its generators
are the defining section of $D$ and $tH^0(X,kL)$.
The Stein factorization of the induced morphism gives a normal
test configuration with relatively ample polarization and generic
fiber $(X,kL)$. This is a test configuration of exponent $k$ for $(X,L)$,
with normalized non-Archimedean metric $\phi_\Tcal=\phi_{1/k}$.

For $0\le c\le1/k$, the restriction of $\mathcal L_c$ to the
strict-transform component of the central fiber is $(1-kc)L$.
The Monge--Amp\`ere mass at $v_{\rm triv}$ is therefore $(1-kc)^n$.
The measure has total mass one, and the remaining mass
$1-(1-kc)^n$ is supported at the valuations of the components of
$D$, where $\min\{v(D),1\}=A_X(v)=1$.
Writing $r_D(v)=\min\{v(D),1\}$, the energy variation formula gives
\[
 \frac{d}{dc}\ENA(\phi_c)
 =-\int_{X^{\NA}}r_D\,\MA(\phi_c)
 =-\bigl[1-(1-kc)^n\bigr],\qquad\ENA(\phi_0)=0.
\]
Integration and the definition of non-Archimedean entropy yield
\begin{equation}\label{eq:dnc-energy}
 \ENA(\phi_c)=-c+\frac{1-(1-kc)^{n+1}}{k(n+1)},\qquad
 H^{\NA}(\phi_c)=1-(1-kc)^n.
\end{equation}

The discrepancy on the blowup skeleton gives $a_Y=r_D$, so for
$\beta>1$,
\begin{equation}\label{eq:dnc-envelope}
 P(\phi_c+a_Y/\beta)=\phi_{(c-\beta^{-1})_+}.
\end{equation}
Indeed, when $c\ge\beta^{-1}$ the obstacle is the semipositive
potential $\phi_{c-\beta^{-1}}$. When $c<\beta^{-1}$ the obstacle
is nonnegative and vanishes at $v_{\rm triv}$; every semipositive
competitor attains its supremum there, so the envelope is zero.

Set $\phi=\phi_{1/k}$. Formula~\eqref{eq:dnc-energy} and
$\sup\phi=0$ give
\[
 J^{\NA}(\phi)=-\ENA(\phi)=\frac{n}{k(n+1)},\qquad
 H^{\NA}(\phi)=1.
\]
Substituting \eqref{eq:dnc-energy} and \eqref{eq:dnc-envelope}
into the entropy formula~\eqref{eq:entropy-slope} gives
\begin{equation}\label{eq:divisor-example}
 H^{\NA}(\phi)-e_\beta(\phi)=\frac{k^n}{(n+1)\beta^n},
 \qquad\beta\ge k.
\end{equation}
For $k=1$, the endpoint $\beta=1$ follows by continuity of the
finite concave function $\beta\mapsto e_\beta(\phi)$.
Since $\cM^{\NA}(\phi)-m_\beta(\phi)=H^{\NA}(\phi)-e_\beta(\phi)$,
dividing \eqref{eq:divisor-example} by $J^{\NA}(\phi)=n/(k(n+1))$
and using \eqref{eq:slope-bridge} proves \eqref{eq:dnc-gap-ratio}.
For $a<n$, the ratio of the right-hand side of
\eqref{eq:dnc-gap-ratio} to $\beta^{-a}$ is
$(k^{n+1}/n)\beta^{a-n}\to0$ as $\beta\to\infty$. This proves the
last assertion for the fixed polarization.

Since $(X,D)$ is log canonical, $A_X(v)\ge v(D)$ for every
valuation $v$.
Since $-\phi(v)=k^{-1}\min\{v(D),1\}$, whenever $\phi(v)<0$,
\[
 \frac{A_X(v)}{-\phi(v)}
 \ge k\,\frac{v(D)}{\min\{v(D),1\}}\ge k.
\]
The valuation $\operatorname{ord}_{D_i}$ of any component $D_i$
has $A_X(v)=v(D)=1$. It attains equality in this bound, so $\tau_\phi=k$.
For $0<\beta\le k$, \eqref{eq:valuative-initial-segment} yields
$e_\beta(\phi)=\beta J^{\NA}(\phi)$, so
\[
 \frac{H^{\NA}(\phi)-e_\beta(\phi)}{J^{\NA}(\phi)}
 =\frac{n+1}{n}k-\beta=G_k(\beta).
\]
Together with \eqref{eq:divisor-example}, this proves
\begin{equation}\label{eq:dnc-terminal-equality}
 H^{\NA}(\phi)-e_\beta(\phi)=G_k(\beta)J^{\NA}(\phi),
 \qquad\beta>0,
\end{equation}
which is equality in \eqref{eq:valuative-terminal-gap} because
$\tau_\phi=k$. Subtracting \eqref{eq:dnc-terminal-equality} at
$\beta$ and $B$ gives
$e_B(\phi)-e_\beta(\phi)=\Gamma_k(\beta,B)J^{\NA}(\phi)$,
proving equality in \eqref{eq:valuative-finite-increment} for every
$0<\beta<B$.
\end{proof}

\subsection{Equality in the parameter bounds}

\begin{proposition}[Optimality of the coefficients]
\label{prop:sharp-universal-kernel}
For $(X,L)=(\mathbf P^n,\mathcal O(1))$, there is a normal ample
test configuration whose ray attains equality in both inequalities
of Theorem~\ref{main:sharp-gap} for every $0<\beta<B$.
After scaling the K\"ahler class, equality holds for every prescribed
value $A>0$ of the $\alpha$-invariant.
\end{proposition}

\begin{proof}
Proposition~\ref{prop:toric-roof-gap} gives
$\alpha(\mathbf P^n,\mathcal O(1))=1$.
Apply Proposition~\ref{prop:dimension-obstruction} to a hyperplane
$D\in|\mathcal O(1)|$ with $k=1$. The metric $\phi$ has
$\tau_\phi=1$ and attains equality in
\eqref{eq:valuative-finite-increment} and
\eqref{eq:valuative-terminal-gap} for every $0<\beta<B$.
Since
\[
 m_B(\phi)-m_\beta(\phi)
 =e_B(\phi)-e_\beta(\phi),
 \qquad
 \cM^{\NA}(\phi)-m_\beta(\phi)
 =H^{\NA}(\phi)-e_\beta(\phi),
\]
\eqref{eq:slope-bridge} gives equality in both inequalities of
Theorem~\ref{main:sharp-gap}.

For $\lambda>0$, scaling the reference form and the associated ray
gives
\begin{equation}\label{eq:sharpness-class-scaling}
 \begin{aligned}
  J_{\lambda\omega_0}^{\rad}(\lambda u)
    &=\lambda J_{\omega_0}^{\rad}(u),\\
  (\cM^\beta_{\lambda\omega_0})^{\rad}(\lambda u)
    &=(\cM^{\lambda\beta}_{\omega_0})^{\rad}(u).
 \end{aligned}
\end{equation}
The ordinary Mabuchi radial slope is unchanged, and the
$\alpha$-invariant is divided by $\lambda$; see the proof of
Proposition~\ref{prop:kahler-scale}.
The scaling identities
\[
 \lambda^{-1}G_1(\lambda\beta)=G_{1/\lambda}(\beta),
 \qquad
 \lambda^{-1}\Gamma_1(\lambda\beta,\lambda B)
   =\Gamma_{1/\lambda}(\beta,B)
\]
therefore show that the scaled ray attains equality for
$A=1/\lambda$. Every prescribed $A>0$ is obtained in this way.
\end{proof}

\begin{lemma}[$\alpha$-invariants of products]
\label{lem:product-alpha}
Let $(X_i,\Omega_i)$ be connected compact K\"ahler manifolds.
For $X=\prod_iX_i$ and $\Omega=\sum_i\operatorname{pr}_i^*\Omega_i$,
\[
 \alpha(X,\Omega)=\min_i\alpha(X_i,\Omega_i).
\]
\end{lemma}

\begin{proof}
It suffices to consider two factors. Choose product reference
forms, with normalized volume measures $d\nu_1$ and $d\nu_2$,
and fix $0<c<\min\{\alpha(X_1,\Omega_1),\alpha(X_2,\Omega_2)\}$.
For a smooth $\Omega$-potential $u$ with $\sup_Xu=0$, the function
$v(x)=\max_{y\in X_2}u(x,y)$ is continuous, $\Omega_1$-psh, and
satisfies $\sup_{X_1}v=0$. Finite maxima of the slices
$u(\,\cdot\,,y)$ converge uniformly to $v$. Regularizing these
maxima and normalizing their suprema gives smooth approximants,
so the uniform $\alpha$-integrability estimate on $X_1$ applies to $v$.

For each $x$, the slice $u(x,\cdot)-v(x)$ has supremum zero.
Applying the integrability estimate first on $X_2$ and then on
$X_1$ gives
\[
 \int_X e^{-cu}\,d\nu_1d\nu_2
 \le C_2\int_{X_1}e^{-cv}\,d\nu_1
 \le C_1C_2.
\]
This proves the lower bound for the product $\alpha$-invariant.
Pulling back potentials from either factor gives the reverse
bound, and induction proves the general case.
\end{proof}

For a smooth projective curve $C$, a point $p\in C$, and a
positive integer $d$, one has $\alpha(C,\mathcal O_C(dp))=1/d$.
Indeed, an effective divisor $D=\sum_i a_ip_i\in|m\,dp|$ satisfies
\[
 \operatorname{lct}(C;D)=\frac1{\max_i a_i}\ge\frac1{md},
\]
while $dp\in|dp|$ gives equality. Demailly's formula
\cite[Theorem~A.3]{DemaillyAlpha} proves the assertion.
Consequently, for $X=\prod_{i=1}^nC_i$ and
$L=\boxtimes_{i=1}^n\mathcal O_{C_i}(dp_i)$,
Lemma~\ref{lem:product-alpha} gives $\alpha(X,L)=1/d$.

\begin{proposition}[Sharpness with discrete automorphisms]
\label{prop:discrete-sharpness}
Let $D$ be a reduced ample Cartier divisor on a connected smooth
projective $n$-fold $X$. Suppose that $(X,D)$ is log canonical and
that $\alpha(X,L)=1/d$ for $L=\mathcal O_X(dD)$ with $d\ge1$ an integer.
Then there is a normal ample test configuration with metric $\phi$
such that, for every $0<\beta<B$,
\begin{equation}\label{eq:discrete-sharpness}
 \frac{H^{\NA}(\phi)-e_\beta(\phi)}{J^{\NA}(\phi)}
   =G_{1/d}(\beta),
 \qquad
 \frac{e_B(\phi)-e_\beta(\phi)}{J^{\NA}(\phi)}
   =\Gamma_{1/d}(\beta,B).
\end{equation}
The hypotheses hold for
\[
 X=\prod_{i=1}^nC_i,\qquad
 D=\sum_{i=1}^n\operatorname{pr}_i^*p_i,
\]
where each $C_i$ is a smooth curve of genus at least two and
$p_i\in C_i$. In this case $\Aut_0(X)$ is trivial and $c_1(L)$
contains a unique cscK metric. Thus the coefficients in
Theorem~\ref{main:sharp-gap} remain sharp under these additional
conditions.
\end{proposition}

\begin{proof}
Blow up $(\mathcal I_D,t)$ on $X\times\mathbf A^1$, and write
$\pi:Y\to X\times\mathbf A^1$ and $F$ for the blowup and its
exceptional divisor. If $D=(f=0)$ locally, the charts are
$t=fs$ and $f=tz$. The first chart is smooth. The second chart
is a hypersurface whose singular locus is contained in
$\operatorname{Sing}(D)\times\{t=z=0\}$, of codimension at least
three because $D$ is reduced. A hypersurface is Cohen--Macaulay,
so Serre's criterion shows that $Y$ is normal.
The central fiber is reduced, with strict-transform component
$X_0$ and one exceptional component $E_i$ for each prime component
$D_i$ of $D$. At the generic point of $E_i$, one has
$\operatorname{ord}_{E_i}(t)=1$ and
$v_{E_i}=\operatorname{ord}_{D_i}$.

For $0\le c\le d$, set
\[
 \mathcal A_c=\pi^*p_1^*L-cF,
 \qquad \phi_c=-cr,
 \qquad r(v)=\min\{v(D),1\}.
\]
Choose $m$ such that $mL$ is globally generated. The sections
$s_D^{md}$ and $t^{md}H^0(X,mL)$ generate $m\mathcal A_d$ on
the two charts, so $\mathcal A_d$ is relatively semiample.
Taking the relative Proj and then its normalization gives a normal
ample test configuration with metric $\phi_d$.
For $0\le c\le d$,
\[
 \mathcal A_c
 =\left(1-\frac cd\right)\mathcal A_0
  +\frac cd\,\mathcal A_d.
\]
Since $\mathcal A_0$ and $\mathcal A_d$ are relatively nef,
$\mathcal A_c$ is relatively nef.

Since $\mathcal A_c|_{X_0}=L-cD=(d-c)D$, the mass of
$\MA(\phi_c)$ at $v_{\rm triv}$ is $(1-c/d)^n$.
The remaining mass is supported at the valuations
$\operatorname{ord}_{D_i}$, where $r=A_X=1$.
Energy variation therefore gives
\[
 \frac{d}{dc}\ENA(\phi_c)
 =-\int_{X^{\NA}}r\,\MA(\phi_c)
 =-\bigl[1-(1-c/d)^n\bigr].
\]
Integrating from $c=0$ yields
\begin{equation}\label{eq:lc-divisor-energy}
 \ENA(\phi_c)
 =-c+\frac d{n+1}\left[1-(1-c/d)^{n+1}\right],
 \qquad
 H^{\NA}(\phi_c)=1-(1-c/d)^n.
\end{equation}
In particular, $\phi=\phi_d$ satisfies
$J^{\NA}(\phi)=nd/(n+1)$ and $H^{\NA}(\phi)=1$.

Choose a smooth SNC model $\mathcal Y$ as in \eqref{eq:envelope}
dominating the blowup, such that $r$ factors through the retraction
$r_{\mathcal Y}$. Put
\(
 a_{\mathcal Y}=A_X\circ r_{\mathcal Y}.
\)
Log canonicity gives $A_X(v)\ge v(D)\ge r(v)$; hence
$a_{\mathcal Y}(v)\ge r(r_{\mathcal Y}(v))=r(v)$ on all of
$X^{\NA}$.
For $\beta>1$, set $c'=(c-\beta^{-1})_+$ and $\chi=\phi_{c'}$.
Since $c-c'\le\beta^{-1}$, we have
$\chi\le\phi_c+a_{\mathcal Y}/\beta$.
If $c'>0$, equality holds on the support of $\MA(\chi)$:
both sides vanish at $v_{\rm triv}$, and
$a_{\mathcal Y}=r=1$ at each $\operatorname{ord}_{D_i}$.
If $c'=0$, the support consists only of $v_{\rm triv}$ and
equality holds there as well. Every psh competitor below the
obstacle is therefore at most $\chi$ almost everywhere with respect to
$\MA(\chi)$. The domination principle
\cite[Corollary~10.6]{BJ} shows that every such competitor is at
most $\chi$ on $X^{\NA}$. Since $\chi$ itself is a competitor,
\begin{equation}\label{eq:lc-divisor-envelope}
 P\!\left(\phi_c+\frac{a_{\mathcal Y}}{\beta}\right)
   =\phi_{(c-\beta^{-1})_+}.
\end{equation}
Combining \eqref{eq:lc-divisor-envelope} with
\eqref{eq:entropy-slope} and \eqref{eq:lc-divisor-energy} at $c=d$
gives
\begin{equation}\label{eq:lc-divisor-tail}
 e_\beta(\phi)=1-\frac1{(n+1)d^n\beta^n},\qquad \beta>1.
\end{equation}

Write $A=1/d$ and $J=J^{\NA}(\phi)=nd/(n+1)$.
Equation~\eqref{eq:alpha-cutoff} gives $e_A(\phi)=AJ$.
For a fixed $\beta>A$, choose $B>\max\{1,\beta\}$.
Equation~\eqref{eq:lc-divisor-tail} gives
\[
 e_B(\phi)-e_A(\phi)=\Gamma_A(A,B)J.
\]
Since $\mathcal R^{\NA}(\phi)$ is independent of the parameter,
Theorem~\ref{thm:gap} gives
\[
 \begin{aligned}
  e_\beta(\phi)-e_A(\phi)
  &\ge\Gamma_A(A,\beta)J,\\
  e_B(\phi)-e_\beta(\phi)
  &\ge\Gamma_A(\beta,B)J.
 \end{aligned}
\]
Since
\[
 \Gamma_A(A,B)
 =\Gamma_A(A,\beta)+\Gamma_A(\beta,B),
\]
both inequalities are equalities.
For $\beta>A$, we have
$e_\beta(\phi)=e_A(\phi)+\Gamma_A(A,\beta)J$.
Combining this identity with \eqref{eq:alpha-cutoff} gives
\begin{equation}\label{eq:lc-divisor-entropy}
 e_\beta(\phi)=
 \begin{cases}
  nd\beta/(n+1),&0<\beta\le1/d,\\[2mm]
  1-1/((n+1)d^n\beta^n),&\beta\ge1/d.
 \end{cases}
\end{equation}
Using $H^{\NA}(\phi)=1$ and $J=nd/(n+1)$ in
\eqref{eq:lc-divisor-entropy} proves \eqref{eq:discrete-sharpness}.

For the product of curves in the proposition, $D$ is reduced,
SNC, and ample, and Lemma~\ref{lem:product-alpha} gives $\alpha(X,L)=1/d$.
The factors have genus at least two, so the product has no
nonzero holomorphic vector fields and $\Aut_0(X)=\{1\}$.
The product of constant-curvature metrics in the prescribed
classes is cscK, and uniqueness follows from \cite{bb}.

Finally, scaling the K\"ahler class by any $\lambda>0$ preserves
the cscK property and the triviality of $\Aut_0(X)$.
The identities \eqref{eq:sharpness-class-scaling}, applied to the
associated rays, preserve equality in the parameter
bounds and replace the $\alpha$-invariant by $1/(d\lambda)$.
Thus these examples realize every prescribed positive value of
the $\alpha$-invariant in compact K\"ahler classes.
\end{proof}

This proves the optimality assertion in
Theorem~\ref{main:sharp-gap}. The optimal coefficient for an
individual polarization may be larger.
Proposition~\ref{prop:dimension-obstruction} shows that the
exponent $n$ is optimal on every fixed polarization.

\section{Stability criteria and quantitative testing}
\label{sec:criteria}\label{sec:analytic-gap}

We derive the existence criteria from Theorem~\ref{main:sharp-gap}
and the correspondences of Darvas--Zhang~\cite{DZ} and
Trusiani~\cite{Tru} for polarized manifolds. We then discuss analytic
coercivity and scaling. The correspondence for arbitrary K\"ahler
classes is proved in Section~\ref{ordinary:kahler}.

\subsection{The polarized correspondence}

\begin{proof}[Proof of Corollary~\ref{cor:polarized-correspondence}]
Darvas--Zhang~\cite[Theorem~1.1]{DZ} prove that conditions
\textup{(i)} and \textup{(iii)} are equivalent. If
\textup{(iii)} holds for some $B>0$ and $\gamma>0$, then
\[
 m_B(\phi_\Tcal)
 \ge\gamma J^{\NA}(\phi_\Tcal)\ge0,
\]
so \textup{(ii)} holds with $\beta=B$.

Conversely, assume \textup{(ii)} holds for some $\beta>0$.
For every $B>\beta$, Theorem~\ref{thm:gap} gives
\[
 m_B(\phi_\Tcal)
 \ge
 m_\beta(\phi_\Tcal)
 +\Gamma_A(\beta,B)J^{\NA}(\phi_\Tcal)
 \ge
 \Gamma_A(\beta,B)J^{\NA}(\phi_\Tcal).
\]
Since $\Gamma_A(\beta,B)>0$ is independent of $\Tcal$,
\textup{(iii)} holds with
\(
 \gamma=\Gamma_A(\beta,B).
\)

Assume that \textup{(iii)} holds with a constant $\gamma>0$.
Theorem~\ref{thm:gap} gives
\[
 \cM^{\NA}(\phi_\Tcal)
 \ge m_B(\phi_\Tcal)
 \ge\gamma J^{\NA}(\phi_\Tcal),
\]
while \eqref{ordinary:correction} gives
\(
 \mathrm{DF}(\Tcal)\ge\cM^{\NA}(\phi_\Tcal).
\)
Hence
\(
 \mathrm{DF}(\Tcal)
 \ge\gamma J^{\NA}(\phi_\Tcal),
\)
and \textup{(iv)} follows.

Conversely, assume \textup{(iv)} holds with a constant
$\delta>0$. For a normal ample test configuration $\Tcal$, let
$\Tcal^{(d)}$ be the normalization of a sufficiently divisible
degree-$d$ base change. By
\cite[Propositions~7.16 and~8.2]{BHJ},
\[
 \mathrm{DF}(\Tcal^{(d)})
 =\cM^{\NA}(\phi_{\Tcal^{(d)}})
 =d\,\cM^{\NA}(\phi_\Tcal),
 \qquad
 J^{\NA}(\phi_{\Tcal^{(d)}})
 =d\,J^{\NA}(\phi_\Tcal).
\]
Applying \textup{(iv)} to $\Tcal^{(d)}$ and dividing by $d$ gives
\begin{equation}\label{eq:uniform-MNA}
 \cM^{\NA}(\phi_\Tcal)
 \ge\delta J^{\NA}(\phi_\Tcal)
\end{equation}
for every normal ample test configuration $\Tcal$.

Let $\phi$ be a model metric in
Li's criterion~\cite[Theorem~1.10]{Li}. The special Fujita
approximation of
Trusiani~\cite[Corollary~A and its proof]{Tru}, applied with the
trivial group, gives test-configuration metrics $\phi_j$ such that
\[
 \phi_j\longrightarrow\phi
 \quad\text{strongly},\qquad
 H^{\NA}(\phi_j)\longrightarrow H^{\NA}(\phi).
\]
Strong convergence gives $J^{\NA}(\phi_j)\longrightarrow J^{\NA}(\phi)$,
while the strong continuity of $\mathcal R^{\NA}$
\cite[Proposition~2.17]{Li} gives
$\mathcal R^{\NA}(\phi_j)\longrightarrow\mathcal R^{\NA}(\phi)$.
Since $\cM^{\NA}=H^{\NA}-\mathcal R^{\NA}$, it follows that
$\cM^{\NA}(\phi_j)\longrightarrow\cM^{\NA}(\phi)$.
Passing to the limit in \eqref{eq:uniform-MNA} yields
$\cM^{\NA}(\phi)\ge\delta J^{\NA}(\phi)$.
By Li~\cite[Theorem~1.10]{Li}, the Mabuchi functional is coercive.
Darvas--Zhang~\cite[Theorem~4.4]{DZ} then give a unique cscK metric. Thus \textup{(iv)} implies \textup{(i)}.
\end{proof}

\begin{corollary}[Negative lower bounds]\label{cor:existence}
Let $(X,L)$ be a connected smooth polarized manifold and put
$A=\alpha(X,L)$. Suppose that, for fixed $\beta>0$ and
\(
 0\le\varepsilon<G_A(\beta),
\)
one has
\begin{equation}\label{eq:kahler-semistability}
 (\cM^\beta)^{\rad}(u^\Tcal)
 \ge-\varepsilon J^{\rad}(u^\Tcal)
 \qquad\text{for every }\Tcal
\end{equation}
in the family of all normal ample test configurations of arbitrary
positive exponent. Then $c_1(L)$ contains a unique cscK metric.

More precisely, for every $B>\beta$, one has
\begin{equation}\label{eq:negative-lower-gain}
 (\cM^B)^{\rad}(u^\Tcal)
 \ge
 \bigl(\Gamma_A(\beta,B)-\varepsilon\bigr)
 J^{\rad}(u^\Tcal)
 \qquad\text{for every }\Tcal.
\end{equation}
The coefficient in \eqref{eq:negative-lower-gain} is strictly positive
whenever $\Gamma_A(\beta,B)>\varepsilon$.
\end{corollary}

\begin{proof}
Combining \eqref{eq:kahler-semistability} with
Theorem~\ref{thm:gap} gives \eqref{eq:negative-lower-gain} for every
$B>\beta$. Since
\[
 \Gamma_A(\beta,B)\longrightarrow G_A(\beta)
 \qquad(B\to\infty)
\]
and $\varepsilon<G_A(\beta)$, there exists a finite $B>\beta$,
independent of $\Tcal$, such that $\Gamma_A(\beta,B)>\varepsilon$.
Thus condition \textup{(iii)} of
Corollary~\ref{cor:polarized-correspondence} holds, so $c_1(L)$
contains a unique cscK metric.

Taking $\varepsilon=\vartheta G_A(\beta)$, with
$0\le\vartheta<1$, in \eqref{eq:negative-lower-gain} gives
\eqref{intro:robust-gain}. Conversely, a unique cscK metric gives,
by condition \textup{(ii)} of
Corollary~\ref{cor:polarized-correspondence}, a fixed finite
$\beta>0$ such that
$(\cM^\beta)^{\rad}(u^\Tcal)\ge0$ for every configuration in that
family. Since $J^{\rad}(u^\Tcal)\ge0$, this implies
\eqref{intro:robust-unreduced} and proves
Corollary~\ref{intro:quantitative-correspondences}\textup{(i)}.
\end{proof}

If $\beta\ge A$, then
\[
 \Gamma_A(\beta,B)
 =
 \frac{A^{n+1}}{n}
 \bigl(\beta^{-n}-B^{-n}\bigr).
\]
Hence $\Gamma_A(\beta,B)>\varepsilon$ whenever
\[
 B>
 \left(
 \beta^{-n}-\frac{n\varepsilon}{A^{n+1}}
 \right)^{-1/n}.
\]
The assumption $\varepsilon<G_A(\beta)=A^{n+1}/(n\beta^n)$
gives $\beta^{-n}-n\varepsilon/A^{n+1}>0$.

\begin{remark}\label{finite:tolerance-endpoint}
On $(\mathbf P^1,\mathcal O(1))$, the Fubini--Study reference
has $A=1$ and $\mathcal R^{\NA}=2J^{\NA}$.
Lemma~\ref{lem:alpha-cutoff} and Theorem~\ref{thm:gap} give
\(
 m_\beta(\phi)\ge-G_1(\beta)J^{\NA}(\phi)
\) $(\beta>0)$
for all normal ample test configurations. Nonunitary projective
automorphisms produce distinct cscK forms in the class, so
the bound with $\varepsilon\ge G_1(\beta)$ does not imply uniqueness.
\end{remark}

\subsection{Scaling and analytic coercivity}

\begin{corollary}[A pointwise coercivity criterion]
\label{cor:pointwise-coercivity}
Let $\Omega=[\omega_0]$ be a K\"ahler class on a connected compact
K\"ahler manifold, and set $A=\alpha(X,\Omega)$. Suppose that, for
some $\beta>0$, $C\in\R$, and
$0\le\varepsilon<G_A(\beta)$, one has
$\cM^\beta(u)\ge-\varepsilon J(u)-C$ for every $u\in\cH$.
Then $\cM$ is coercive.
\end{corollary}

\begin{proof}
Choose $0<a<c<A$ such that $G_a(\beta)>\varepsilon$, and put
$J_*=L_c/(c-a)$ as in Theorem~\ref{thm:analytic-gap}.
Equation~\eqref{eq:pointwise-gap} and the assumed lower bound give
\[
 \cM(u)\ge
 \bigl(G_a(\beta)-\varepsilon\bigr)J(u)
 -C-G_a(\beta)J_*.
\]
Since $G_a(\beta)-\varepsilon>0$, $\cM$ is coercive.
\end{proof}

Let $(X,\omega_0)$ be a connected compact K\"ahler manifold and
set $\Omega=[\omega_0]$. Define
\[
 \lambda^{\rm Kah}_\beta(X,\Omega)=
 \inf_{\substack{\Tcal\text{ relatively K\"ahler}\\
                  J^{\rad}(u^\Tcal)>0}}
 \frac{(\cM^\beta)^{\rad}(u^\Tcal)}
      {J^{\rad}(u^\Tcal)},
\]
where $\Tcal$ ranges over the testing family in
Corollary~\ref{cor:kahler-correspondence}.

\begin{proposition}[Scaling of a compact K\"ahler class]
\label{prop:kahler-scale}
For every $r>0$ and $\beta>0$,
\begin{equation}\label{eq:kahler-scale}
 \lambda^{\rm Kah}_\beta(X,r\Omega)
 =\frac1r\lambda^{\rm Kah}_{r\beta}(X,\Omega),
 \qquad
 \alpha(X,r\Omega)=\frac1r\alpha(X,\Omega).
\end{equation}
Hence the semistable parameters for $r\Omega$ are precisely
$r^{-1}$ times those for $\Omega$. If the testing family contains
a configuration of positive norm and the set of semistable parameters
is nonempty, denote the infimum of the semistable parameters by
$\beta_{\rm crit}(\Omega)$.
Then
\[
 \beta_{\rm crit}(r\Omega)
 =r^{-1}\beta_{\rm crit}(\Omega),
 \qquad
 \frac{\beta_{\rm crit}(r\Omega)}
      {\alpha(X,r\Omega)}
 =
 \frac{\beta_{\rm crit}(\Omega)}
      {\alpha(X,\Omega)}.
\]
\end{proposition}

\begin{proof}
Scaling the relatively K\"ahler class by $r$ gives a bijection
between the test configurations for $\Omega$ and $r\Omega$, with
associated rays related by $u_t\mapsto ru_t$. The normalized
reference measure is unchanged. From the definitions,
\[
 \AM_{r\omega_0}(ru)=r\AM_{\omega_0}(u),\qquad
 J_{r\omega_0}(ru)=rJ_{\omega_0}(u),\qquad
 \Ent^\beta_{r\omega_0}(ru)
 =\Ent^{r\beta}_{\omega_0}(u).
\]
Moreover, $\Ric(r\omega_0)=\Ric(\omega_0)$ and
$\bar S_{r\omega_0}=r^{-1}\bar S_{\omega_0}$, while
$\AM_{\Ric(r\omega_0),r\omega_0}(ru)
=\AM_{\Ric(\omega_0),\omega_0}(u)$. Hence
\[
 \cJ_{\Ric(r\omega_0),r\omega_0}(ru)
 =\cJ_{\Ric(\omega_0),\omega_0}(u),
 \qquad
 \cM^\beta_{r\omega_0}(ru)
 =\cM^{r\beta}_{\omega_0}(u).
\]
Taking radial slopes and then the infimum over the testing family
proves the first identity in \eqref{eq:kahler-scale}.

Under $u\mapsto ru$, normalized $\omega_0$-psh potentials
correspond to normalized $r\omega_0$-psh potentials, and
$e^{-au}=e^{-(a/r)(ru)}$. Therefore
$\alpha(X,r\Omega)=r^{-1}\alpha(X,\Omega)$.
Finally, $\beta$ is semistable for $r\Omega$ exactly when
$r\beta$ is semistable for $\Omega$. Taking infima gives the
identities for $\beta_{\rm crit}$.
\end{proof}

\section{Equivariant testing and torus twists}
\label{sec:groups}

Let $(X,L)$ be a connected smooth polarized manifold and set
$A=\alpha(X,L)$. Write
$\G=\Aut_0(X,L)$. Let $\HH\subset\G$ be a connected reductive
algebraic subgroup, let $\ZZ_\HH=Z(\HH)^\circ$, and fix a maximal
torus $\SSS\subset\HH$. When necessary, we pass to a finite cover
and a positive power of $L$ to define the actions and linearizations.
All invariants are normalized to $L$.

For an $\HH$-equivariant normal ample test configuration, define
\begin{align}
 J_{\HH}^{\NA}(\phi)
 &=\inf_{\xi\in N_\R(\ZZ_\HH)}
      J^{\NA}(\phi_\xi),
 \label{eq:central-norm}\\
 \RR_{\beta,\HH}^{\rm max}(\Tcal)
 &=\sup_{\xi\in N_\Q(\SSS)}
      (\cM^\beta)^{\rad}(u^{\Tcal_\xi}).
 \label{eq:full-sup}
\end{align}
Equation~\eqref{eq:central-norm} also defines $J_{\HH}^{\NA}$ on
$(\mathcal E^{1,\NA})^{K_\HH}$, where $K_\HH\subset\HH$ is
maximal compact; see \cite[Section~2.1.3]{Li}. When $\G$ is reductive and $\HH=\G$,
\eqref{eq:central-norm} is the central reduced norm. After clearing denominators and making a normalized base change,
a rational twist in \eqref{eq:full-sup} is again a normal ample test
configuration. Thus Theorem~\ref{thm:gap} applies.

\begin{lemma}[Weyl averaging of the reduced norm]
\label{lem:weyl}
Let $\phi$ be the metric of an $\HH$-equivariant normal ample test
configuration. Then
\begin{equation}\label{eq:weyl-norm}
 \inf_{\xi\in N_\R(\SSS)}J^{\NA}(\phi_\xi)
 =
 \inf_{\xi\in N_\R(\ZZ_\HH)}J^{\NA}(\phi_\xi)
 =
 J_{\HH}^{\NA}(\phi).
\end{equation}
The same infima are obtained over rational cocharacters.
\end{lemma}

\begin{proof}
Set $F(\xi)=J^{\NA}(\phi_\xi)$ on $N_\R(\SSS)$.
We use the linearizations induced by twisting, without separately
normalizing $\sup\phi_\xi=0$. Since $J^{\NA}$ is unchanged by
adding constants, this choice does not affect $F$.
In an $\SSS$-weight presentation of $\phi$, the function
$\xi\mapsto\sup\phi_\xi$ is the maximum of finitely many affine
functions, while $\xi\mapsto\ENA(\phi_\xi)$ is affine by
\cite[Example~2.21 and Lemma~2.23]{Li}. Hence
$F(\xi)=\sup\phi_\xi-\ENA(\phi_\xi)$ is convex and continuous.

Let $W_\HH$ be the Weyl group of $(\HH,\SSS)$. Since $\phi$ is
$\HH$-equivariant, conjugation gives $F(w\xi)=F(\xi)$ for
$w\in W_\HH$. For
\[
 \bar\xi=\frac1{|W_\HH|}\sum_{w\in W_\HH}w\xi
 \in N_\R(\SSS)^{W_\HH}=N_\R(\ZZ_\HH),
\]
convexity gives $F(\bar\xi)\le F(\xi)$. Therefore the infimum over
$N_\R(\ZZ_\HH)$ is at most the infimum over $N_\R(\SSS)$, while
the reverse inequality follows from
$N_\R(\ZZ_\HH)\subset N_\R(\SSS)$. Continuity and density of rational cocharacters give the same
infima over $N_\Q(\SSS)$ and $N_\Q(\ZZ_\HH)$.
\end{proof}

\begin{lemma}[Equivariant strong and entropy approximation]
\label{lem:shared-entropy-approximation}
Let $\HH\subset\G$ be connected and reductive, and let
$K_\HH\subset\HH$ be maximal compact. For every finite-entropy
$\phi\in(\mathcal E^{1,\NA})^{K_\HH}$, there are normal ample
$\HH$-equivariant test configurations whose metrics $\phi_j$ satisfy
\begin{equation}\label{ordinary:entropy-approximation}
 d_1^{\NA}(\phi_j,\phi)\longrightarrow0,
 \qquad
 H^{\NA}(\phi_j)\longrightarrow H^{\NA}(\phi).
\end{equation}
Moreover,
$\cM^{\NA}(\phi_j)\to\cM^{\NA}(\phi)$ and
$J_{\HH}^{\NA}(\phi_j)\to J_{\HH}^{\NA}(\phi)$.
\end{lemma}

\begin{proof}
By \cite[Proposition~6.3]{Li}, choose $\HH$-equivariant model
metrics $\psi_k$ such that
$d_1^{\NA}(\psi_k,\phi)\to0$ and
$H^{\NA}(\psi_k)\to H^{\NA}(\phi)$.

Fix $k$ and write $\psi=\psi_k=P(\phi_{\mathcal L})$ on an
$\HH$-equivariant smooth dominating SNC model. After translating
by a multiple of the central fiber, we may assume that
$\mathcal L$ is big and that its augmented base locus is contained
in the central fiber. This translates the associated
non-Archimedean metric by a constant and leaves
$H^{\NA}$, $\cM^{\NA}$, and $J_{\HH}^{\NA}$ unchanged.

The equivariant special Fujita approximation
\cite[Corollary~A and its proof]{Tru} produces ample models
$p_j^*\mathcal L=\mathcal L_j+E_j$ with $E_j\ge0$,
$(\mathcal L_j^{n+1})\to\operatorname{vol}(\mathcal L)$, with
non-Archimedean entropy converging to $H^{\NA}(\psi)$. Away from
the central fiber, the modifications are isomorphisms and $E_j$
vanishes. Thus the generic polarized fiber remains $(X,L)$.
After reversing the central-fiber translation,
let $\phi_{k,j}$ be the associated metrics. Then
$\phi_{k,j}\le\psi$,
$\ENA(\phi_{k,j})\to\ENA(\psi)$, and the formula for
$d_1^{\NA}$ between ordered metrics gives
\[
 d_1^{\NA}(\phi_{k,j},\psi)
 =\ENA(\psi)-\ENA(\phi_{k,j})\longrightarrow0.
\]

Choose $j(k)$ so that
$d_1^{\NA}(\phi_{k,j(k)},\psi_k)
 +|H^{\NA}(\phi_{k,j(k)})-H^{\NA}(\psi_k)|<k^{-1}$.
The sequence $\phi_k=\phi_{k,j(k)}$ satisfies
\eqref{ordinary:entropy-approximation}. The strong continuity of
$\mathcal R^{\NA}$ \cite[Proposition~2.17]{Li} gives
$\cM^{\NA}(\phi_k)\to\cM^{\NA}(\phi)$, while
\cite[Lemma~6.4]{Li} gives
$J_{\HH}^{\NA}(\phi_k)\to J_{\HH}^{\NA}(\phi)$.
\end{proof}

\begin{theorem}[Equivariant testing with maximal-torus twists]
\label{thm:groups}
Assume that the Futaki character vanishes on $\Lie\ZZ_\HH$.
Suppose that, for some fixed finite $\beta>0$ and
$0\le\varepsilon<G_A(\beta)$,
\begin{equation}\label{eq:group-condition}
 \RR_{\beta,\HH}^{\rm max}(\Tcal)
 \ge
 -\varepsilon J_{\HH}^{\NA}(\phi_\Tcal)
\end{equation}
for every normal ample $\HH$-equivariant test configuration
$\Tcal$. Then $c_1(L)$ contains a cscK metric.
\end{theorem}

\begin{proof}
Since $\HH$ is reductive,
$\Lie\HH=\Lie\ZZ_\HH\oplus[\Lie\HH,\Lie\HH]$.
The Futaki character vanishes on the derived algebra, so the
assumption implies that it vanishes on $\Lie\HH$, and therefore
on $\Lie\SSS$.

Fix a normal ample $\HH$-equivariant test configuration $\Tcal$
and set $\phi=\phi_\Tcal$. For
$\xi\in N_\Q(\SSS)$, applying Li's twisting identity
\cite[Lemma~2.24]{Li} with the reductive group $\SSS$ gives
$\cM^{\NA}(\phi_\xi)=\cM^{\NA}(\phi)$.
After normalized base change, $\phi_\xi$ is the metric of a normal
ample test configuration. Hence Theorem~\ref{thm:gap} gives
\[
 \cM^{\NA}(\phi)
 \ge
 m_\beta(\phi_\xi)
 +G_A(\beta)J^{\NA}(\phi_\xi)
 \ge
 m_\beta(\phi_\xi)
 +G_A(\beta)J_{\HH}^{\NA}(\phi),
\]
where the last inequality follows from
Lemma~\ref{lem:weyl}. Taking the supremum over
$\xi\in N_\Q(\SSS)$ yields
\begin{equation}\label{eq:group-margin}
 \cM^{\NA}(\phi)
 \ge
 \RR_{\beta,\HH}^{\rm max}(\Tcal)
 +G_A(\beta)J_{\HH}^{\NA}(\phi)
 \ge
 \delta J_{\HH}^{\NA}(\phi),
 \qquad
 \delta=G_A(\beta)-\varepsilon>0.
\end{equation}

Let $\psi$ be an $\HH$-equivariant model metric with
$H^{\NA}(\psi)<\infty$. By
Lemma~\ref{lem:shared-entropy-approximation}, there are normal ample
$\HH$-equivariant test configurations with metrics $\phi_j$ such that
$\cM^{\NA}(\phi_j)\to\cM^{\NA}(\psi)$ and
$J_{\HH}^{\NA}(\phi_j)\to J_{\HH}^{\NA}(\psi)$.
Passing to the limit in \eqref{eq:group-margin} gives
$\cM^{\NA}(\psi)\ge\delta J_{\HH}^{\NA}(\psi)$.

If $H^{\NA}(\psi)=+\infty$, then
$\cM^{\NA}(\psi)=+\infty$, since $\mathcal R^{\NA}(\psi)$ is finite
on $\mathcal E^{1,\NA}$ by \cite[Proposition~2.17]{Li}. Thus
$\cM^{\NA}(\psi)\ge\delta J_{\HH}^{\NA}(\psi)$ holds for every
model metric in the equivariant finite-energy space.
The criterion for model metrics
\cite[Definition~2.25 and Theorem~6.5]{Li}, applied with the
reductive group $\HH$, gives a cscK metric in $c_1(L)$.
\end{proof}

\section{Geodesic-ray proofs and reduced parameter bounds}
\label{sec:center-direct}

We prove the central finite-parameter correspondence and
Theorem~\ref{main:levi}, then estimate the change in the reduced
quantized slopes as the parameter increases.

\subsection{The central criterion and its ray proof}

Throughout this section, $\G=\Aut_0(X,L)$ is the connected
polarized automorphism group defined in the Introduction.
When $\G$ is reductive, fix a maximal compact subgroup
$K\subset\G$, a $K$-invariant reference form, and set
$\mathbb T=Z(\G)^\circ$. Then
$J_{\mathbb T}^{\NA}(\phi)
=\inf_{\xi\in N_\R(\mathbb T)}J^{\NA}(\phi_\xi)$
is \eqref{eq:central-norm} with $\HH=\G$.

Let $\mathcal R_K$ denote the $K$-invariant finite-energy geodesic
rays $u=\{u_t\}_{t\ge0}$ from $0$, normalized by
$\AM(u_t)=0$. For $u,v\in\mathcal R_K$, set
\[
 \rho(u,v)=\lim_{t\to\infty}\frac{d_1(u_t,v_t)}{t},
 \qquad
 R(u)=\rho(u,0),
 \qquad
 s_\beta(u)=(\cM^\beta)^{\rad}(u).
\]
Thus
$s_\beta(u^\Tcal)=m_\beta(\phi_\Tcal)$
for a test configuration $\Tcal$.

For $\xi\in N_\R(\mathbb T)$, let $\sigma_\xi(t)$ denote the
corresponding real one-parameter subgroup and define
\[
 (\tau_\xi u)_t=\sigma_\xi(t)\cdot u_t,
\]
where the action is normalized to preserve $\AM=0$.
We choose the time normalization so that rational $\xi$ give the
corresponding normalized algebraic product twists. Define
\[
 \RR_\beta^{\rm cen}(u)
 =\sup_{\xi\in N_\Q(\mathbb T)}
   s_\beta(\tau_\xi u),
 \qquad
 N_{\mathbb T}(u)
 =\inf_{\xi\in N_\R(\mathbb T)}
   J^{\rad}(\tau_\xi u).
\]
For a test configuration $\Tcal$, we also write
$\RR_\beta^{\rm cen}(\Tcal)=\RR_\beta^{\rm cen}(u^\Tcal)$.

If $H\subset\G$ is connected and the Futaki character vanishes on
$\Lie H$, then
\begin{equation}\label{eq:finite-energy-futaki}
 \cM(g\cdot v)=\cM(v),
 \qquad g\in H,\quad v\in\mathcal E^1.
\end{equation}
For smooth potentials this is the Mabuchi--Futaki identity
\cite{mabuchi86}. It extends to $\mathcal E^1$ because the normalized
action is a $d_1$-isometry \cite[Lemma~5.9]{dr} and $\cM$ is the
greatest $d_1$-lower semicontinuous extension of its restriction to
smooth potentials \cite[Theorem~1.2]{bdl}: applying this property
to $v\mapsto\cM(g\cdot v)$ and then to $g^{-1}$ gives equality.

\begin{theorem}[The central finite-parameter correspondence]
\label{thm:center-direct}
The following are equivalent:
\begin{enumerate}[label=\textup{(\roman*)}]
\item $c_1(L)$ contains a cscK metric;
\item $\G$ is reductive, the Futaki character vanishes on
$\Lie\G$, and there exists a fixed finite $\beta>0$ such that
\[
 \RR_\beta^{\rm cen}(\Tcal)\ge0
\]
for every normal ample $\G$-equivariant test configuration $\Tcal$;
\item $\G$ is reductive, the Futaki character vanishes on
$\Lie\G$, and there exist $\beta_0>0$ and $\gamma>0$ such that
\[
 \RR_\beta^{\rm cen}(\Tcal)
 \ge
 \gamma J_{\mathbb T}^{\NA}(\phi_\Tcal)
\]
for every $\beta\ge\beta_0$ and every normal ample
$\G$-equivariant test configuration $\Tcal$.
\end{enumerate}
\end{theorem}

\begin{corollary}[Existence from a finite-parameter lower bound]
\label{cor:center-negative}
Suppose that $\G$ is reductive and its Futaki character vanishes.
Let $\beta>0$ and $0\le\varepsilon<G_A(\beta)$. If
\begin{equation}\label{eq:center-tolerance}
 \RR_\beta^{\rm cen}(\Tcal)
 \ge
 -\varepsilon J_{\mathbb T}^{\NA}(\phi_\Tcal)
\end{equation}
for every normal ample $\G$-equivariant test configuration
$\Tcal$, then $c_1(L)$ contains a cscK metric.
\end{corollary}

\subsubsection*{Continuity of quantized slopes and the radial gap}

Fix $\beta>0$ and set $a=2^{-n}$. The estimates in
\cite[Proposition~5.1, Theorem~5.2 and Lemma~5.3]{DZ} give
\begin{equation}\label{eq:center-holder}
 |s_\beta(u)-s_\beta(v)|
 \le
 C_\beta\bigl(
 \rho(u,v)^a(R(u)+R(v))^{1-a}+\rho(u,v)
 \bigr).
\end{equation}
Apply \cite[Proposition~5.1, Theorem~5.2 and Lemma~5.3]{DZ}
at time $t$ to $u_t-b_t$ and $v_t-b_t$, where
$b_t=\max\{\sup_Xu_t,\sup_Xv_t\}$.
Since $\AM(u_t)=\AM(v_t)=0$, \eqref{eq:sup-mean} and
\eqref{eq:normalized-distance-J} give
$0\le b_t\le Ct(R(u)+R(v))+C$.
Subtracting $b_t$ from both potentials changes neither their
$d_1$-distance nor the difference of their quantized Mabuchi energies. Dividing the
finite-time estimate by $t$ and using
$|\liminf A_t-\liminf B_t|\le\limsup|A_t-B_t|$
gives \eqref{eq:center-holder}.

The $d_1$-continuity estimates for $J$ give continuity of
$J^{\rad}$ in $\rho$; see \cite[proof of Theorem~1.1]{DZ}.
Dividing \eqref{eq:normalized-distance-J} by the ray parameter
gives constants $c_J,C_J>0$ such that
\begin{equation}\label{eq:center-J-distance}
 c_JR(u)\le J^{\rad}(u)\le C_JR(u).
\end{equation}

A ray $u\in\mathcal R_K$ with finite Mabuchi slope is maximal by
\cite[Theorem~1.2]{Li}. The test-configuration approximation in
\cite[proof of Theorem~10.1, equation~(71), and Remark~10.2]{DZ}
gives rays $u^j$ of normal ample $\G$-equivariant test configurations,
with $\AM(u_t^j)=0$ and $\rho(u^j,u)\to0$. The multiplier ideals are coherent and $K$-invariant, so they are
invariant under the complexification $\G$. Passing to the relative
ample model leaves the associated ray unchanged.

Theorem~\ref{thm:gap} applied to $u\in\mathcal R_K$ with finite
Mabuchi slope gives
\begin{equation}\label{eq:center-ray-gap}
 s_\beta(u)
 \le
 \cM^{\rad}(u)-G_A(\beta)J^{\rad}(u).
\end{equation}

\subsubsection*{Central twists and compactness of optimizers}

Let $T$ be an effective polarized algebraic torus with maximal
compact subgroup $K_T$, and choose the reference form to be
$K_T$-invariant. We use the same time normalization for real
cocharacters and rational algebraic product twists.

\begin{lemma}[Torus-orbit distance]
\label{lem:torus-orbit-distance}
Let $v\in\mathcal E^1_0$ be $K_T$-invariant. For every
$\xi\in N_\R(T)$, the energy-normalized orbit
$t\mapsto\sigma_\xi(t)\cdot v$ is a constant-speed
$d_1$-geodesic. There is a norm $\|\cdot\|_*$ on $N_\R(T)$,
independent of $v$, such that
\begin{equation}\label{eq:torus-orbit-distance}
 d_1\bigl(\sigma_\xi(t)\cdot v,\sigma_\zeta(t)\cdot v\bigr)
 =t\|\xi-\zeta\|_*,
 \qquad t\ge0.
\end{equation}
For smooth $v$,
$\|\xi\|_*=\int_X|h_{\xi,v}|\,d\mu_v$, where $h_{\xi,v}$ is the
mean-zero Hamiltonian generating the normalized orbit.
\end{lemma}

\begin{proof}
Assume first that $v$ is smooth. The complexified
one-parameter orbit generated by $\xi$ is a smooth geodesic.
After energy normalization its velocity is the pullback of the
mean-zero Hamiltonian $h_{\xi,v}$. The $d_1$ path-length formula
\cite{darvas} therefore gives
$d_1(v,\sigma_\xi(t)\cdot v)
=t\int_X|h_{\xi,v}|\,d\mu_v$.

The integral $\int_X|h_{\xi,v}|\,d\mu_v$ does not depend on the
$K_T$-invariant potential $v$. Indeed, join two invariant
K\"ahler forms by an invariant path $\omega_s$. The equivariant
Moser flow $f_s$ satisfies $f_s^*\omega_s=\omega_0$, and the
mean-zero Hamiltonians satisfy
$h_{\xi,s}\circ f_s=h_{\xi,0}$. Hence their distributions and
their $L^1$ norms agree. Linearity of
$\xi\mapsto h_{\xi,v}$ gives homogeneity and the triangle
inequality, while effectivity of the $T$-action gives
$\|\xi\|_*=0$ only for $\xi=0$.

For $K_T$-invariant $v\in\mathcal E^1_0$, choose smooth strictly
$\omega_0$-psh potentials decreasing to $v$. Average them over
$K_T$ and subtract their Aubin--Yau energies. Denote the
normalized potentials by $v_j\in\mathcal E^1_0$. They are
$K_T$-invariant and satisfy
$d_1(v_j,v)\to0$ by monotone continuity of $\AM$ and
\eqref{eq:ordered-d1}. Since the $T$-action is by
$d_1$-isometries, the smooth orbit-distance formula passes to the
limit and shows that $t\mapsto\sigma_\xi(t)\cdot v$ has constant
$d_1$-speed.

Finally, applying the isometry $\sigma_{-\zeta}(t)$ and using
commutativity of $T$ gives \eqref{eq:torus-orbit-distance}.
\end{proof}

Set $T=\mathbb T$. Assume that $\G$ is reductive
and that its Futaki character vanishes, as in conditions
\textup{(ii)} and \textup{(iii)} of
Theorem~\ref{thm:center-direct}. Set $p_\xi=\tau_\xi0$.
The complexified torus action preserves the homogeneous
Monge--Amp\`ere equation and hence preserves invariant weak
geodesic rays. Since $\mathbb T$ is central, its twists also
preserve $K$-invariance. Thus, after energy normalization,
$\tau_\xi u\in\mathcal R_K$ for every $\xi\in N_\R(\mathbb T)$.
The $d_1$-isometric action and
Lemma~\ref{lem:torus-orbit-distance} give
\begin{equation}\label{eq:center-isometry}
 \rho(\tau_\xi u,\tau_\xi v)=\rho(u,v),\qquad
 \rho(\tau_\xi u,\tau_\zeta u)=\|\xi-\zeta\|_*.
\end{equation}
In particular, $\rho(p_\xi,0)=\|\xi\|_*$.

Since the Futaki character vanishes on $\Lie\mathbb T$,
$\cM^{\rad}(p_\xi)=0$. By
\eqref{eq:center-ray-gap} and \eqref{eq:center-J-distance},
\[
 s_\beta(p_\xi)
 \le-G_A(\beta)J^{\rad}(p_\xi)
 \le-k_\beta\|\xi\|_*,
 \qquad k_\beta:=c_JG_A(\beta)>0.
\]
Let $r=\|\xi\|_*$ and $R=R(u)$. By
\eqref{eq:center-isometry},
$\rho(\tau_\xi u,p_\xi)=R$ and
$R(\tau_\xi u)\le r+R$. Hence
\eqref{eq:center-holder}, after enlarging $C_\beta$, gives
\begin{equation}\label{eq:center-escape}
 s_\beta(\tau_\xi u)
 \le-k_\beta r
 +C_\beta R^a(r+R)^{1-a}
 +C_\beta R.
\end{equation}
Applying \eqref{eq:center-holder} to $u$ and the zero ray also gives
$s_\beta(u)\ge-C_\beta R$.

For $R>0$, put $s=r/R$. The right-hand side of
\eqref{eq:center-escape}, divided by $R$, is
$-k_\beta s+C_\beta(s+1)^{1-a}+C_\beta$.
Since $a<1$, it tends to $-\infty$ as $s\to\infty$.
Thus, for every $R_0>0$, there is $L_{\beta,R_0}>0$ such that,
whenever $R(u)\le R_0$ and $\|\xi\|_*>L_{\beta,R_0}$, one has
$s_\beta(\tau_\xi u)<s_\beta(u)$.
Hence the supremum over $N_\Q(\mathbb T)$ may be restricted to
the ball $\|\xi\|_*\le L_{\beta,R_0}$. If $R(u)=0$, then $u$ is
the zero ray. Continuity in $\xi$ and density of rational
cocharacters therefore give
\[
 \RR_\beta^{\rm cen}(u)
 =
 \max_{\xi\in N_\R(\mathbb T)}
 s_\beta(\tau_\xi u).
\]
For $R(u),R(v)\le R_0$, \eqref{eq:center-holder} and
\eqref{eq:center-isometry} give
\[
 |\RR_\beta^{\rm cen}(u)-\RR_\beta^{\rm cen}(v)|
 \le C_{\beta,R_0}
 \bigl(\rho(u,v)^a+\rho(u,v)\bigr).
\]

We next prove continuity of $N_{\mathbb T}$.
Indeed, if $r=\|\xi\|_*$, then
$R(\tau_\xi u)\ge r-R(u)$, and hence
\[
 J^{\rad}(\tau_\xi u)\ge c_J(r-R(u)),
 \qquad
 N_{\mathbb T}(u)\le J^{\rad}(u)\le C_JR(u).
\]
Thus, for $R(u)\le R_0$, the infimum defining
$N_{\mathbb T}(u)$ may be restricted to a common compact ball.
The continuity of $J^{\rad}$ and \eqref{eq:center-isometry}
then gives continuity of $N_{\mathbb T}$, and rational
cocharacters give the same infimum. For a test configuration $\Tcal$, apply
\eqref{eq:slope-bridge} at each fixed rational twist. Taking
infima gives
\begin{equation}\label{eq:center-norm-bridge}
 N_{\mathbb T}(u^\Tcal)
 =J_{\mathbb T}^{\NA}(\phi_\Tcal).
\end{equation}

\begin{proof}[Sufficiency and Corollary~\ref{cor:center-negative}]
Suppose that the Mabuchi functional is not coercive modulo
$\mathbb T$. The construction in
\cite[proof of Proposition~6.2, equations~(165) to~(177)]{Li}
produces a $K$-invariant ray $u$ of finite Mabuchi slope such that
\[
 \cM^{\rad}(u)\le0,\qquad
 J((\tau_\xi u)_t)\ge c_*t-C
\]
for every $t>0$ and $\xi\in N_\R(\mathbb T)$, where $c_*>0$ and
$C$ are independent of $t$ and $\xi$. After normalization by
$\AM(u_t)=0$, the radial $J$ and Mabuchi slopes are unchanged,
and hence $N_{\mathbb T}(u)\ge c_*>0$.

By \cite[proof of Theorem~10.1, equation~(71), and Remark~10.2]{DZ},
choose normal ample $\G$-equivariant configuration rays $u^j$ such that
$\rho(u^j,u)\to0$. By \eqref{eq:center-norm-bridge} and
\eqref{eq:center-tolerance},
\[
 \RR_\beta^{\rm cen}(u^j)
 \ge-\varepsilon N_{\mathbb T}(u^j).
\]
The continuity of $\RR_\beta^{\rm cen}$ and $N_{\mathbb T}$
therefore gives
\[
 \RR_\beta^{\rm cen}(u)
 \ge-\varepsilon N_{\mathbb T}(u).
\]

For every central twist,
Equation~\eqref{eq:finite-energy-futaki} gives
$\cM^{\rad}(\tau_\xi u)=\cM^{\rad}(u)$.
Applying \eqref{eq:center-ray-gap} to $\tau_\xi u$ and taking the
supremum over rational $\xi$ yields
\[
 \RR_\beta^{\rm cen}(u)
 \le
 \cM^{\rad}(u)-G_A(\beta)N_{\mathbb T}(u).
\]
Therefore
\[
 \cM^{\rad}(u)
 \ge
 \bigl(G_A(\beta)-\varepsilon\bigr)N_{\mathbb T}(u)>0,
\]
contradicting $\cM^{\rad}(u)\le0$.

Thus the Mabuchi functional is coercive on $\cH^K$ modulo
$\mathbb T$. By
\cite[proof of Proposition~6.2, equation~(164)]{Li},
this coercivity on invariant potentials implies the existence
of a cscK metric. This proves
Corollary~\ref{cor:center-negative}; taking $\varepsilon=0$
proves that condition \textup{(ii)} implies condition
\textup{(i)} in Theorem~\ref{thm:center-direct}.
\end{proof}

\begin{proof}[Necessity in Theorem~\ref{thm:center-direct}]
Assume that $c_1(L)$ contains a cscK metric. Then $\G$ is
reductive and its Futaki character vanishes. Choose a maximal
compact subgroup $K$ preserving a cscK reference.
We use Hisamoto's coercivity theorem \cite[Theorem~4.3]{His}
on $\cH^K$ and the reduced slope formula \cite[Theorem~B]{His}.

Set $\mathcal N(u)=u-\AM(u)$ and define
\[
 D_{\mathbb T}(v)
 =\inf_{\sigma\in\mathbb T}d_1(0,\sigma\cdot v),
 \qquad
 J_{\mathbb T}(u)
 =\inf_{\sigma\in\mathbb T}J(\sigma\cdot u).
\]
There are constants $\delta,c>0$ and $C$ such that
\[
 \cM(u)\ge
 \delta D_{\mathbb T}(\mathcal N(u))-C,
 \qquad
 D_{\mathbb T}(\mathcal N(u))
 \ge cJ_{\mathbb T}(u)-C.
\]
Moreover,
\begin{equation}\label{eq:center-lipschitz}
 \begin{aligned}
 |D_{\mathbb T}(v)-D_{\mathbb T}(w)|&\le d_1(v,w),\\
 d_1(\mathcal N(v),\mathcal N(w))&\le2d_1(v,w).
 \end{aligned}
\end{equation}

Fix $u\in\cH^K$, set $w=u-\sup_Xu$, and let
$w^\beta$ be the optimizer in \eqref{eq:quantized-optimizer}.
Write $L_w=d_1(0,w)$, $h=\Ent^\beta(w)$, and
$e=d_1(w,w^\beta)$. For sufficiently large $\beta$,
\cite[Proposition~4.1 and Corollary~4.5]{DZ} give
\begin{align*}
 e
 &\le
 \beta^{-1}
 \bigl(h+C\log(L_w+2)+C\log\beta+C\bigr),\\
 (1+\epsilon_\beta^{\mathrm{DZ}})\cM^\beta(w)
 &\ge
 \cM(w^\beta)
 -\epsilon_\beta^{\mathrm{DZ}}L_w
 -\epsilon_\beta^{\mathrm{DZ}},
 \qquad
 \epsilon_\beta^{\mathrm{DZ}}\longrightarrow0.
\end{align*}

Since $w$ and the reference data are $K$-invariant, uniqueness in
\eqref{eq:quantized-optimizer} implies that $w^\beta$ is
$K$-invariant. Since $\mathcal N(w)=\mathcal N(u)$,
\eqref{eq:center-lipschitz} gives
\[
 D_{\mathbb T}(\mathcal N(w^\beta))
 \ge
 D_{\mathbb T}(\mathcal N(u))-2e.
\]
Thus, with $D=D_{\mathbb T}(\mathcal N(u))$,
$\cM(w^\beta)\ge\delta D-2\delta e-C$.

Set $Q=\cM^\beta(u)=\cM^\beta(w)$. We have
$h=Q+\cJ_{\Ric(\omega_0)}(w)$, and
\eqref{eq:ricci-energy-bound} gives
$|\cJ_{\Ric(\omega_0)}(w)|\le C(L_w+1)$.
Also, $\log(L_w+2)\le L_w+\log2$. Substitute the bound for $e$ in
$\cM(w^\beta)\ge\delta D-2\delta e-C$ and use the lower bound
for $(1+\epsilon_\beta^{\mathrm{DZ}})\cM^\beta(w)$. We obtain
\[
 (1+\epsilon_\beta^{\mathrm{DZ}})Q
 \ge
 \delta D-\frac{2\delta}{\beta}Q
 -\left(
   \epsilon_\beta^{\mathrm{DZ}}+\frac{C\delta}{\beta}
  \right)L_w-C_\beta,
\]
where $C_\beta$ is independent of $u$.
Since $\sup_Xw=0$, \eqref{eq:sup-mean} gives
$L_w=-\AM(w)\le J(u)+C$.

Increasing $C$ if necessary, set
$A_\beta=1+\epsilon_\beta^{\mathrm{DZ}}+2\delta/\beta$ and
$b_\beta=\epsilon_\beta^{\mathrm{DZ}}+C\delta/\beta$.
Then
\begin{equation}\label{eq:center-necessity-estimate}
 A_\beta\cM^\beta(u)
 \ge
 \delta D_{\mathbb T}(\mathcal N(u))
 -b_\beta J(u)-C_\beta,
 \qquad
 A_\beta\longrightarrow1,\quad
 b_\beta\longrightarrow0.
\end{equation}
For each fixed parameter, the continuity of $\cM^\beta$ and $J$
and the Lipschitz continuity of
$D_{\mathbb T}\circ\mathcal N$ extend
\eqref{eq:center-necessity-estimate} to finite-energy
$K$-invariant potentials.

Fix a normal ample $\G$-equivariant test configuration $\Tcal$
and set $N=J_{\mathbb T}^{\NA}(\phi_\Tcal)$. For each rational
$\xi\in N_\Q(\mathbb T)$, apply
\eqref{eq:center-necessity-estimate} to the ray
$\tau_\xi u^\Tcal$. Hisamoto's reduced slope formula gives
\[
 \lim_{t\to\infty}
 \frac{J_{\mathbb T}((\tau_\xi u^\Tcal)_t)}{t}
 =N.
\]
Using
$D_{\mathbb T}(\mathcal N(v))\ge cJ_{\mathbb T}(v)-C$
and taking radial slopes gives
\[
 A_\beta s_\beta(\tau_\xi u^\Tcal)
 \ge
 \delta cN-b_\beta J^{\rad}(\tau_\xi u^\Tcal).
\]
Taking the supremum over rational $\xi$ and using
\eqref{eq:center-norm-bridge} gives
\[
 A_\beta\RR_\beta^{\rm cen}(\Tcal)
 \ge
 (\delta c-b_\beta)N.
\]

Choose $\beta_0$ so large that
$A_\beta\le2$ and $b_\beta\le\delta c/2$ whenever
$\beta\ge\beta_0$. Then
$\RR_\beta^{\rm cen}(\Tcal)\ge(\delta c/4)N$
for every $\beta\ge\beta_0$, so condition \textup{(iii)} holds
with $\gamma=\delta c/4$.
Since $J_{\mathbb T}^{\NA}\ge0$, condition \textup{(iii)}
implies condition \textup{(ii)} by taking $\beta=\beta_0$.
\end{proof}

\begin{proof}[Completion of Theorem~\ref{main:levi} and
Corollary~\ref{intro:quantitative-correspondences}\textup{(ii)}]
Take $\HH$ to be the maximal connected reductive subgroup fixed in Theorem~\ref{main:levi};
Lemma~\ref{lem:weyl} identifies $N_\HH(\Tcal)$ with
$J_\HH^{\NA}(\phi_\Tcal)$.
Theorem~\ref{thm:groups}, with $\varepsilon=0$ or
$\varepsilon=\vartheta G_A(\beta)<G_A(\beta)$, proves sufficiency
in Theorem~\ref{main:levi} and
Corollary~\ref{intro:quantitative-correspondences}\textup{(ii)},
respectively.
Conversely, if a cscK metric exists, then $\G$ is reductive and its
Futaki character vanishes, so its maximal connected reductive subgroup $\HH$ equals $\G$.
By Theorem~\ref{thm:center-direct}, there is a fixed finite
$\beta$ at which the central optimized slopes are nonnegative for every
normal ample $\G$-equivariant test configuration.
Since central twists are among the maximal-torus twists,
$\RR_{\beta,\HH}^{\rm max}\ge\RR_\beta^{\rm cen}\ge0$, proving
necessity in both statements.
Finally, Proposition~\ref{prop:reduced-parameter-bounds}, with
$\varepsilon=\vartheta G_A(\beta)$ and
$\Gamma_A(\beta,B)=G_A(\beta)-G_A(B)$, gives the stated lower bound for
every $B>\beta$.
\end{proof}

\subsection{Rays with nonpositive Mabuchi slope and product twists}

Let $\HH\subset\G$ be connected and reductive, with maximal torus
$\SSS$ and central torus $\ZZ_\HH$, and assume that the Futaki
character vanishes on $\Lie\ZZ_\HH$. Choose a maximal compact
subgroup $K_\HH\subset\HH$ such that
$K_\SSS=K_\HH\cap\SSS$ is maximal compact in $\SSS$, and fix a
$K_\HH$-invariant reference form. On the energy-normalized ray
space $\mathcal R_{K_\SSS}$, define
\[
 \RR_\beta^\SSS(u)
 =\sup_{\xi\in N_\Q(\SSS)}
   s_\beta(\tau_\xi u),
 \qquad
 N_\SSS(u)
 =\inf_{\xi\in N_\R(\SSS)}
   J^{\rad}(\tau_\xi u).
\]
For $u\in\mathcal R_{K_\HH}$, define
$N_{\ZZ_\HH}(u)
=\inf_{\xi\in N_\R(\ZZ_\HH)}J^{\rad}(\tau_\xi u)$.

\begin{proposition}[A ray with nonpositive Mabuchi slope]
\label{prop:subgroup-direct-ray}
For every $\beta>0$, $\RR_\beta^\SSS$ is locally H\"older
continuous in $\rho$. The functions $N_\SSS$ on
$\mathcal R_{K_\SSS}$ and $N_{\ZZ_\HH}$ on
$\mathcal R_{K_\HH}$ are continuous, and the infima defining them
are unchanged if the real cocharacters are replaced by rational
cocharacters.

If $c_1(L)$ contains no cscK metric, there exists
$u\in\mathcal R_{K_\HH}$ with finite Mabuchi slope such that
\[
 \cM^{\rad}(u)\le0,
 \qquad
 N_\SSS(u)=N_{\ZZ_\HH}(u)>0.
\]
Moreover, there are normal ample $\HH$-equivariant test
configurations $\Tcal_j$, with metrics
$\phi_j=\phi_{\Tcal_j}$ and associated energy-normalized rays
$u^j\in\mathcal R_{K_\HH}$, such that $\rho(u^j,u)\to0$ and
\[
 J_\HH^{\NA}(\phi_j)\longrightarrow N_\SSS(u),
 \qquad
 \RR_{\beta,\HH}^{\rm max}(\Tcal_j)
 \longrightarrow\RR_\beta^\SSS(u)
\]
for every fixed $\beta>0$.
\end{proposition}

\begin{proof}
Since the Futaki character vanishes on the derived algebra of
$\HH$, its vanishing on $\Lie\ZZ_\HH$ implies that it vanishes on
$\Lie\HH$, and hence on $\Lie\SSS$.

The estimates \eqref{eq:center-holder} and
\eqref{eq:center-J-distance} apply on
$\mathcal R_{K_\SSS}$ because their proofs use only the
energy normalization. Lemma~\ref{lem:torus-orbit-distance} gives
\[
 \rho(\tau_\xi u,\tau_\xi v)=\rho(u,v),
 \qquad
 \rho(\tau_\xi u,\tau_\zeta u)=\|\xi-\zeta\|_*.
\]
The $\SSS$-twists preserve $\mathcal R_{K_\SSS}$, although they
need not preserve $\mathcal R_{K_\HH}$.

Fix $\beta>0$ and set $\eta=2^{-n}$. For
$p_\xi=\tau_\xi0$, Equation~\eqref{eq:finite-energy-futaki} gives
$\cM^{\rad}(p_\xi)=0$. If
$r=\|\xi\|_*$, $R=R(u)$, and
$k_\beta=c_JG_A(\beta)>0$, then
\eqref{eq:center-ray-gap} and \eqref{eq:center-holder} give
\[
 s_\beta(\tau_\xi u)
 \le
 -k_\beta r
 +C_\beta R^\eta(r+R)^{1-\eta}
 +C_\beta R.
\]
The zero twist satisfies
$s_\beta(u)\ge-C_\beta R$.
Since $\eta<1$, there is $L_\beta>0$, independent of $u$, such
that for $R>0$ the supremum defining $\RR_\beta^\SSS(u)$ may be
restricted to $r\le L_\beta R$. If $R=0$, then $u$ is the zero
ray.

For the reduced norm,
$J^{\rad}(\tau_\xi u)\ge c_J(r-R)$, whereas
$J^{\rad}(u)\le C_JR$. Hence the infimum defining
$N_\SSS(u)$ may also be restricted to $r\le L_JR$ for a constant
$L_J$ independent of $u$. Applying these estimates with
$T=\ZZ_\HH$ also restricts the infimum defining
$N_{\ZZ_\HH}$ on $\mathcal R_{K_\HH}$ to a common compact ball.

For rays with bounded $R$, the supremum defining $\RR_\beta^\SSS$
and the infima defining $N_\SSS$ and $N_{\ZZ_\HH}$ can be taken
over fixed compact balls. Equation~\eqref{eq:center-holder} and
the torus isometries give a H\"older estimate for
$u\mapsto s_\beta(\tau_\xi u)$, uniformly for cocharacters
in these balls. Taking suprema proves the local H\"older continuity
of $\RR_\beta^\SSS$. The continuity of $J^{\rad}$ in $\rho$ and
the torus isometries similarly give joint continuity of
$J^{\rad}(\tau_\xi u)$ on the compact parameter balls. Taking
infima proves continuity of $N_\SSS$ and $N_{\ZZ_\HH}$.
Continuity in $\xi$ and density of rational cocharacters show that
the same infima are obtained over rational cocharacters.

Assume that $c_1(L)$ contains no cscK metric. By
\cite[proof of Proposition~6.2, equation~(164)]{Li},
the Mabuchi functional is not coercive modulo the center on
$\cH^{K_\HH}$.
The construction in
\cite[Section~2.1.3 and proof of Proposition~6.2,
equations~(165) to~(177)]{Li}, followed by energy normalization,
produces a ray $u\in\mathcal R_{K_\HH}$ with finite Mabuchi slope
such that
\[
 \cM^{\rad}(u)\le0,
 \qquad
 N_{\ZZ_\HH}(u)\ge c_*>0.
\]
By \cite[Theorem~1.2]{Li}, the ray $u$ is maximal.

Apply the test-configuration approximation in
\cite[proof of Theorem~10.1, equation~(71)]{DZ}. Its multiplier ideals are
$K_\HH$-invariant. Since they are coherent and
$\HH=K_\HH^\C$, the complexification argument in
\cite[Remark~10.2]{DZ} shows that they are $\HH$-invariant.
The full group $\G$ need not be reductive.
Replacing each semiample model by its relative ample model
preserves the associated ray. After energy normalization, we
obtain normal ample $\HH$-equivariant test configurations
$\Tcal_j$, with metrics $\phi_j$ and rays
$u^j\in\mathcal R_{K_\HH}$, such that
$\rho(u^j,u)\to0$.

For every $j$, Lemma~\ref{lem:weyl},
\eqref{eq:slope-bridge}, and the density of rational cocharacters give
\[
 N_\SSS(u^j)
 =N_{\ZZ_\HH}(u^j)
 =J_\HH^{\NA}(\phi_j).
\]
By continuity of $N_\SSS$ and $N_{\ZZ_\HH}$,
\[
 N_\SSS(u)
 =\lim_{j\to\infty}N_\SSS(u^j)
 =\lim_{j\to\infty}N_{\ZZ_\HH}(u^j)
 =N_{\ZZ_\HH}(u)\ge c_*>0.
\]
For every fixed $\beta>0$, the local H\"older continuity of
$\RR_\beta^\SSS$ and the identity
$\RR_{\beta,\HH}^{\rm max}(\Tcal_j)=\RR_\beta^\SSS(u^j)$ also give
\[
 \RR_{\beta,\HH}^{\rm max}(\Tcal_j)
 \longrightarrow
 \RR_\beta^\SSS(u).
\]
\end{proof}

\subsection{Quantitative growth of optimized slopes}

We compare the maximal-torus suprema at two parameters.
The radial gap bounds their difference from below; for the upper bound
we show that the real maximizing twists remain in a fixed ball as the
parameter varies over a compact interval.

\begin{proposition}[Reduced parameter bounds]
\label{prop:reduced-parameter-bounds}
Let $\HH\subset\G$ be connected and reductive, and suppose that
the Futaki character vanishes on $\Lie\ZZ_\HH$. For an
$\HH$-equivariant normal ample test configuration $\Tcal$, set
$N(\Tcal)=J_\HH^{\NA}(\phi_\Tcal)$.
For every compact interval $I=[a,b]\subset(0,\infty)$, there is
$C_I<\infty$, depending only on $(X,L)$, the fixed reference form,
$\HH$, and $I$, and independent of $\Tcal$, its exponent, and the
twist, such that
\begin{equation}\label{eq:reduced-size-bound}
 |\RR_{\beta,\HH}^{\rm max}(\Tcal)|\le C_I N(\Tcal),
 \qquad \beta\in I.
\end{equation}
For $\beta,B\in I$ with $\beta<B$,
\begin{equation}\label{eq:reduced-parameter-bounds}
 \Gamma_A(\beta,B)N(\Tcal)
 \le\RR_{B,\HH}^{\rm max}(\Tcal)
     -\RR_{\beta,\HH}^{\rm max}(\Tcal)
 \le C_I(B-\beta)N(\Tcal).
\end{equation}
In particular, if $N(\Tcal)=0$, then
$\RR_{\beta,\HH}^{\rm max}(\Tcal)=0$ for every $\beta>0$.

If, for some fixed $\beta>0$ and $\varepsilon\ge0$, every
configuration satisfies
$\RR_{\beta,\HH}^{\rm max}\ge-\varepsilon N$, then
$\RR_{B,\HH}^{\rm max}\ge
(\Gamma_A(\beta,B)-\varepsilon)N$ for every $B>\beta$.
Thus $\varepsilon=0$ gives a uniform positive lower bound in terms
of the reduced norm at every $B>\beta$. If $0\le\varepsilon<G_A(\beta)$,
such a lower bound holds whenever $\Gamma_A(\beta,B)>\varepsilon$.
\end{proposition}

\begin{proof}
On the $K_\SSS$-invariant ray space of
Proposition~\ref{prop:subgroup-direct-ray}, put
$Q_\beta(u)=\sup_{\xi\in N_\Q(\SSS)}
s_\beta(\tau_\xi u)$.
For the ray $u=u^\Tcal$, one has
$Q_\beta(u)=\RR_{\beta,\HH}^{\rm max}(\Tcal)$.
By Proposition~\ref{prop:subgroup-direct-ray} and
Lemma~\ref{lem:weyl},
\[
 N(\Tcal)
 =
 \inf_{\zeta\in N_\Q(\ZZ_\HH)}
 J^{\rad}(\tau_\zeta u).
\]
A rational central twist gives another $\HH$-equivariant test
configuration and leaves $Q_\beta$ unchanged, since it translates
$N_\Q(\SSS)$.

We first find a common ball containing the real maximizing twists for
$\beta\in I$. Put $R=R(u)$, $r=\|\xi\|_*$, and
$\eta=2^{-n}$. Since $\beta\le b$, parameter monotonicity,
\eqref{eq:center-ray-gap}, \eqref{eq:center-J-distance}, and
\eqref{eq:center-holder} give
\[
 s_\beta(\tau_\xi u)
 \le
 -k_b r+C_I R^\eta(r+R)^{1-\eta}+C_I R,
 \qquad k_b=c_JG_A(b)>0.
\]
Since $\beta\ge a$, parameter monotonicity and
\eqref{eq:center-holder}, applied to $u$ and the zero ray, give
$s_\beta(u)\ge-C_I R$.
Hence there is $L_I>0$, independent of $u$ and $\beta\in I$,
such that, whenever $R>0$ and $r>L_IR$,
$s_\beta(\tau_\xi u)<s_\beta(u)$.
Thus the supremum defining $Q_\beta(u)$ may be restricted to
$r\le L_IR$. If $R=0$, then $u$ is the zero ray and
$Q_\beta(u)=0$. Continuity in $\xi$ and density of rational
cocharacters show that the supremum is attained over
$N_\R(\SSS)$.

For $\gamma=a,b$, \eqref{eq:center-holder} gives
$|s_\gamma(v)|\le C_I R(v)$, and parameter monotonicity
gives the same bound for every $\gamma\in I$.
If $\xi_\beta$ is a real maximizer of $Q_\beta(u)$, then
$R(\tau_{\xi_\beta}u)\le(1+L_I)R(u)$. Consequently,
\begin{equation}\label{eq:reduced-optimizer-size}
 |Q_\beta(u)|\le C_I R(u),
 \qquad
 J^{\rad}(\tau_{\xi_\beta}u)\le C_I R(u),
 \qquad \beta\in I.
\end{equation}

Suppose first that $N:=N(\Tcal)>0$. Choose
$\zeta\in N_\Q(\ZZ_\HH)$ such that
$J^{\rad}(\tau_\zeta u)\le2N$, and replace $u$ by
$\tau_\zeta u$. This does not change $Q_\beta$.
By \eqref{eq:center-J-distance}, the twisted ray satisfies
$R(u)\le2N/c_J$. Hence \eqref{eq:reduced-optimizer-size} gives
\eqref{eq:reduced-size-bound} and
$J^{\rad}(\tau_{\xi_\gamma}u)\le C_IN$ for every real maximizer
$\xi_\gamma$ with $\gamma\in I$.

If $N=0$, choose rational central twists $\zeta_j$ such that
$J^{\rad}(\tau_{\zeta_j}u)\to0$. Then
$R(\tau_{\zeta_j}u)\to0$ by
\eqref{eq:center-J-distance}. Since $Q_\beta$ is invariant under
central twists, \eqref{eq:reduced-optimizer-size} gives
$Q_\beta(u)=0$.

Assume that $N>0$. For every
$\xi\in N_\Q(\SSS)$, Theorem~\ref{thm:gap} and
Lemma~\ref{lem:weyl} give
\[
 s_B(\tau_\xi u)
 \ge s_\beta(\tau_\xi u)
     +\Gamma_A(\beta,B)J^{\rad}(\tau_\xi u)
 \ge s_\beta(\tau_\xi u)+\Gamma_A(\beta,B)N.
\]
Taking suprema gives
$Q_B(u)\ge Q_\beta(u)+\Gamma_A(\beta,B)N$.

For the upper bound, let $\xi_B$ be a real maximizer of
$Q_B(u)$. The upper bound in \eqref{eq:radial-secant-bound}
extends from rational to real twists by
\eqref{eq:center-holder} and Lemma~\ref{lem:torus-orbit-distance}.
It gives
\[
 \begin{aligned}
 Q_B(u)-Q_\beta(u)
 &\le
 s_B(\tau_{\xi_B}u)
 -s_\beta(\tau_{\xi_B}u)\\
 &\le
 (B-\beta)J^{\rad}(\tau_{\xi_B}u)
 \le C_I(B-\beta)N.
 \end{aligned}
\]
This proves \eqref{eq:reduced-parameter-bounds}. Substituting the
assumed bound $\RR_{\beta,\HH}^{\rm max}\ge-\varepsilon N$ into
its lower inequality gives the assertions about negative lower bounds.
\end{proof}

\begin{corollary}[Reduced thresholds]
\label{cor:reduced-threshold}
Under the assumptions of
Proposition~\ref{prop:reduced-parameter-bounds}, suppose that the
$\HH$-equivariant testing family contains a configuration of
positive reduced norm. Define
\[
 \Lambda_{\beta,\HH}^{\rm red}
 =
 \inf_{N(\Tcal)>0}
 \frac{\RR_{\beta,\HH}^{\rm max}(\Tcal)}{N(\Tcal)}.
\]
Then $\beta\mapsto\Lambda_{\beta,\HH}^{\rm red}$ is finite,
locally Lipschitz, and strictly increasing. For $\beta<B$ in a
compact interval $I\subset(0,\infty)$,
\[
 \Gamma_A(\beta,B)
 \le
 \Lambda_{B,\HH}^{\rm red}
 -\Lambda_{\beta,\HH}^{\rm red}
 \le C_I(B-\beta).
\]
Consequently, $\Lambda_{\beta,\HH}^{\rm red}$ has at most one
positive zero. If such a zero $\beta_*$ exists, reduced
semistability holds exactly for $\beta\ge\beta_*$, while uniform
reduced positivity holds exactly for $\beta>\beta_*$.

If $\HH$ is a maximal connected reductive subgroup and the Futaki character vanishes on
$\Lie\ZZ_\HH$, then $c_1(L)$ contains a cscK metric if and only if
$\Lambda_{\beta,\HH}^{\rm red}\ge0$ for some finite $\beta>0$.
\end{corollary}

\begin{proof}
For $N(\Tcal)>0$, set
$f_\beta(\Tcal)
=\RR_{\beta,\HH}^{\rm max}(\Tcal)/N(\Tcal)$.
Equation~\eqref{eq:reduced-size-bound} gives
$|f_\beta(\Tcal)|\le C_I$ on every compact parameter interval,
so $\Lambda_{\beta,\HH}^{\rm red}$ is finite.

For $\beta<B$ in $I$,
\eqref{eq:reduced-parameter-bounds} gives
\[
 f_\beta(\Tcal)+\Gamma_A(\beta,B)
 \le f_B(\Tcal)
 \le f_\beta(\Tcal)+C_I(B-\beta).
\]
Taking infima over the same configurations of positive reduced
norm gives the stated bounds. Since
$\Gamma_A(\beta,B)>0$ for $\beta<B$, the threshold is strictly
increasing; the upper bound gives local Lipschitz continuity.
Strict increase determines the signs on either side of a positive
zero $\beta_*$.

By Proposition~\ref{prop:reduced-parameter-bounds},
the optimized invariant vanishes on configurations with zero
reduced norm. Such configurations therefore satisfy the
semistability inequality. For a maximal connected reductive subgroup, the
last equivalence is Theorem~\ref{main:levi}.
\end{proof}

Corollary~\ref{cor:reduced-threshold} does not assert concavity
or a global $1$-Lipschitz bound for $\Lambda_{\beta,\HH}^{\rm red}$.
If no configuration has positive reduced norm, the threshold is
undefined and every optimized invariant is zero. If a positive zero $\beta_*$ exists, then for
$\beta<\beta_*$ sufficiently close to $\beta_*$ one has
$0<-\Lambda_{\beta,\HH}^{\rm red}
\le C_I(\beta_*-\beta)<G_A(\beta)$
for any compact interval $I\Subset(0,\infty)$ containing
$\beta_*$ in its interior.

\subsection{Finite-parameter bounds on Fano manifolds}
\label{sec:fano-tolerance}

Let $\HH\subset\G$ be connected and reductive, with maximal torus
$\SSS$ and central torus $\ZZ_\HH$. For a normal ample
$\HH$-equivariant test configuration $\Tcal$, write
\[
 N(\Tcal)=J_\HH^{\NA}(\phi_\Tcal),
 \qquad
 Q_\beta(\Tcal)=\RR_{\beta,\HH}^{\rm max}(\Tcal).
\]

\begin{theorem}[Anticanonical tolerance and its endpoint]
\label{thm:fano-tolerance-endpoint}
Let $X$ be a smooth Fano manifold and let $L$ be ample with
$-K_X\sim_\Q\lambda L$ for some $\lambda>0$. Put
$A=\alpha(X,L)$, and assume that the Futaki character vanishes on
$\Lie\ZZ_\HH$.

If, for some fixed finite $\beta_0>0$,
\begin{equation}\label{eq:fano-tolerance-endpoint}
 Q_{\beta_0}(\Tcal)
 \ge
 -G_A(\beta_0)N(\Tcal)
\end{equation}
for every normal ample $\HH$-equivariant test configuration
$\Tcal$, then $c_1(L)$ contains a cscK metric.
If $\HH$ is a maximal connected reductive subgroup, the existence of a cscK metric is
equivalent to the vanishing of the Futaki character on the Lie
algebra of the central torus together with
\eqref{eq:fano-tolerance-endpoint} at some finite parameter.

For $0<A<\lambda$, define
\begin{equation}\label{eq:fano-finite-tolerance}
 F_{A,\lambda}(\beta)
 =
 \begin{cases}
  A\lambda/(\lambda-A)-\beta,
     &0<\beta\le A,\\[2pt]
  A^{n+2}/((\lambda-A)\beta^n),
     &\beta>A.
 \end{cases}
\end{equation}
If, for some fixed finite $\beta>0$ and
$0\le\varepsilon<F_{A,\lambda}(\beta)$,
\begin{equation}\label{eq:fano-stronger-test}
 Q_\beta(\Tcal)\ge-\varepsilon N(\Tcal)
\end{equation}
for every normal ample $\HH$-equivariant test configuration
$\Tcal$, then $c_1(L)$ contains a cscK metric.
Moreover,
$F_{A,\lambda}(\beta)>G_A(\beta)$ for every $\beta>0$ whenever
$\lambda/(n+1)<A<\lambda$.
We do not include the endpoint $\varepsilon=F_{A,\lambda}(\beta)$.
\end{theorem}

\begin{proof}
Fix $0<a<\min\{A,\lambda\}$.
Gibbs' inequality and $\alpha$-integrability give
$\Ent(w)\ge aI(w)-C_a$ for every finite-entropy
$w\in\mathcal E^1$.
Write $\Ric(\omega_0)=\lambda\omega_0+dd^cf$, with $f$ smooth.
Integration by parts in \eqref{eq:Jtheta} gives
\[
 \cJ_{\Ric(\omega_0)}(w)
 =
 \lambda(I(w)-J(w))
 +\int_X f\,(d\mu_w-d\nu_0).
\]
Since $f$ is smooth and both measures have mass one, the integral is
uniformly bounded. Thus
$I(w)=J(w)+\lambda^{-1}\cJ_{\Ric(\omega_0)}(w)+O(1)$.
Substituting into
$\cM(w)\ge aI(w)-\cJ_{\Ric(\omega_0)}(w)-C_a$ gives
\begin{equation}\label{eq:fano-ricci-lower-bound}
 \cJ_{\Ric(\omega_0)}(w)
 \ge
 \frac{a\lambda}{\lambda-a}J(w)
 -\frac{\lambda}{\lambda-a}\cM(w)-C_a'.
\end{equation}

If $\beta\le a$, use $\Ent^\beta\le\beta J$.
If $\beta>a$, Corollary~\ref{cor:rescaled-parameter-convexity},
followed by $B\to\infty$ at fixed $w$, gives
$\Ent^\beta(w)\le
t\,aJ(w)+(1-t)\Ent(w)$, where $t=(a/\beta)^n$.
Subtracting $\cJ_{\Ric(\omega_0)}(w)$ and using
\eqref{eq:fano-ricci-lower-bound}, we obtain
\[
\begin{aligned}
 \cM^\beta(w)
 &\le (1-t)\cM(w)
       +t\bigl(aJ(w)-\cJ_{\Ric(\omega_0)}(w)\bigr)\\
 &\le\left(1+\frac{at}{\lambda-a}\right)\cM(w)
       -\frac{a^2t}{\lambda-a}J(w)+C_{a,\beta}.
\end{aligned}
\]
Set $C_{a,\lambda}(\beta)
=1+\frac{a}{\lambda-a}\min\{1,(a/\beta)^n\}$.
The two parameter ranges give
\begin{equation}\label{eq:fano-pointwise-control}
 \cM^\beta(w)
 \le C_{a,\lambda}(\beta)\cM(w)
    -F_{a,\lambda}(\beta)J(w)+C_{a,\beta},
\end{equation}
where $F_{a,\lambda}$ is defined by
\eqref{eq:fano-finite-tolerance} with $A$ replaced by $a$.

We first prove existence under \eqref{eq:fano-stronger-test}.
Suppose, to the contrary, that no cscK metric exists.
Proposition~\ref{prop:subgroup-direct-ray} gives a
$K_\HH$-invariant ray $u$ with finite Mabuchi slope
$m=\cM^{\rad}(u)\le0$ and
$N:=N_\SSS(u)>0$.
The configurations $\Tcal_j$ in
Proposition~\ref{prop:subgroup-direct-ray} satisfy
$Q_\beta(\Tcal_j)\to\RR_\beta^\SSS(u)$ and $N(\Tcal_j)\to N$.
Hence \eqref{eq:fano-stronger-test} implies
$\RR_\beta^\SSS(u)\ge-\varepsilon N$.

The Futaki character vanishes on $\Lie\ZZ_\HH$ by assumption
and on the derived algebra because it is a character. Thus it
vanishes on $\Lie\HH$, and Equation~\eqref{eq:finite-energy-futaki} gives
$\cM^{\rad}(\tau_\xi u)=m$ for every
$\xi\in N_\Q(\SSS)$.
Apply \eqref{eq:fano-pointwise-control} to
$(\tau_\xi u)_t$, divide by $t$, and take upper limits as
$t\to\infty$.
Since
$J^{\rad}(\tau_\xi u)\ge N$, taking the supremum over rational
$\xi$ gives
\[
 -\varepsilon N
 \le
 \RR_\beta^\SSS(u)
 \le
 C_{a,\lambda}(\beta)m
 -F_{a,\lambda}(\beta)N
 \le
 -F_{a,\lambda}(\beta)N.
\]
This is impossible if
$\varepsilon<F_{a,\lambda}(\beta)$.
Since $F_{a,\lambda}(\beta)\to F_{A,\lambda}(\beta)$ as
$a\uparrow A<\lambda$, the assumption
$\varepsilon<F_{A,\lambda}(\beta)$ gives the required
contradiction for $a$ sufficiently close to $A$.

We assume \eqref{eq:fano-tolerance-endpoint} and prove
existence.
Since the Futaki character vanishes on $\Lie\SSS$, Li's twisting
identity \cite[Lemma~2.24]{Li}, applied with reductive group
$\SSS$, gives, for every rational maximal-torus twist,
$\cM^{\NA}(\phi_\xi)=\cM^{\NA}(\phi_\Tcal)$.
Equation~\eqref{eq:main-infinite} gives
\[
 \cM^{\NA}(\phi_\Tcal)
 \ge
 m_{\beta_0}(\phi_\xi)
 +G_A(\beta_0)J^{\NA}(\phi_\xi).
\]
Since
$J^{\NA}(\phi_\xi)\ge N(\Tcal)$, taking the supremum over rational
$\xi$ and using \eqref{eq:fano-tolerance-endpoint} yields
\[
 \cM^{\NA}(\phi_\Tcal)
 \ge
 Q_{\beta_0}(\Tcal)
 +G_A(\beta_0)N(\Tcal)
 \ge0.
\]
Since $\mathrm{DF}\ge\cM^{\NA}$
\cite[Definition~7.13 and Proposition~7.15]{BHJ}, every normal
ample $\HH$-equivariant test configuration has nonnegative
Donaldson--Futaki invariant. After rescaling the polarization to
$-K_X$, Zhuang's equivariant criterion
\cite[Theorem~1.1\textup{(1)}]{ZhuangEq} gives K-semistability of
$(X,-K_X)$.

The Fujita--Odaka $\alpha$-invariant bound
\cite[Theorem~1.1]{JiangAlpha} and
$\alpha(X,-K_X)=A/\lambda$ therefore give
\[
 \frac{A}{\lambda}\ge\frac1{n+1}.
\]
If $A=\lambda/(n+1)$, Jiang's rigidity theorem
\cite[Theorem~1.2]{JiangAlpha} gives
$X\simeq\mathbf P^n$, and the scaled Fubini--Study metric gives a
cscK metric in $c_1(L)$.

Suppose next that
$\lambda/(n+1)<A<\lambda$.
Equations~\eqref{eq:fano-finite-tolerance} and~\eqref{eq:constants}
give, for every $\beta>0$,
\[
 F_{A,\lambda}(\beta)-G_A(\beta)
 =\frac{(n+1)A-\lambda}{\lambda-A}
   G_A(\max\{A,\beta\})>0.
\]
Choose $a<A$ sufficiently close to $A$ so that
$F_{a,\lambda}(\beta_0)>G_A(\beta_0)$.
Then \eqref{eq:fano-tolerance-endpoint} is the strict tolerance
condition with
$\varepsilon=G_A(\beta_0)$, contradicting nonexistence by the
first part of the proof.

Finally, suppose $A\ge\lambda$. For $0<a<\lambda\le A$,
the preceding estimates apply. Since
$F_{a,\lambda}(\beta_0)\to\infty$ as $a\uparrow\lambda$, we can
fix such an $a$ with $F_{a,\lambda}(\beta_0)>G_A(\beta_0)$.
The argument for \eqref{eq:fano-stronger-test}, with this $a$
and $\varepsilon=G_A(\beta_0)$, gives a contradiction.

For a maximal connected reductive subgroup, necessity follows from
Theorem~\ref{main:levi}, which gives the Futaki condition and
$Q_\beta(\Tcal)\ge0$ at one fixed finite parameter for every
normal ample $\HH$-equivariant test configuration.
Since $G_A(\beta)N(\Tcal)\ge0$, this implies
\eqref{eq:fano-tolerance-endpoint}.
\end{proof}

\begin{remark}[Failure of the endpoint for a log pair]
\label{rem:log-endpoint}
The endpoint conclusion of Theorem~\ref{thm:fano-tolerance-endpoint}
fails for the analogous quantized energy of a log pair. Consider
\[
 X=\mathbf P^1,\qquad
 \Delta=\tfrac12[\infty]+\tfrac14[0]+\tfrac14[1],\qquad
 L=-(K_X+\Delta)\sim_\Q\mathcal O(1).
\]
The identity component of $\Aut(X,\Delta)$ is trivial.
Let $\omega_{\rm FS}\in c_1(\mathcal O(1))$ be the Fubini--Study
form of mass one, and put $d\nu=\omega_{\rm FS}$. Write $d_p$ for
the coefficient of $[p]$ in $\Delta$, and let $s_p$ be a section
of $\mathcal O(1)$ vanishing at $p$, with its Fubini--Study norm.
The probability measure
\[
 d\mu=C\prod_{p\in\{\infty,0,1\}}|s_p|^{-2d_p}\,d\nu
\]
satisfies $\Ric(\mu)-[\Delta]=\omega_{\rm FS}$.
Its density is bounded below by a positive constant $c$ and
belongs to $L^p(\nu)$ for $1<p<2$. Since
$\alpha(\mathbf P^1,\mathcal O(1))=1$, H\"older's inequality gives
\begin{equation}\label{eq:log-alpha-integrability}
 \sup_{\substack{u\in\operatorname{PSH}(X,\omega_{\rm FS})\\
                 \sup_Xu=0}}
 \int_Xe^{-au}\,d\mu<\infty,\qquad 0<a<\tfrac12.
\end{equation}
Smooth approximations of a logarithmic pole at $\infty$ show that
this uniform bound fails for $a\ge\tfrac12$. Thus the
$\alpha$-invariant relative to $\mu$ is $A_\Delta=\tfrac12$.

Define $\Ent_\mu^\beta$ by the supremum over smooth competitors in
\eqref{eq:analytic-quantization}, with $\omega_0=\omega_{\rm FS}$
and reference measure $\mu$. For this log pair the Ricci energy is
$\cJ_{\omega_{\rm FS}}=J$, so put
$\cM_\Delta^\beta=\Ent_\mu^\beta-J$.
Both $\Ent_\mu^\beta$ and $J$ are invariant under additive
constants, so it suffices to assume $\sup_Xu=0$.
For $0<\beta<\tfrac12$, taking $w=0$ in
\eqref{eq:analytic-quantization} with reference measure $\mu$
and using \eqref{eq:log-alpha-integrability} gives
\begin{equation}\label{eq:log-endpoint-entropy}
 \beta J(u)-C_\beta
 \le\Ent_\mu^\beta(u)
 \le\Ent_\nu^\beta(u)-\log c
 \le\beta J(u)-\log c,
 \qquad u\in\mathcal E^1.
\end{equation}
Here the upper bounds follow from $\mu\ge c\nu$ and
\eqref{eq:entropy-pointwise-bound}. No regularity of an optimizer
for $\mu$ is needed.
Consequently, for every finite-energy geodesic ray $u$ from a
smooth potential, the radial limit exists and
\[
 (\cM_\Delta^\beta)^{\rad}(u)=(\beta-1)J^{\rad}(u).
\]
In particular, at $\beta_0=\tfrac14$,
\begin{equation}\label{eq:log-endpoint-equality}
 (\cM_\Delta^{1/4})^{\rad}(u)
 =-\tfrac34J^{\rad}(u)
 =-G_{1/2}(1/4)J^{\rad}(u),
\end{equation}
where $G$ is evaluated in dimension one. This holds for the ray
of every normal ample test configuration of the log pair.
Nevertheless, the pair has no conical K\"ahler--Einstein metric.
Indeed, its angles divided by $2\pi$ are
$\theta=(\tfrac12,\tfrac34,\tfrac34)\in(0,1)^3$, and the necessary strict
inequality
$2+\sum_j(\theta_j-1)<2\min_j\theta_j$
has both sides equal to $1$. This contradicts the necessary
condition of Luo--Tian~\cite{LuoTian} for spherical metrics with
three or more cone angles in $(0,2\pi)$.
\end{remark}

\section{Critical parameters and polarized examples}\label{sec:bounds}

We study the finite-parameter threshold and prove
Theorem~\ref{main:ruled}. We first compute the threshold for
finite-group-equivariant test configurations on $\mathbf P^1$ and
then construct a fixed ruled surface with varying polarizations.
For these polarizations, we determine the critical parameter for
smooth-divisor deformations and bound the full critical parameter
from below.

\subsection{Concavity and strict growth}
\label{sec:threshold-examples}

Let $(X,L)$ be a connected smooth polarized manifold of dimension
$n\ge1$, and put $A=\alpha(X,L)$. Unless a smaller testing family is
specified, $\Tcal$ ranges over all normal ample test configurations,
with arbitrary positive exponent and invariants normalized to $L$.
Define
\begin{equation}\label{eq:two-thresholds}
 \lambda_\beta^{\NA}(X,L)
 :=\inf_{J^{\NA}(\phi_\Tcal)>0}
       \frac{m_\beta(\phi_\Tcal)}{J^{\NA}(\phi_\Tcal)}.
\end{equation}
By \eqref{eq:slope-bridge}, the quotient in
\eqref{eq:two-thresholds} equals
$(\cM^\beta)^{\rad}(u^\Tcal)/J^{\rad}(u^\Tcal)$.
The parameter $\beta$ is real; intersection formulas are first used
at rational parameters and then extended by continuity.
The threshold in \eqref{eq:two-thresholds} differs from the optimal
coercivity constants for Donaldson's $J$-functional studied in
\cite{Sjo}. If $J^{\NA}(\phi)=0$, then
\eqref{eq:radial-entropy-bound}, \eqref{eq:ricci-slope-bound}, and
\eqref{eq:radial-mabuchi-splitting} give $m_\beta(\phi)=0$.
Thus $\lambda_\beta^{\NA}\ge0$ is equivalent to nonnegativity of
$m_\beta$ on the full testing family.

\begin{proposition}[Concavity and strict growth]
\label{prop:threshold-growth}
The function $\beta\mapsto\lambda_\beta^{\NA}$ is finite and concave
on $(0,\infty)$. For every $0<\beta<B$,
\begin{equation}\label{eq:threshold-growth-v2}
 \Gamma_A(\beta,B)
 \le\lambda_B^{\NA}-\lambda_\beta^{\NA}\le B-\beta.
\end{equation}
Consequently, it is strictly increasing and $1$-Lipschitz, and has
at most one positive zero.
\end{proposition}
\begin{proof}
For $J^{\NA}(\phi)>0$, set
$f_\beta(\phi)=m_\beta(\phi)/J^{\NA}(\phi)$.
Equations~\eqref{eq:radial-entropy-bound},
\eqref{eq:radial-mabuchi-splitting}, and
\eqref{eq:ricci-slope-bound} give
$-C\le f_\beta(\phi)\le\beta+C$, with $C$ independent of $\phi$.
Proposition~\ref{prop:dimension-obstruction} gives a configuration of
positive norm. Thus the infimum in \eqref{eq:two-thresholds} is finite.

For $a,b>0$ and $0\le\theta\le1$, parameter concavity in
Lemma~\ref{lem:parameter-concavity} gives
\[
 f_{\theta a+(1-\theta)b}(\phi)
 \ge\theta f_a(\phi)+(1-\theta)f_b(\phi)
 \ge\theta\lambda_a^{\NA}+(1-\theta)\lambda_b^{\NA}.
\]
Taking the infimum over $\phi$ proves concavity.
Equations~\eqref{eq:main-gap} and \eqref{eq:radial-secant-bound}
give
\[
 f_\beta(\phi)+\Gamma_A(\beta,B)
 \le f_B(\phi)\le f_\beta(\phi)+B-\beta.
\]
Taking infima proves \eqref{eq:threshold-growth-v2}.
Since $\Gamma_A(\beta,B)>0$, the function is strictly increasing;
the upper bound proves the $1$-Lipschitz property.
\end{proof}

For the reduced threshold, Corollary~\ref{cor:reduced-threshold}
gives a Lipschitz constant on each compact parameter interval.
The constant is $1$ in the unreduced estimate
\eqref{eq:threshold-growth-v2}.

\begin{corollary}[The exact initial threshold]
\label{cor:initial-threshold}
Set $\mathfrak R_L=\sup_{J^{\NA}(\phi_\Tcal)>0}
\mathcal R^{\NA}(\phi_\Tcal)/J^{\NA}(\phi_\Tcal)$.
Then $\mathfrak R_L$ is finite and
\begin{equation}\label{eq:initial-threshold}
 \lambda_\beta^{\NA}=\beta-\mathfrak R_L,
 \qquad 0<\beta\le A.
\end{equation}
For every $\beta>0$, $\lambda_\beta^{\NA}\le\beta-\mathfrak R_L$.
If $0<\mathfrak R_L\le A$, the unique positive zero is
$\beta_{\rm crit}=\mathfrak R_L$. If $\mathfrak R_L\le0$, then
$\lambda_\beta^{\NA}>0$ for every $\beta>0$.
\end{corollary}
\begin{proof}
Equation~\eqref{eq:ricci-slope-bound} gives
$\mathfrak R_L\in\R$.
For $0<\beta\le A$, \eqref{eq:alpha-cutoff} and
\eqref{eq:radial-mabuchi-splitting} give
\[
 \lambda_\beta^{\NA}
 =\inf_{J^{\NA}(\phi)>0}
   \left(\beta-\frac{\mathcal R^{\NA}(\phi)}{J^{\NA}(\phi)}\right)
 =\beta-\mathfrak R_L.
\]
For arbitrary $\beta>0$, the upper bound in
\eqref{eq:radial-entropy-bound} gives the inequality
$\lambda_\beta^{\NA}\le\beta-\mathfrak R_L$.
If $0<\mathfrak R_L\le A$, then
$\lambda_{\mathfrak R_L}^{\NA}=0$; strict increase makes this the
unique zero $\beta_{\rm crit}$. If $\mathfrak R_L\le0$, \eqref{eq:initial-threshold} gives
$\lambda_\beta^{\NA}>0$ on $(0,A]$.
Equation~\eqref{eq:threshold-growth-v2} gives the same inequality
for $\beta>A$.
\end{proof}

\begin{proposition}[Change of polarization scale]
\label{prop:polarization-scale}
For every positive integer $k$ and every $\beta>0$,
\begin{equation}\label{eq:polarization-scale}
 \lambda_\beta^{\NA}(X,L^{\otimes k})
 =\frac1k\lambda_{k\beta}^{\NA}(X,L),\qquad
 \alpha(X,L^{\otimes k})=\frac1k\alpha(X,L).
\end{equation}
A positive critical parameter exists for one polarization if and
only if it exists for the other. In that case,
$\beta_{\rm crit}(L^{\otimes k})=\beta_{\rm crit}(L)/k$, and
$\beta_{\rm crit}/\alpha$ has the same value for both polarizations.
\end{proposition}
\begin{proof}
Apply the proof of Proposition~\ref{prop:kahler-scale} with
reference forms $\omega_0$ and $k\omega_0$. It gives
\[
 m_\beta^{kL}(k\phi)=m_{k\beta}^L(\phi),\qquad
 J_{kL}^{\NA}(k\phi)=kJ_L^{\NA}(\phi).
\]
For a configuration of $(X,L)$, taking the $k$th tensor power of
its polarization gives a configuration for $(X,L^{\otimes k})$ with metric $k\phi$.
Conversely, a configuration for $(X,L^{\otimes k})$ of exponent $r$
is a configuration for $(X,L)$ of exponent $kr$, with normalized
metric $\psi/k$ if its original metric is $\psi$.
Since all positive exponents are allowed, $\phi\mapsto k\phi$
is a bijection between the two families of normalized metrics.
Taking infima proves the identity for $\lambda_\beta^{\NA}$.
Proposition~\ref{prop:kahler-scale} gives the identity for $\alpha$.
Finally, $\lambda_\beta^{\NA}(X,L^{\otimes k})\ge0$ exactly when
$\lambda_{k\beta}^{\NA}(X,L)\ge0$, proving the claims about
$\beta_{\rm crit}$.
\end{proof}

\subsection{Finite-group-equivariant thresholds on the projective line}

The irreducible representation $Q_8\subset\mathrm{SU}(2)$ induces
an action on $\mathbf P^1$ with image
$Q_8/\{\pm1\}\simeq(\mathbf Z/2\mathbf Z)^2$.
The projective image is generated by $z\mapsto-z$ and $z\mapsto1/z$.
It has the two-point orbit $\{0,\infty\}$ and no common fixed point,
so its smallest orbit has cardinality two.

\begin{proposition}[The threshold on $\mathbf P^1$ with a finite group action]
\label{prop:fibre-threshold}
Let $\Gamma\subset\mathrm{PGL}_2$ be finite with no common fixed
point, and let $r\ge2$ be its smallest orbit cardinality.
Define $\lambda_{\beta,\Gamma}^{\rm fib}$ by the infimum in
\eqref{eq:two-thresholds}, restricted to normal ample
$\Gamma$-equivariant configurations of $(\mathbf P^1,\mathcal O(1))$
of positive norm and arbitrary linearizable positive exponent.
Then
\begin{equation}\label{eq:fibre-threshold}
 \lambda_{\beta,\Gamma}^{\rm fib}=
 \begin{cases}
  \beta-2,&0<\beta\le r,\\
  2r-2-r^2/\beta,&\beta\ge r.
 \end{cases}
\end{equation}
For the projective image of $Q_8$, $r=2$ and the unique zero is
$\beta=2$.
\end{proposition}
\begin{proof}
The inverse image of $\Gamma$ in $\mathrm{SL}_2$ is finite and
linearizes $\mathcal O(1)$; equivalently, one may use sufficiently
divisible $\Gamma$-linearizable exponents. Choose a
$\Gamma$-invariant Fubini--Study reference. Since
$\Ric(\omega_0)=2\omega_0$ and $J_0=J$ in dimension one,
\eqref{eq:Jtheta} gives $\mathcal R^{\NA}=2J^{\NA}$.

Fix $\phi$ with $J^{\NA}(\phi)>0$ and $\sup\phi=0$.
Choose a degree-$k$ weight presentation, and let $B_k$ be the base
divisor of its top-weight subspace. The divisor $B_k$ is nonzero
and $\Gamma$-invariant, and $\deg B_k\le k$.
Its multiplicities are constant on each orbit, so
$\operatorname{mult}_pB_k\le k/r$ for every $p$.
By \eqref{eq:valuative-top-ideal},
$\tau_\phi=k\operatorname{lct}(\mathbf P^1;B_k)\ge r$.
Thus \eqref{eq:valuative-initial-segment} gives
$e_\beta(\phi)=\beta J^{\NA}(\phi)$ for $0<\beta\le r$.
For $\beta>r$, \eqref{eq:valuative-finite-increment} and the
monotonicity of $\Gamma_a(r,\beta)$ in $a$ give
\[
 e_\beta(\phi)
 \ge\bigl(r+\Gamma_r(r,\beta)\bigr)J^{\NA}(\phi)
 =\left(2r-\frac{r^2}{\beta}\right)J^{\NA}(\phi).
\]
Subtracting $2J^{\NA}(\phi)$ proves the lower bound in
\eqref{eq:fibre-threshold}.

For the reverse inequality, let $D$ be a reduced orbit divisor of
degree $r$. Its deformation to the normal cone, with rational
$0<c<1/r$, is $\Gamma$-equivariant and ample. The calculation in
\eqref{eq:dnc-energy}, with $n=1$ and $k=r$, gives
$\phi_c(v)=-c\min\{v(D),1\}$ and
$J^{\NA}(\phi_c)=rc^2/2$.
Since $D$ is reduced, $A_X(v)\ge v(D)$, with equality at the
valuation of each point of $D$. Hence $\tau_{\phi_c}=1/c$ and
$e_\beta(\phi_c)=\beta J^{\NA}(\phi_c)$ when $\beta c\le1$.
When $\beta c>1$, one has $\beta>r\ge2$, so
\eqref{eq:dnc-envelope} and \eqref{eq:entropy-slope} give
$e_\beta(\phi_c)=\beta\bigl(J^{\NA}(\phi_c)
-J^{\NA}(\phi_{c-1/\beta})\bigr)=rc-r/(2\beta)$.
Therefore
\[
 \frac{m_\beta(\phi_c)}{J^{\NA}(\phi_c)}
 =\begin{cases}
   \beta-2,&\beta c\le1,\\
   2/c-1/(\beta c^2)-2,&\beta c>1.
  \end{cases}
\]
For $\beta\le r$, every such $c$ attains the lower bound.
For $\beta>r$, letting rational $c\uparrow1/r$ gives
$2r-2-r^2/\beta$.
\end{proof}

\subsection{A fixed cscK surface with unbounded critical parameters}

We use the stable-bundle correspondence of
Narasimhan--Seshadri~\cite{NS} and Nagata's bound~\cite{Nagata}.
For related constructions on projective bundles, see
\cite{acgt08,acgt11}.

\begin{theorem}[The ruled-surface threshold estimate]
\label{thm:q8-effective}
There is a smooth ruled surface $X$ with $\Aut_0(X)=\{1\}$ and
ample line bundles $L_d$, for integers $d\ge4$, such that each
$c_1(L_d)$ contains a cscK metric and
$\lambda_\beta^{\NA}(X,L_d)$ has a unique zero
$\beta_{\rm crit}(d)$. Let $R_d$ be the positive root of
\begin{equation}\label{eq:q8-effective-v2}
 2(2d+1)R_d^2-3d(d-1)R_d-2d=0,\qquad
 \beta_{\rm crit}(d)\ge R_d>2.
\end{equation}
Then $R_d/d\to3/4$. In particular,
$\beta_{\rm crit}(40)\ge(130+22\sqrt{35})/9>28.9059$.
There is no single finite parameter at which every $(X,L_d)$ is
finite-parameter semistable.
\end{theorem}
\begin{proof}
Let $C$ be a genus-two curve. In the presentation
$\pi_1(C)=\langle a_1,b_1,a_2,b_2\mid[a_1,b_1][a_2,b_2]=1\rangle$,
send $a_1,a_2$ to $i\in Q_8$ and $b_1,b_2$ to $j\in Q_8$.
Since $[i,j]=-1$, this defines a surjection onto $Q_8$.
Let $p:C'\to C$ be the connected finite \'etale cover determined
by its kernel. The irreducible representation
$Q_8\subset\mathrm{SU}(2)$ defines a flat unitary bundle $E$ of
rank two with trivial determinant. With the quotient-line
convention, set
\begin{equation}\label{eq:q8-surface-polarization}
 X=\mathbf P(E)=(\mathbf P^1\times C')/Q_8,\qquad
 L_d=\mathcal O_X(1)\otimes\pi^*A_1^d,
 \qquad \deg A_1=1,
\end{equation}
where $\pi:X\to C$ is the projection.
The diagonal action is free, and the linear representation
makes $\mathcal O(1)$ descend to the quotient.
Choose a constant-curvature form $\omega_C\in c_1(A_1)$.
The invariant product form
$p_1^*\omega_{\rm FS}+d\,p_2^*p^*\omega_C$
descends to a cscK form in $c_1(L_d)$, so $L_d$ is ample for $d>0$.

The curve $C'$ has no holomorphic vector fields. Thus the vector
fields on $\mathbf P^1\times C'$ form $\mathfrak{sl}_2$.
Their $Q_8$-invariant subspace is zero by Schur's lemma, because
the representation on $\C^2$ is irreducible.
Consequently, $\Aut_0(X)=\{1\}$.

By \cite{NS}, $E$ is stable of degree zero. Nagata's bound
\cite{Nagata} gives
$\deg E-2\max_{M\subset E}\deg M\le g(C)=2$,
where $M$ ranges over saturated line subbundles.
Stability and $\deg E=0$ make the left-hand side a positive even
integer, hence it equals two. There is therefore a quotient
$E\twoheadrightarrow B$ of degree one with kernel $B^{-1}$.
The quotient defines a smooth section $D$ with
$\mathcal O_X(D)=\mathcal O_X(1)\otimes\pi^*B$.
Write $\xi=c_1(\mathcal O_X(1))$ and let $f$ be the numerical
fiber class. Then
\begin{equation}\label{eq:q8-intersections}
 \xi^2=f^2=0,\quad \xi\cdot f=1,\quad
 L_d\equiv\xi+df,\quad D\equiv\xi+f,\quad
 K_X\equiv-2\xi+2f.
\end{equation}
In particular, $V=2d$, $L_d\cdot D=d+1$, $D^2=2$,
$K_X\cdot D=0$, and $\bar S=-2K_X\cdot L_d/V=2-2/d$.

Blow up $D\times\{0\}$ in $X\times\mathbf P^1$, with map
$\varpi:Y\to X\times\mathbf P^1$ and exceptional divisor $F$.
Put $\mathcal L_c=\varpi^*p_1^*L_d-cF$ for rational $0<c<1$.
The class $L_d-cD\equiv(1-c)\xi+(d-c)f$ is ample: its pullback
to the finite product cover is a sum of positive classes from
the two factors. The deformation-to-the-normal-cone criterion
\cite{rossthomas} gives a normal ample test configuration.
Its normalized metric is $\phi_c(v)=-c\min\{v(D),1\}$.
Both central components are reduced, and the model discrepancy
is $a_Y(v)=\min\{v(D),1\}$.

Set $J(c)=J^{\NA}(\phi_c)$ and $H(c)=H^{\NA}(\phi_c)$.
The blowup identities
$(\varpi^*p_1^*L_d)^2F=0$,
$(\varpi^*p_1^*L_d)F^2=-(d+1)$, and $F^3=-2$ give
$\mathcal L_c^3=-3(d+1)c^2+2c^3$ and
$\mathcal L_c^2F=2(d+1)c-2c^2$.
Moreover,
$(\varpi^*p_1^*K_X)\mathcal L_c^2=-c^2K_X\cdot D=0$.
The energy, entropy, and Mabuchi intersection formulas therefore give
\begin{equation}\label{eq:q8-slopes-v2}
 \begin{aligned}
 J(c)&=\frac{d+1}{2d}c^2-\frac{c^3}{3d},&
 H(c)&=J'(c),\\
 \cM^{\NA}(\phi_c)&=H(c)-\bar S J(c),&
 \bar S&=2-\frac2d.
 \end{aligned}
\end{equation}
For $\beta c\ge1$, the envelope in \eqref{eq:envelope} is
$\phi_{c-1/\beta}$. Thus \eqref{eq:entropy-slope} gives
\begin{equation}\label{eq:q8-quantized-slopes}
 e_\beta(\phi_c)=\beta\bigl(J(c)-J(c-1/\beta)\bigr),\qquad
 m_\beta(\phi_c)=e_\beta(\phi_c)-\bar S J(c).
\end{equation}
For fixed $\beta>1$, letting rational $c\uparrow1$ yields
\begin{equation}\label{eq:q8-threshold-upper}
 \lambda_\beta^{\NA}(X,L_d)\le h_d(\beta):=
 \frac{\displaystyle
   \frac{2d+1}{3d^2}-\frac{d-1}{2d\beta}
                      -\frac1{3d\beta^2}}
      {\displaystyle\frac12+\frac1{6d}}.
\end{equation}
The function $h_d$ is strictly increasing. Its value at $\beta=2$
is negative for $d\ge4$, since its numerator there is
$(-3d^2+10d+4)/(12d^2)$.
Multiplying its numerator by $6d^2\beta^2$ gives
$2(2d+1)\beta^2-3d(d-1)\beta-2d$, whose positive zero is $R_d>2$.
For every $1<\beta<R_d$, a rational $c<1$ sufficiently close to
one has $m_\beta(\phi_c)<0$.

Since $\Aut_0(X)=\{1\}$, the cscK metric is unique
by \cite{bb}. The criterion of
Darvas--Zhang~\cite[Theorem~1.1]{DZ} then gives
$\lambda_B^{\NA}(X,L_d)>0$ at some finite $B$.
Proposition~\ref{prop:threshold-growth}, together with
$\lambda_2^{\NA}<0$, gives a unique zero and
$\beta_{\rm crit}(d)\ge R_d$.
For $d=40$, the root equation reduces to
$81R^2-2340R-40=0$, with positive root
$(130+22\sqrt{35})/9$.
Solving \eqref{eq:q8-effective-v2} for its positive root gives
$R_d/d\to3/4$.
\end{proof}

All configurations used as $c\uparrow1$ have $c<1$ and are ample.
We do not require ampleness at $c=1$ or attainment of the full threshold.

\begin{proposition}[The $\alpha$-invariant of the ruled examples]
\label{prop:q8-alpha}
For the surface and polarizations in
\eqref{eq:q8-surface-polarization}, with any degree-one line bundle
$A_1$ on $C$, one has $\alpha(X,L_d)=1/d$ for every integer $d\ge1$.
Consequently, $e_\beta(\phi)=\beta J^{\NA}(\phi)$ for every normal
ample configuration metric of $(X,L_d)$ and $0<\beta\le1/d$.
\end{proposition}
\begin{proof}
The class $\xi$ is nef, since its pullback to the finite product
cover is the hyperplane class from $\mathbf P^1$.
Let $\Delta\sim_\Q L_d$ be effective. For a fiber $F$, write
$\Delta=aF+\Delta'$ with $F\not\subset\operatorname{Supp}\Delta'$.
Nefness gives $a\le\xi\cdot\Delta=d$, and
$\deg(\Delta'|_F)=\Delta'\cdot F=1$.
Every coefficient of $\Delta'|_F$ is at most one, so
$(F,\Delta'|_F)$ is log canonical. Inversion of adjunction
\cite[Theorem~1.1]{FujinoAdjunction} implies that $(X,F+\Delta')$ is log
canonical near $F$. Since $d\ge1$ and
$\Delta/d\le F+\Delta'$, the pair $(X,\Delta/d)$ is log canonical
near $F$. The fibers cover $X$, so Demailly's characterization
\cite[Theorem~A.3]{DemaillyAlpha} gives $\alpha(X,L_d)\ge1/d$.

Since the monodromy group is finite, there is a nonzero invariant
in a symmetric power of the representation space: for $v\ne0$, the product
$\prod_{g\in Q_8}g\cdot v$ is nonzero in the symmetric algebra
and is invariant under $Q_8$. Let $m_0$ be its degree.
For each positive multiple $m$ of $m_0$, the bundle
$\operatorname{Sym}^mE$ has a nonzero parallel section. Under the
quotient-line convention, it defines
$s_m\in H^0(X,\mathcal O_X(m))$.
Fix $p\in C$. For sufficiently large such $m$, Riemann--Roch gives
$0\ne\sigma_m\in H^0(C,A_1^{md}(-(md-2)p))$, since this line
bundle has degree two on a genus-two curve. Hence
$\Delta_m=m^{-1}\operatorname{div}(s_m\pi^*\sigma_m)\sim_\Q L_d$
has coefficient at least $d-2/m$ along $\pi^{-1}(p)$.
Therefore $\alpha(X,L_d)\le(d-2/m)^{-1}$, and letting $m\to\infty$
gives $\alpha(X,L_d)\le1/d$.
The entropy identity, including $\beta=1/d$, follows from
\eqref{eq:alpha-cutoff}.
\end{proof}

\begin{proof}[Proof of Theorem~\ref{main:ruled}]
\label{proof:main-ruled}
Take $X$ and $L_d$ from Theorem~\ref{thm:q8-effective}.
Proposition~\ref{prop:q8-alpha} gives $\alpha(X,L_d)=1/d$, and
\eqref{eq:q8-effective-v2} gives $\beta_{\rm crit}(X,L_d)\ge R_d$
with $R_d/d\to3/4$. Thus
\[
 \liminf_{d\to\infty}\frac{\beta_{\rm crit}(X,L_d)}d\ge\frac34,
 \qquad
 \frac{\beta_{\rm crit}(X,L_d)}{\alpha(X,L_d)}\ge dR_d\to\infty.
\]
Also $\beta_{\rm crit}(X,L_d)\ge R_d\to\infty$.
The line bundles are integral and have fiber degree one.
Equation~\eqref{eq:q8-intersections} shows that their numerical
classes are pairwise nonproportional.
\end{proof}

The pure fiber family consists of products of $Q_8$-equivariant
configurations of $\mathbf P^1$ with $C'$, followed by the finite
quotient. For a potential pulled back from the fiber, the product
Monge--Amp\`ere measure and uniqueness in
\eqref{eq:quantized-optimizer} show that the quantized solution
is the pullback of the fiber solution. The normalized energies
satisfy $J_X=J_{\mathbf P^1}$ and
$\Ent_X^\beta=\Ent_{\mathbf P^1}^\beta$; in
\eqref{eq:Jtheta}, the base Ricci contribution cancels the base
scalar-curvature contribution, giving
$\cJ_{\Ric,X}=2J_{\mathbf P^1}$.
These identities also hold after the finite quotient.
Proposition~\ref{prop:fibre-threshold} therefore gives the exact
critical parameter $2$ for this family. Equation~\eqref{eq:q8-effective-v2} gives
$\beta_{\rm crit}(X,L_d)\ge R_d>2$ for $d\ge4$.
We do not determine the full critical parameter.

\begin{proposition}[Optimality among smooth divisor deformations]
\label{prop:smooth-divisor-optimality}
Fix $(X,L_d)$ from Theorem~\ref{thm:q8-effective}, with $d\ge4$.
Among deformations to the normal cone of nonzero smooth reduced
divisors, using all rational parameters in the open ample interval,
the smallest parameter at which every quantized Mabuchi invariant
is nonnegative is $R_d$.
At $\beta=R_d$, the normalized invariants for the degree-one section
tend to zero as the deformation parameter $c$ increases to one.
Smooth multisections and vertical divisors give no larger critical
parameter. The infimum of $m_\beta/J^{\NA}$ over this family is
zero at $R_d$ and positive for $\beta>R_d$.
\end{proposition}
\begin{proof}
Write $D\equiv a\xi+bf$. Nefness of $f$ and $\xi$ gives
$a,b\in\mathbf Z_{\ge0}$. If $a=1$, the horizontal component is
a section, and stability of $E$ gives $b\ge1$.
The open ample interval is given by $c>0$, $ac<1$, and $bc<d$,
as follows by pulling $L_d-cD$ back to the product cover.
Write $J_D(c)=J^{\NA}(\phi_{D,c})$ and
$e_{\beta,D}(c)=e_\beta(\phi_{D,c})$ for
$\phi_{D,c}(v)=-c\min\{v(D),1\}$.
The intersections $L_d\cdot D=ad+b$, $D^2=2ab$, and
$K_X\cdot D=2a-2b$ give
\begin{equation}\label{eq:smooth-divisor-slopes}
 \begin{aligned}
 J_D(c)&=\frac{ad+b}{2d}c^2-\frac{ab}{3d}c^3,\\
 m_\beta(\phi_{D,c})
 &=e_{\beta,D}(c)-\left(2-\frac2d\right)J_D(c)
                         +\frac{b-a}{d}c^2.
 \end{aligned}
\end{equation}
If $a=0$, then $m_\beta(\phi_{D,c})=e_{\beta,D}(c)+(2/d)J_D(c)
\ge0$ for every $\beta>0$.

Assume $a\ge1$, so $c<1$.
Smoothness and reducedness of $D$ give $\tau_{\phi_{D,c}}=1/c$:
$A_X(v)\ge v(D)$ for all divisorial $v$, with equality at every
component of $D$. Thus \eqref{eq:valuative-initial-segment}
gives the first branch of
\begin{equation}\label{eq:smooth-divisor-entropy}
 e_{\beta,D}(c)=
 \begin{cases}
  \beta J_D(c),&\beta c\le1,\\
  \beta\bigl(J_D(c)-J_D(c-1/\beta)\bigr),&\beta c\ge1.
 \end{cases}
\end{equation}
The second branch follows from \eqref{eq:entropy-slope} and the
envelope $\phi_{D,c-1/\beta}$; here $\beta c\ge1$ implies
$\beta>1$ because $c<1$.
If $a\ge2$, then $c<1/2$, and \eqref{eq:smooth-divisor-slopes}
and \eqref{eq:smooth-divisor-entropy} give
\[
 m_2(\phi_{D,c})
 =\frac{bc^2}{d}\left(1+\frac1d-\frac{2ac}{3d}\right)\ge0.
\]
Equation~\eqref{eq:radial-secant-bound} gives nonnegativity for
all $\beta\ge2$.

It remains to treat $a=1$ and $b\ge1$.
Write $J_b(c)=c^2/2+bP(c)$, where
$P(c)=c^2/(2d)-c^3/(3d)$.
Treat $b$ as a real variable in the explicit formulas for $J_b(c)$
and $e_{\beta,D}(c)$.
Since $P'(c)=c(1-c)/d\ge0$ on $[0,1]$,
\eqref{eq:smooth-divisor-entropy} gives
$\partial_b e_{\beta,D}(c)\ge0$ in both branches.
Furthermore,
\[
 \partial_b\bigl(m_\beta(\phi_{D,c})-e_{\beta,D}(c)\bigr)
 =\frac{c^2}{d}-\left(2-\frac2d\right)P(c)
 =\frac{c^2\bigl(3+2(d-1)c\bigr)}{3d^2}>0.
\]
For every fixed admissible $c$, the quantized Mabuchi invariant is
therefore minimized at $b=1$.

For $a=b=1$, one has
$J_1(c)=c^2(d+1-2c/3)/(2d)$ and
\[
 (\log J_1)''(c)
 =-\frac2{c^2}-\frac{(2/3)^2}{(d+1-2c/3)^2}<0.
\]
For $1/\beta<c<1$, this implies
$\frac{d}{dc}\log\bigl(J_1(c-1/\beta)/J_1(c)\bigr)>0$.
Equation~\eqref{eq:smooth-divisor-entropy} shows that
$e_{\beta,D}(c)/J_1(c)$ is nonincreasing in $c$ and is constant
when $\beta c\le1$.
Consequently, for $\beta>1$, the infimum of
$m_\beta(\phi_{D,c})/J_1(c)$ for the degree-one section is
$h_d(\beta)$ from \eqref{eq:q8-threshold-upper}.

At $\beta\ge R_d>2$, the degree-one section has nonnegative
invariant, and the comparisons for $a=0$, $a\ge2$, and $a=1$
prove nonnegativity for every smooth-divisor deformation.
For $1<\beta<R_d$, $h_d(\beta)<0$, so a rational $c<1$
sufficiently close to one gives a negative invariant.
For $0<\beta\le1$, the degree-one section has
$m_\beta(\phi_{D,c})/J_1(c)=\beta-(2-2/d)<0$.
At $R_d$, all normalized invariants are nonnegative, and those for
the degree-one section tend to zero as $c\uparrow1$.
Thus the infimum over the whole family is zero. For $\beta>R_d$, \eqref{eq:main-gap} gives
$m_\beta(\phi_{D,c})\ge\Gamma_A(R_d,\beta)J_D(c)$ for the whole
family, so its normalized infimum is positive.
\end{proof}

\subsection{Negative lower bounds below the semistability threshold}

\begin{corollary}[A negative lower bound below $\beta_{\rm crit}$]
\label{cor:strict-negative-lower-bound}
Fix $(X,L_d)$ from Theorem~\ref{thm:q8-effective}, and put
$A=\alpha(X,L_d)$. For every
\begin{equation}\label{eq:strict-tolerance-interval}
 \max\{0,\beta_{\rm crit}-G_A(\beta_{\rm crit})\}
 <\beta<\beta_{\rm crit},
\end{equation}
the number $\varepsilon=\beta_{\rm crit}-\beta$ satisfies
$0<-\lambda_\beta^{\NA}\le\varepsilon<G_A(\beta)$, and
\[
 (\cM^\beta)^{\rad}(u^\Tcal)
 \ge-\varepsilon J^{\rad}(u^\Tcal)
 \qquad\text{for every normal ample }\Tcal.
\]
At least one configuration has negative quantized Mabuchi slope.
Thus the common negative lower bound does not imply nonnegativity
on all configurations at the same parameter $\beta$.
\end{corollary}
\begin{proof}
Since $\lambda_{\beta_{\rm crit}}^{\NA}=0$,
\eqref{eq:threshold-growth-v2} and
\eqref{eq:strict-tolerance-interval} give
\[
 0<-\lambda_\beta^{\NA}
 \le\beta_{\rm crit}-\beta=\varepsilon
 <G_A(\beta_{\rm crit})\le G_A(\beta).
\]
For $J^{\NA}(\phi_\Tcal)>0$, \eqref{eq:two-thresholds} gives
$m_\beta(\phi_\Tcal)/J^{\NA}(\phi_\Tcal)
\ge\lambda_\beta^{\NA}\ge-\varepsilon$.
Zero-norm configurations have $m_\beta=0$.
Equation~\eqref{eq:slope-bridge} gives the radial inequality.
The strict inequality $\lambda_\beta^{\NA}<0$ implies that some
configuration has $m_\beta<0$, whether or not the infimum is attained.
\end{proof}

\subsection{Comparison with the ordinary threshold}
\label{subsec:ordinary-separation}

Define
\begin{equation}\label{eq:ordinary-threshold-definition}
 \lambda_{\mathrm{ord}}^{\NA}
 :=\inf_{J^{\NA}(\phi_\Tcal)>0}
       \frac{\cM^{\NA}(\phi_\Tcal)}{J^{\NA}(\phi_\Tcal)}.
\end{equation}
The testing family is the full family in \eqref{eq:two-thresholds}.
The bound \eqref{eq:ricci-slope-bound} and $H^{\NA}\ge0$ give a
finite lower bound. Any fixed configuration of positive norm gives a
finite upper bound. Thus $\lambda_{\mathrm{ord}}^{\NA}\in\R$.

If $(X,L)$ is ordinarily K-semistable but not uniformly K-stable,
then $\lambda_{\mathrm{ord}}^{\NA}=0$; the equivalence of the DF
and non-Archimedean Mabuchi formulations follows from normalized
base change \cite[Propositions~7.16 and~8.2]{BHJ}.
Thus \eqref{eq:main-infinite} gives
$\lambda_\beta^{\NA}\le-G_A(\beta)<0$ for every finite $\beta>0$.
Ordinary K-semistability cannot replace finite-parameter
semistability in Corollary~\ref{cor:polarized-correspondence}.

Papazachariou~\cite[Theorem~4.8]{Pap} proves convergence with an
$O(\beta^{-1/2})$ error under ordinary K-semistability, using an
$L^1$-normalized threshold. The following estimate holds for either
sign of $\lambda_{\mathrm{ord}}^{\NA}$.

\begin{proposition}[Convergence of stability thresholds]
\label{prop:all-sign-threshold}
Let $\delta=\lambda_{\mathrm{ord}}^{\NA}$. There are constants
$C_R,C>0$, depending only on $(X,L)$, such that
$S:=\delta+C_R>0$ and, for every $\beta>1$,
\begin{equation}\label{eq:all-sign-threshold}
 G_A(\beta)\le\delta-\lambda_\beta^{\NA}
 \le\frac{\delta_+S}{\beta}+C\sqrt{\frac S\beta},
 \qquad \delta_+=\max\{\delta,0\}.
\end{equation}
In particular, $\lambda_\beta^{\NA}$ increases strictly to
$\lambda_{\mathrm{ord}}^{\NA}$ without a semistability assumption.
\end{proposition}
\begin{proof}
Choose $C_R>1+\sup_{J^{\NA}(\phi)>0}
|\mathcal R^{\NA}(\phi)|/J^{\NA}(\phi)$ over the full testing family.
Equation~\eqref{eq:ricci-slope-bound} makes this choice possible,
and $H^{\NA}\ge0$ gives $S>0$.
Fix rational $\beta>1$ and $\phi$ with $J:=J^{\NA}(\phi)>0$ and
$\sup\phi=0$. Let $\psi=\psi_\beta$ be the envelope in
\eqref{eq:envelope}. Then $\phi\le\psi\le0$, both suprema are
zero, and \eqref{eq:entropy-slope} gives
\begin{equation}\label{eq:ordinary-envelope-energy-gap}
 D:=\ENA(\psi)-\ENA(\phi)=\frac{e_\beta(\phi)}\beta,
 \qquad J^{\NA}(\psi)=J-D.
\end{equation}
By \cite[Corollary~9.6]{DZ}, $H^{\NA}(\psi)\le e_\beta(\phi)<\infty$.
Lemma~\ref{lem:shared-entropy-approximation}, applied with the
trivial group, extends
$\cM^{\NA}\ge\delta J^{\NA}$ from ample configuration metrics
to $\psi$. Under this approximation, both $\cM^{\NA}$ and $J^{\NA}$ converge.
No sign assumption on $\delta$ is needed.
Papazachariou's estimate \cite[Theorem~3.10]{Pap} therefore gives
\begin{equation}\label{eq:ordinary-quantized-comparison}
 m_\beta(\phi)
 \ge\cM^{\NA}(\psi)-C\sqrt{DJ}
 \ge\delta J-\delta D-C\sqrt{DJ}.
\end{equation}
If $m_\beta(\phi)\ge\delta J$, then
$m_\beta(\phi)\ge(\delta-\delta_+S/\beta-C\sqrt{S/\beta})J$.
Otherwise, \eqref{eq:ordinary-envelope-energy-gap} and
\eqref{eq:radial-mabuchi-splitting} give
$\beta D=m_\beta(\phi)+\mathcal R^{\NA}(\phi)<SJ$.
Since $D\ge0$, \eqref{eq:ordinary-quantized-comparison} implies
\[
 m_\beta(\phi)
 \ge\left(\delta-\frac{\delta_+S}{\beta}
                   -C\sqrt{\frac S\beta}\right)J.
\]
Taking infima proves the upper bound in
\eqref{eq:all-sign-threshold} for rational $\beta>1$;
\eqref{eq:threshold-growth-v2} extends it to real $\beta>1$.
For the lower bound, \eqref{eq:main-infinite} gives
$\cM^{\NA}(\phi)/J^{\NA}(\phi)
\ge m_\beta(\phi)/J^{\NA}(\phi)+G_A(\beta)$.
Taking infima gives $\delta\ge\lambda_\beta^{\NA}+G_A(\beta)$.
\end{proof}

In \eqref{eq:all-sign-threshold}, the lower bound has order
$\beta^{-n}$, and the upper bound has order $\beta^{-1/2}$.

\begin{corollary}[Critical parameter and the ordinary stability constant]
\label{cor:critical-ordinary-margin}
Suppose that $\delta=\lambda_{\mathrm{ord}}^{\NA}>0$.
Then $\beta_{\rm crit}(X,L)<\infty$. If $0<\delta\le A/n$, then
\begin{equation}\label{eq:critical-ordinary-margin}
 \frac{\beta_{\rm crit}(X,L)}{A}
 \ge \left(\frac{A}{n\delta}\right)^{1/n}.
\end{equation}
\end{corollary}
\begin{proof}
Proposition~\ref{prop:all-sign-threshold} gives
$\lambda_\beta^{\NA}\to\delta>0$, so the set defining
$\beta_{\rm crit}(X,L)$ is nonempty.
Taking infima in \eqref{eq:main-infinite} gives
$\delta\ge\lambda_\beta^{\NA}+G_A(\beta)$ for every $\beta>0$.
If $\lambda_\beta^{\NA}\ge0$ and $\delta\le A/n$, then
$G_A(\beta)\le\delta\le G_A(A)$, hence $\beta\ge A$.
Thus $A^{n+1}/(n\beta^n)\le\delta$.
Taking the infimum over such $\beta$ proves
\eqref{eq:critical-ordinary-margin}.
\end{proof}

Thus, for a family of fixed dimension, $\delta/A\to0^+$ implies
$\beta_{\rm crit}/A\to\infty$.

\section{Ordinary uniform K-stability}
\label{ordinary:section}

\subsection{Polarized manifolds}
\label{ordinary:polarized}

For polarized classes, Trusiani~\cite[Theorem~A and
Definition~2.3]{Tru} proves the following theorem, written here
with our normalization.

\begin{theorem}[Reduced uniform K-stability and cscK metrics]
\label{ordinary:reduced}
Let $(X,L)$ be a connected smooth polarized manifold.
The class $c_1(L)$ contains a cscK metric if and only if the
connected polarized automorphism group
$\G=\Aut_0(X,L)$ is reductive and there is $\delta>0$ such that
\begin{equation}\label{ordinary:reduced-DF}
 \mathrm{DF}(\Tcal)\ge\delta J_{\G}^{\NA}(\phi_\Tcal)
\end{equation}
for every normal ample $\G$-equivariant test configuration of
arbitrary positive exponent, with invariants normalized to $L$.
Here $\mathbb T=Z(\G)^\circ$ and
$J_{\G}^{\NA}(\phi)=
\inf_{\xi\in N_\R(\mathbb T)}J^{\NA}(\phi_\xi)$,
as in \eqref{eq:central-norm} with $\HH=\G$.
\end{theorem}

By Lemma~\ref{lem:weyl}, the infimum defining the reduced norm can
also be taken over rational cocharacters. The normalized $\Q$-line bundles in
\cite[Section~2.5]{Tru} correspond to integral configurations of
arbitrary positive exponent by clearing denominators.
For a configuration of exponent $r$, let
$\overline{\mathcal L}$ be its compactified integral polarization,
so $\mathcal L|_{\mathcal X_1}\simeq L^{\otimes r}$.
Applying \cite[Definition~7.13 and Proposition~7.15]{BHJ} to
$r^{-1}\overline{\mathcal L}$ gives
\begin{equation}\label{ordinary:correction}
 \cM^{\NA}(\Tcal)=\mathrm{DF}(\Tcal)
 +\frac{(\mathcal X_{0,\mathrm{red}}-\mathcal X_0)
                  \cdot\overline{\mathcal L}^{\,n}}{r^nV}
 \le\mathrm{DF}(\Tcal),
 \qquad V=L^n.
\end{equation}

Let $\Tcal^{(d)}$ be the normalization of a degree-$d$ base
change, with $d$ sufficiently divisible.
By \cite[Proposition~7.16]{BHJ},
\begin{equation}\label{ordinary:base-change}
 \mathrm{DF}(\Tcal^{(d)})
 =\cM^{\NA}(\Tcal^{(d)})=d\cM^{\NA}(\Tcal).
\end{equation}
For $\xi\in N_\Q(\mathbb T)$, normalized base change and the
homogeneity of $J^{\NA}$ \cite[Proposition~7.8]{BHJ} give
\[
 J^{\NA}\bigl((\phi_{\Tcal^{(d)}})_\xi\bigr)
 =dJ^{\NA}\bigl((\phi_\Tcal)_{\xi/d}\bigr).
\]
Since $\xi\mapsto\xi/d$ is a bijection of $N_\Q(\mathbb T)$,
taking infima yields
$J_\G^{\NA}(\phi_{\Tcal^{(d)}})=dJ_\G^{\NA}(\phi_\Tcal)$.
Normalized base change preserves ampleness and
$\G$-equivariance. Applying \eqref{ordinary:reduced-DF} to
$\Tcal^{(d)}$ and using \eqref{ordinary:base-change} gives
$\cM^{\NA}(\Tcal)\ge\delta J_\G^{\NA}(\phi_\Tcal)$.
Conversely, this Mabuchi bound implies
\eqref{ordinary:reduced-DF} by \eqref{ordinary:correction}.
Without twists, the same base-change calculation gives the
unreduced equivalence in \cite[Proposition~8.2]{BHJ}.

\begin{corollary}[A sequence with vanishing Donaldson--Futaki ratios]
\label{ordinary:zero-margin-sequence}
Let $(X,L)$ be a connected smooth polarized manifold, set
$A=\alpha(X,L)$, and let $S\subset\G$ be an algebraic torus.
Let $\Tcal_j$ be normal ample $S$-equivariant test configurations,
and write $\Tcal_{j,\xi}=(\Tcal_j)_\xi$ for their rational product
twists. Assume that the Futaki character vanishes on $\Lie S$,
that $N_j:=\inf_{\xi\in N_\Q(S)}
J^{\NA}(\phi_{\Tcal_{j,\xi}})>0$, and that
$\mathrm{DF}(\Tcal_j)/N_j\to0$.
Then, for every fixed $\beta>0$,
\begin{equation}\label{ordinary:zero-margin-bound}
 \limsup_{j\to\infty}\frac1{N_j}
    \sup_{\xi\in N_\Q(S)}
       (\cM^\beta)^{\rad}(u^{\Tcal_{j,\xi}})
 \le-G_A(\beta)<0.
\end{equation}
\end{corollary}

\begin{proof}
For each $j$ and $\xi\in N_\Q(S)$, Li's twisting identity
\cite[Lemma~2.24]{Li}, applied with reductive group $S$, gives
$\cM^{\NA}(\phi_{\Tcal_{j,\xi}})=
\cM^{\NA}(\phi_{\Tcal_j})$.
Equation~\eqref{eq:main-infinite} and
\eqref{ordinary:correction} therefore give
\[
 \begin{aligned}
 m_\beta(\phi_{\Tcal_{j,\xi}})
 &\le\cM^{\NA}(\phi_{\Tcal_j})
       -G_A(\beta)J^{\NA}(\phi_{\Tcal_{j,\xi}})\\
 &\le\mathrm{DF}(\Tcal_j)-G_A(\beta)N_j.
 \end{aligned}
\]
For each $j$, the supremum over $\xi$ is finite: the upper bound
is independent of $\xi$, and the value at $\xi=0$ is finite.
Taking the supremum, dividing by $N_j>0$, and using
\eqref{eq:slope-bridge} gives
\[
 \frac1{N_j}\sup_{\xi\in N_\Q(S)}
 (\cM^\beta)^{\rad}(u^{\Tcal_{j,\xi}})
 \le\frac{\mathrm{DF}(\Tcal_j)}{N_j}-G_A(\beta).
\]
Taking the upper limit as $j\to\infty$ proves
\eqref{ordinary:zero-margin-bound}.
\end{proof}

\subsection{K\"ahler classes}
\label{ordinary:kahler}

The analytic parameter bounds also give the finite-parameter
correspondence for arbitrary compact K\"ahler classes. This part uses
the transcendental results of Trusiani~\cite{TruK}.

\begin{corollary}[The compact K\"ahler correspondence]
\label{main:correspondence}\label{cor:iff}\label{cor:kahler-correspondence}
Let $\Omega$ be a K\"ahler class on a connected compact K\"ahler
manifold $X$. The following are equivalent:
\begin{enumerate}[label=\textup{(\roman*)}]
\item $\Omega$ contains a unique cscK metric;
\item for some finite $\beta>0$,
$s_\beta(u^\Tcal)\ge0$ for every normal relatively
K\"ahler test configuration $\Tcal$;
\item for some finite $B>0$ and $\gamma>0$,
$s_B(u^\Tcal)\ge\gamma J^{\rad}(u^\Tcal)$ for every
such $\Tcal$;
\item $(X,\Omega)$ is uniformly K-stable: for some $\delta>0$,
$\mathrm{DF}(\Tcal)\ge\delta J^{\NA}(\Tcal)$ for every smooth
dominating K\"ahler test configuration $\Tcal$.
\end{enumerate}
\end{corollary}

Theorem~\ref{main:sharp-gap} shows that condition \textup{(ii)}
implies condition \textup{(iii)}. Uniqueness means uniqueness of the K\"ahler form in its class.
The test configurations in conditions \textup{(ii)} and \textup{(iii)}
are those of Darvas--Zhang~\cite[Definition~5.4]{DZ}; those in
condition \textup{(iv)} are those of
Trusiani~\cite[Section~2.9]{TruK}.
Darvas--Zhang~\cite[Theorem~5.5]{DZ} prove that conditions
\textup{(i)} and \textup{(iii)} are equivalent, and
Trusiani~\cite[Theorem~A]{TruK} proves that conditions
\textup{(i)} and \textup{(iv)} are equivalent.

\begin{proof}[Proof of Corollary~\ref{main:correspondence}]
By Darvas--Zhang~\cite[Theorem~5.5]{DZ} and
Trusiani~\cite[Theorem~A]{TruK}, and since \textup{(iii)} implies
\textup{(ii)}, it suffices to prove that \textup{(ii)} implies
\textup{(iii)}.

Assume \textup{(ii)} holds for some $\beta>0$. For every
$B>\beta$, Theorem~\ref{thm:gap} gives
\[
 s_B(u^\Tcal)
 \ge
 s_\beta(u^\Tcal)
 +\Gamma_A(\beta,B)J^{\rad}(u^\Tcal)
 \ge
 \Gamma_A(\beta,B)J^{\rad}(u^\Tcal).
\]
Since
\(
 \Gamma_A(\beta,B)>0
\)
is independent of $\Tcal$, condition \textup{(iii)} holds with
\(
 \gamma=\Gamma_A(\beta,B).
\)
\end{proof}

Let $\Omega$ be a K\"ahler class on a connected compact K\"ahler
manifold $X$ of dimension $n\ge1$, and set
$A=\alpha(X,\Omega)$ and $V=\int_X\Omega^n$. We use the smooth dominating K\"ahler and big
test configurations of Trusiani~\cite[Sections~2.9 and~2.10]{TruK}.
Both $\mathrm{DF}$ and $J^{\NA}$ are divided by $V$, so their
quotient is unchanged. Define
\begin{equation}\label{ordinary:thresholds}
 \delta_\Omega^{\rm Kah}
 =\inf_{\substack{\Tcal\text{ K\"ahler}\\J^{\NA}(\Tcal)>0}}
       \frac{\mathrm{DF}(\Tcal)}{J^{\NA}(\Tcal)},\qquad
 \delta_\Omega^{\rm big}
 =\inf_{\substack{\Tcal\text{ big}\\J^{\NA}(\Tcal)>0}}
       \frac{\mathrm{DF}(\Tcal)}{J^{\NA}(\Tcal)}.
\end{equation}
For a configuration with structure map $q:\mathcal X\to\mathbf P^1$,
adding $q^*\eta$ to its polarization class, with $\eta$ sufficiently
positive on $\mathbf P^1$, makes a relatively K\"ahler class
K\"ahler without changing $\mathrm{DF}$ or $J^{\NA}$.
We also set
\begin{equation}\label{ordinary:finite-threshold}
 \lambda_\beta^{\rm Kah}
 =\inf_{\substack{\Tcal\text{ normal relatively K\"ahler}\\
                  J^{\rad}(u^\Tcal)>0}}
      \frac{s_\beta(u^\Tcal)}{J^{\rad}(u^\Tcal)},
\end{equation}
using the family in Corollary~\ref{cor:kahler-correspondence}.
This family contains
all the smooth dominating K\"ahler configurations in
\eqref{ordinary:thresholds}.

\begin{corollary}[Finite parameters and ordinary K-stability]
\label{ordinary:unreduced}
For every compact K\"ahler class,
\begin{equation}\label{ordinary:optimal-margins}
 \delta_\Omega:=\delta_\Omega^{\rm Kah}
       =\delta_\Omega^{\rm big},\qquad
 \delta_\Omega\ge\lambda_\beta^{\rm Kah}+G_A(\beta)
       \quad(\beta>0).
\end{equation}
For every $c\in\R$, the bound $\mathrm{DF}\ge cJ^{\NA}$ on all
smooth dominating K\"ahler configurations is equivalent to the
same bound on all smooth dominating big configurations.
In particular, if \eqref{eq:kahler-semistability} holds on the
compact K\"ahler testing family for some $\beta>0$ and
$0\le\varepsilon<G_A(\beta)$, then
\begin{equation}\label{ordinary:DF-bound}
 \mathrm{DF}(\Tcal)\ge
 \bigl(G_A(\beta)-\varepsilon\bigr)J^{\NA}(\Tcal)
\end{equation}
for every configuration in either ordinary testing family.
Moreover, $\delta_\Omega>0$ is equivalent to each condition of
Corollary~\ref{cor:kahler-correspondence}. If $\delta_\Omega=0$,
then $\lambda_\beta^{\rm Kah}\le-G_A(\beta)<0$ for every $\beta>0$.
\end{corollary}

\begin{proof}
Let $\Tcal$ be a smooth dominating K\"ahler configuration with
$J^{\NA}(\Tcal)>0$. The cohomological slope formulas and
central-fiber correction
\cite[Definition~3.15, Corollary~4.14, Theorem~5.6 and
Remark~5.7]{SDtrans}, together with \eqref{eq:main-infinite}, give
\begin{equation}\label{ordinary:ray-bound}
 \begin{aligned}
 \frac{\mathrm{DF}(\Tcal)}{J^{\NA}(\Tcal)}
 &\ge\frac{\cM^{\rad}(u^\Tcal)}{J^{\rad}(u^\Tcal)}\\
 &\ge\frac{s_\beta(u^\Tcal)}{J^{\rad}(u^\Tcal)}
       +G_A(\beta)
 \ge\lambda_\beta^{\rm Kah}+G_A(\beta).
 \end{aligned}
\end{equation}
The formulas in \cite{SDtrans} use $s=-\log|\tau|$; we use
$t=-\log|\tau|^2=2s$. Both analytic slopes are multiplied by the
same factor, so their quotient is unchanged.
The $J$-slope formula extends to the associated weak ray because
a smooth compatible ray differs from it by a uniformly bounded
potential. Indeed, $\|v-w\|_\infty\le C$ implies
$|J(v)-J(w)|\le2C$ by monotonicity and translation covariance
of $\AM$. The Mabuchi slope formula applies to the weak ray by
\cite[Remark~5.7]{SDtrans}.
Taking the infimum in \eqref{ordinary:ray-bound} gives
$\delta_\Omega^{\rm Kah}\ge\lambda_\beta^{\rm Kah}+G_A(\beta)$.
A zero-norm K\"ahler configuration is trivial and has
$\mathrm{DF}=0$ \cite[Section~2.9]{TruK}.

For every smooth dominating big configuration $\Tcal$,
Trusiani's special K\"ahler Fujita approximation
\cite[Section~5, equations~(5.2) and~(5.3)]{TruK} gives smooth
dominating K\"ahler configurations $\Tcal_j$ satisfying
\begin{equation}\label{ordinary:fujita-convergence}
 \mathrm{DF}(\Tcal_j)\longrightarrow\mathrm{DF}(\Tcal),\qquad
 J^{\NA}(\Tcal_j)\longrightarrow J^{\NA}(\Tcal).
\end{equation}
If $\mathrm{DF}\ge cJ^{\NA}$ on the K\"ahler family for a fixed
$c\in\R$, then
\[
 \mathrm{DF}(\Tcal)-cJ^{\NA}(\Tcal)
 =\lim_{j\to\infty}
   \bigl(\mathrm{DF}(\Tcal_j)-cJ^{\NA}(\Tcal_j)\bigr)\ge0.
\]
This limit also applies when $J^{\NA}(\Tcal)=0$ or $c<0$.
The reverse implication follows because every K\"ahler
configuration is big.
If $J^{\NA}(\Tcal)>0$, the approximating norms in
\eqref{ordinary:fujita-convergence} are eventually positive,
and the quotients converge. Hence
$\delta_\Omega^{\rm Kah}\le
\mathrm{DF}(\Tcal)/J^{\NA}(\Tcal)$.
Taking the infimum over positive-norm big configurations gives
$\delta_\Omega^{\rm Kah}\le\delta_\Omega^{\rm big}$; inclusion of
the K\"ahler family gives the reverse inequality.

Under \eqref{eq:kahler-semistability}, one has
$\lambda_\beta^{\rm Kah}\ge-\varepsilon$.
Thus \eqref{ordinary:ray-bound} gives \eqref{ordinary:DF-bound}
on positive-norm K\"ahler configurations. The zero-norm case and
\eqref{ordinary:fujita-convergence} extend it to both families.
Finally, $\delta_\Omega>0$ is ordinary uniform K-stability, which
is equivalent to the existence of a unique cscK metric by
\cite[Theorem~A]{TruK}. Corollary~\ref{cor:kahler-correspondence}
then gives the finite-parameter equivalences. The case
$\delta_\Omega=0$ follows from \eqref{ordinary:optimal-margins}.
\end{proof}

\section*{Acknowledgements}

The authors thank Yaxiong Liu for helpful suggestions and discussions.

The first named author is partially supported by the National Key
Research and Development Program of China (NKPs)
[Grant Number~2024YFA1013201], the Natural Science Foundation of
Chongqing (NSFCQ) [Grant Number~CSTB2024NSCQ-LZX0040], and the Special
Project of Chongqing Municipal Science and Technology Bureau
[Grant Number~2025CCZ015].

The second named author is supported by the National Natural Science
Foundation of China [grant number~12601080]; the Scientific Research
Foundation of Chongqing University of Technology
[grant number~2026ZDZ012]; and the Youth Project of the Science and
Technology Research Program of Chongqing Education Commission of China
[grant number~KJQN202601108].

\par\medskip
\noindent\textbf{Declaration of AI use}
\par\nobreak\smallskip
\noindent
The authors used OpenAI ChatGPT and Codex for exploratory mathematical checks, manuscript organization, and language editing. They reviewed and edited the resulting content and take full responsibility for this article.

\end{document}